\documentclass[12pt,a4paper]{amsart}
\usepackage{microtype}
\usepackage[margin=2.54cm]{geometry}
\usepackage[utf8]{inputenc}
\usepackage[T1]{fontenc}
\usepackage{lmodern}
\usepackage{amsfonts}
\usepackage{epsfig}
\usepackage{graphicx}
\usepackage{breqn}
\usepackage{amsmath}
\usepackage{amssymb}
\usepackage{latexsym}
\usepackage{rotating}
\usepackage{pstricks, pst-node, pst-text, pst-3d}
\usepackage{amsbsy}
\usepackage{subfig}
\usepackage{enumerate}
\usepackage{amsthm}
\usepackage{mathrsfs}
\usepackage{hyperref}
\usepackage{textcomp}
\usepackage{color}
\usepackage{verbatim}
\usepackage{subfig}
\usepackage{afterpage}
\usepackage{pdflscape}
\usepackage{rotating}
\usepackage[super]{nth}
\usepackage{nag}
\usepackage{enumitem}

\usepackage{todonotes}

\allowdisplaybreaks[1]

\newtheorem{theorem}{Theorem}
\newtheorem{proposition}{Proposition}
\newtheorem{corollary}{Corollary}
\newtheorem{lemma}{Lemma}

\theoremstyle{definition}
\newtheorem{definition}{Definition}

\theoremstyle{remark}

\newtheorem*{remark*}{Remark}

\DeclareMathOperator{\re}{Re}
\DeclareMathOperator{\im}{Im}
\DeclareMathOperator{\Log}{Log}

\DeclareMathOperator{\res}{Res}
\newcommand{\tinrange}{{t\in[0,T_{\max{}})}}
\newcommand{\autfam}{\{M_t\}_\tinrange}

\newcommand{\utilde}{{\tilde{u}_{\lambda(t)}}}
\newcommand{\newutilde}{{-T^{\tilde{\sigma}} \circ M_t \circ h^{-1}_{u_t}(1)}}

\newcommand{\utildenot}{{\tilde{u}_{\tilde{t}}}}
\newcommand{\mathhyphen}{{\textrm{-}}}

\newcommand{\atilde}{{\tilde{a}_\utilde}}
\newcommand{\newatilde}{{\tilde{a}_\newutilde}}
\newcommand{\thetatilde}{{\tilde{\theta}_\utilde}}
\newcommand{\newthetatilde}{{\tilde{\theta}_\newutilde}}
\newcommand{\etilde}{{e^{i\thetatilde}}}
\newcommand{\newetilde}{{e^{i\newthetatilde}}}

\newcommand{\hutat}{h_{u_t}(A_t)}

\DeclareMathOperator{\arctanh}{arctanh}

\DeclareMathOperator{\sech}{sech}

\DeclareMathOperator{\id}{id}

\begin{document}
\title{General Slit Löwner Chains}
\author[Georgy Ivanov]{Georgy Ivanov$^{\dag}$}
\author{Alexey Tochin}
\author{Alexander Vasil'ev}

\address{Department of Mathematics,
University of Bergen, P.O.~Box~7803, Bergen N-5020, Norway}
\email{georgy.ivanov@math.uib.no}\email{alexey.tochin@math.uib.no}\email{alexander.vasiliev@math.uib.no}

\thanks{All authors have been  supported by EU FP7 IRSES program STREVCOMS, grant no. PIRSES-GA-2013-612669, by the grants of the Norwegian Research Council \# 204726/V30, \# 213440/BG\@. The author$^{\dag}$ has also been supported by Meltzerfondet.}


\subjclass[2010]{30C35, 34M99, 60D05, 60J67}

\keywords{L\"owner equation, stochastic flows, general L\"owner theory, SLE, slit evolution, SDE}

\date{04/04 2014 }

\begin{abstract}
We use general Löwner theory to define general slit Löwner chains in the unit disk, which in the stochastic case lead to slit   holomorphic stochastic flows. Radial, chordal and dipolar $SLE$ are classical examples of such flows. Our approach, however, allows to  construct new processes of $SLE$ type that possess conformal invariance and the domain Markov property. The local behavior of these processes is similar to that of classical $SLEs$.
\end{abstract}

\maketitle

\section{Introduction}\label{introd}
The classical Löwner theory was introduced in 1923 by Karl Löwner (Charles Loewner)~\cite{Loewner}, and was later developed by Kufarev~\cite{Kufarev, Kufarev47Tomsk} \/ and Pommerenke~\cite{Pommerenke1, Pommerenke2}. The Löwner differential equation became one of the most powerful tools for solving extremal problems in the theory of univalent functions, culminating in the proof of the Bieberbach conjecture by  de Branges in 1984~\cite{Branges}.

In the modern period, Löwner theory has again attracted a lot of interest due to discovery of Stochastic (Schramm)-L\"owner Evolution ($SLE$), a stochastic process that has made it possible to describe analytically the scaling limits of several two-dimensional lattice models in statistical physics, see \cite{SLE1, Schramm}. Several other connections with mathematical physics have also been 
discovered, in particular, relations between $SLE$ martingales and singular representations of the Virasoro algebra in \cite{BB02, Friedrich2003a, MK}, a Hamiltonian formulation of the Löwner evolution and relations to KP integrable hierarchies in \cite{MarkinaVasiliev10, MarkinaVasiliev12,Takebe06}.

In line with these important achievements, the classical (deterministic) L\"owner  theory itself has undergone remarkable development, so that its various versions have been realized as special cases of the general L\"owner theory, see~\cite{BCDV14}.

$SLE$ theory focuses on describing probability measures on families of curves  which possess the property of {\it conformal invariance} \/and the {\it domain Markov property}. So far, the following types of $SLE$ have been studied: the {\it chordal $SLE$}~\cite{SLE1, RohdeSchrammBasic, Schramm}, the {\it radial $SLE$}~\cite{LawlerBook, RohdeSchrammBasic}, the {\it dipolar $SLE$}~\cite{BBernardHoudayer}, and $SLE(\kappa,\rho)$~\cite{Dub05, Lawler03, Werner05}. The random curves are described by the sets of initial conditions  for which the solution to the corresponding $SLE$ differential equation blows up in finite time (\emph{$SLE$ hulls}). Due to the aforementioned special properties of the measures it is possible to reformulate these differential equations as stochastic differential equations (diffusion equations with holomorphic coefficients).

In this paper we address the following questions. What are other possible diffusion equations with holomorphic coefficients that generate random families of curves? How similar are the properties of these curves to the properties of $SLE$ curves?

The paper is organized as follows. 

In Section~\ref{sec:Preliminaries} we briefly review the definitions of the classical $SLE$ processes, formulate their conformal invariance and domain Markov property, cite main results of the general Löwner theory and explain how it incorporates chordal, radial and dipolar theories as special cases. 

In Section~\ref{sec:Witt} we show how the Virasoro generators $\ell_n,$ $n=-2,\ldots,1$, can be used for representing complete vector fields and vector fields generating slit semiflows (\emph{slit holomorphic vector fields}). Complete vector fields admit the representation
\[
 \sigma(z) = \sigma_{-1} \ell_{-1}(z) + \sigma_{0} \ell_{0}(z) + \sigma_1 \ell_1(z), \quad \sigma_{-1},\sigma_{0}, \sigma_{1} \in \mathbb{R},
\]
and slit holomorphic vector fields have the following convenient representation
\begin{equation}
\label{eq:breprintro}
 b(z) = b_{-2} \ell_{-2}(z) + b_{-1} \ell_{-1}(z) + b_{0} \ell_0(z) + b_{1} \ell_1(z),
\end{equation}
where $b_{-2}>0,$ $b_{-1}, \, b_{0}, \, b_1 \in \mathbb{R}$.

In Section~\ref{sec:Kunita} we extend Kunita's results on flows of stochastic differential equations (see, e.g., \cite{Kunita1984}) to the case of semicomplete fields, giving sufficient conditions for a diffusion equation to generate a flow of holomorphic endomorphisms of the unit disk. 

In Section~\ref{sec:GSLE} we combine the results of Sections~\ref{sec:Witt} and~\ref{sec:Kunita} and define general slit Löwner chains, which  possess the property of conformal invariance and the domain Markov property. These chains are described by the \emph{general slit Löwner equation}
\[
 \begin{cases}
  \frac{\partial}{\partial t} g_t(z) =  - V(t,g_t(z)), \quad t\geq 0,\\
  g_0(z) = z, \quad z\in \mathbb{D},
 \end{cases}
\]
where
$ V(t,z) :=  \frac{1}{h'_{u_t}(z)} \,b(h_{u_t}(z)),
$
and $b(z)$ is a vector field of form \eqref{eq:breprintro}. The family $\{h_{t}\}_{t\in \mathbb{R}}$ is a flow (one-parameter group) of automorphisms of the unit disk generated by a complete field $\sigma(z)$. The driving function $u_t$ is an analog of the driving function from the classical Löwner theory. We always normalize $b$ and $\sigma$ in such a way that $b_{-2} = 2$ and $\sigma_{-1} = 1$.

If we put $u_t = \sqrt{\kappa} B_t$ ($\kappa\geq 0,$ and $B_t$ is a standard Brownian motion), then the process $G_t(z) = h_{u_t} \circ g_t (z)$ solves the Stratonovich SDE
\begin{equation}
 \begin{cases}
  dG_t(z) = -b(G_t(z)) \, dt + \sqrt{\kappa} \,\sigma(G_t(z)) \circ dB_t,\\
G_0(z) = z,
 \end{cases} z\in \mathbb{D}.
\end{equation}

 The radial, chordal and dipolar $SLEs$ correspond to particular choices of $b(z)$ and $\sigma(z)$. Our approach  allows to  construct new processes that possess conformal invariance and the domain Markov property, and we call the stochastic flow $G_t(z)$ a \emph{slit holomorphic stochastic flow} or $(b,\sigma)$-$SLE_\kappa$.  

Another framework providing a unified treatment of the radial, chordal and dipolar $SLEs$ is the $SLE(\kappa, \rho)$ theory, where the usual chordal Löwner equation is used, but the driving function is a complicated stochastic process. In our approach we always use a multiple of the standard Brownian motion as the driving function, but modify the vector fields $b$ and $\sigma$, so that the evolution is described by a single diffusion equation.

We propose a classification for the family of such processes, identifying processes that can be obtained one from another by means of simple transformations.

In Section~\ref{sec:relations} we show how the evolution of hulls generated by a general slit Löwner equation can be described in terms of the radial Löwner equation, and also how two general slit Löwner equations are related to each other. We use this technique in Section~\ref{sec:geometry} to prove that, similarly to the classical cases, the hulls of these processes are generated by a curve almost surely, and the results about local properties of the generated curves can be transferred from the classical $SLE$ theory to this general setting.

The Appendices contain some necessary background theory, such as theory of holomorphic semiflows (one-parameter semigroups), necessary elements of
differential geometry and the theory of stochastic flows in the unit disk.

\medskip
\noindent
{\bf Acknowledgments.} The authors would like to thank Nam-Gyu Kang and Pavel Gumenyuk for their comments on the original draft of the paper.

\section{Preliminaries}
\label{sec:Preliminaries}

\subsection{Classical L\"owner equations}
We call the chordal, radial and dipolar Löwner equations, as well as the related domain evolutions, classical, in contrast to other evolutions that can be described by the general L\"owner theory. A common feature of the classical equations is that all the maps constituting the evolution share at least one interior or boundary common fixed point.

\subsubsection*{Chordal L\"owner equation}

Let $u_t:[0,+\infty)\to \mathbb{R}$ be a continuous function of $t$ (the \emph{driving function}), and let $\{g_t\}_{t\geq 0}$ be the solution to the \emph{chordal Löwner equation driven by $u_t$}:
\begin{equation}
\label{eq:chordLoewner}
 \begin{cases}
  \frac{\partial }{\partial t} g_t(z) = \frac{2}{g_t(z) - u_t},\\
  g_0(z)= z,
 \end{cases} z\in \mathbb{H}.
\end{equation}

The family $\{g_t\}_{t\geq 0}$ is called the \emph{chordal Löwner chain} driven by $u_t$.

Denote by $H_t$ the set of all $z\in \mathbb{H},$ for which the solution $g_t(z)$ to (\ref{eq:chordLoewner}) is defined at time $t$. Then for each $t\geq 0,$ $H_t$ is a simply connected domain, called the \emph{evolution domain of \eqref{eq:chordLoewner} at time $t$}. The function $g_t(z)$ maps $H_t$ conformally onto $\mathbb{\mathbb{H}}$  and has the following behavior at infinity
\[
g_{t}(z) = z + \frac{2t}{z} + O(1/|z|^2), \quad z \to \infty,
\]
so that, in particular,  $\infty$ is the common fixed boundary point of the maps $\{g_t(z)\}_{t\geq 0}$

Let $\hat{\mathbb{H}}$ denote the closure of $\mathbb{H}$.  We say that the family of evolution domains $\{H_t\}_{t\geq 0}$ is \emph{generated by a curve} $\gamma:[0,+\infty)\to \hat{\mathbb{H}}$, if $H_t$ is the unbounded component of $\mathbb{H}\setminus \gamma[0,t]$. For an arbitrary continuous driving function $u_t,$ it is not true in general that the evolution domains $\{H_t\}_{t\geq 0}$ of the corresponding chordal Löwner equation  are generated by a curve. 

By putting $u_t=\sqrt{\kappa} B_t$ ($\kappa \geq 0,$ and $B_t$ is a standard Brownian motion), we obtain random families of conformal maps $\{g_t(z)\}_{t\geq 0}$ known as the chordal $SLE_\kappa$.  It is known that in this case the corresponding evolution domains are almost surely generated by a curve. The corresponding random family of curves is also called $SLE_\kappa,$ and is the main object of interest in $SLE$ theory.

If $\{g_t\}_{t\geq 0}$ is a chordal $SLE_\kappa$ (driven by $\sqrt{\kappa}B_t$), then for a $r>0$ the maps $\tilde{g}_t:= r^{-1} g_{r^2 t}(r z)$ are driven by $\sqrt{\kappa} \,r^{-1} B_{r^2 t} = \sqrt{\kappa} \tilde{B}_t$. Since $\tilde{B}_t$ is also a standard Brownian motion, $\tilde{g}_t$ has the distribution of $SLE_\kappa$. This property of chordal $SLE$ is called \emph{chordal $SLE$ scaling}.

\subsubsection*{Radial L\"owner equation} 
Again, let $u_t:[0,+\infty)\to \mathbb{R}$ be a continuous function (\emph{the driving function}), and  $\{g_t(z)\}_{t\geq 0}$ be the solution to the \emph{radial Löwner equation} driven by $u_t:$
\begin{equation}
\label{eq:radLoewner}
 \begin{cases}
  \frac{\partial}{\partial t} g_t(z) = g_t(z) \, \frac{e^{i u_t}+ g_t(z)}{e^{iu_t}-g_t(z)},\\
  g_0(z) = z,
 \end{cases}\\z\in \mathbb{D},
\end{equation}
The family $\{g_t(z)\}_{t\geq 0}$ is called the \emph{radial Löwner chain} driven by $u_t$.

Denote by $D_t$ the set of all $z\in \mathbb{D},$ for which the solution $g_t(z)$ to (\ref{eq:radLoewner}) is defined at  time $t$. Then $D_t$ is a simply connected domain, called the evolution domain of equation \eqref{eq:radLoewner} at time $t$. The function $g_t$ maps $D_t$ conformally onto $\mathbb{D},$ and, moreover, $g_t(0) = 0,$ $g'_t(0)= e^t$.

We say that the family of evolution domains $\{D_t\}_{t\geq 0}$ of a radial Löwner equation  is \emph{generated by the curve} $\gamma:[0,+\infty)\to \mathbb{D},$  if $D_t$ is the connected component of $\mathbb{D} \setminus \gamma[0,t]$  containing $0$. For an arbitrary driving function $u_t$, it is not true that the corresponding family of evolution domains is generated by a curve.

If we put $u_t = \sqrt{\kappa} B_t$, $\kappa \geq 0$  then the random family of domains $\{D_t(\omega)\}_{t\geq 0}$ is almost surely generated by a random curve $\gamma$. These random families of curves $\gamma$ (as well as the corresponding random families of domains $D_t$, and the corresponding random families of conformal maps $g_t(z)$) are referred to as \emph{the radial Schramm-Löwner evolution} $SLE_\kappa$.

\subsubsection*{Dipolar L\"owner equation}
Dipolar Löwner equation is usually formulated in the infinite strip $\mathbb{S}:= \{z:0<\im z< \pi\}:$

\begin{equation}
\label{eq:dipolar}
 \begin{cases}
  \frac{\partial}{\partial t} g_t(z) = \frac{1}{\tanh[(g_t(z) - u_t)/2]},\\
   g_0(z) = z,
 \end{cases} z\in \mathbb{S},
\end{equation}
where the driving function $u_t$ is a continuous real-valued function, as in the previous two cases. 
We denote by $S_t$ the set of all $z\in \mathbb{S}$  for which the solution to (\ref{eq:dipolar}) exists (the evolution domain of dipolar Löwner equation \eqref{eq:dipolar} at time $t$). The function $g_t(z)$ maps $S_t$ conformally onto $\mathbb{S},$ and 
fixes two points, $+\infty$ and $-\infty$, at the boundary of $\mathbb{S}$.

If we set $u_t = \sqrt{\kappa}B_t,$ then the family of evolution domains $\{S_t\}_{t\geq 0}$ is almost surely generated by a random curve $\gamma(t)$  (that is, $S_t$ is the unbounded component of  $\mathbb{S}\setminus \gamma[0,t]$). The random curve $\gamma(t)$, as well as the associated maps, are called the dipolar $SLE_\kappa$.

\subsection{Conformal invariance and domain Markov property of classical \texorpdfstring{$SLEs$}{SLEs}}
\label{subsec:confinvdmp}
The two properties which make $SLE$ processes so important for applications in statistical physics are \emph{the conformal invariance} and \emph{the domain Markov property}. The formulations of these properties differ slightly for radial, chordal and dipolar $SLE$. We formulate them in detail only for the radial case, and then, briefly mention how they can be reformulated for the other two cases.

\subsubsection{Boundary behavior of conformal isomorphisms}  Let us first recall some basic facts about boundary behavior of conformal maps.  By the Riemann mapping theorem, for any simply connected (s.c.) hyperbolic domain $D \subset \hat{\mathbb{C}}$ there exists a conformal isomorphism $f: \mathbb{D} \to D$.  If the boundary $\partial D$ of $D$ is a Jordan curve, then by Carathéodory's theorem it is possible to extend the map $f$ to a homeomorphism $f:\hat{D} \to \hat{G}$. If the boundary $\partial D$ is only locally connected, it is still possible to construct a continuous (but not injective) extension $f: \hat{\mathbb{D}} \to \hat{D}$. If the boundary $\partial D$ is not locally connected, then there exists no continuous extension  of $f$ to $\hat{\mathbb{D}}$ at all.

Instead of the boundary points of $D$ one can consider the set of \emph{prime ends} of $D$, denoted by $P(D)$. Then the union $D\cup P(D)$ can be endowed with a natural topology, so that every conformal isomorphism $f:\mathbb{D} \to D$ extends uniquely to a homeomorphism $f:\hat{\mathbb{D}} \to D\cup P(D)$ (see \cite[Chapter 17]{MilnorDynamics} for an introduction to the theory of prime ends).

We can summarize this in the following version of a Riemann mapping theorem.

\begin{theorem}[The Riemann mapping theorem]
Let $D \subset \hat{\mathbb{C}}$ be a simply connected hyperbolic domain, let $P(D)$ denote the set of its prime ends, $a\in D,$ and let the prime ends $b_1,$ $b_2,$ $b_3 \in P(D)$ be distinct and ordered anticlockwise. Then there exists a conformal isomorphism
\[
 f: \mathbb{D} \to D,
\]
which can be extended to a map
\[
 f : \hat{\mathbb{D}} \to D\cup P(D),
\]
which is a homeomorphism with respect to the natural topology of $D\cup P(D)$. Moreover, the map $f$ can be specified uniquely by imposing any of the following three normalization conditions
\begin{enumerate}
 \item  $f(0) = a$ and $f'(0) >0,$
 \item $f(0)= a$ and $f(1) = b_1,$
\item $f(-i) = b_1,$ $f(1) = b_2,$  and $f(i) = b_3$.
\end{enumerate}

\end{theorem}

We say that a curve $\gamma:(0,s] \to D$ lands at a prime end  $p\in P(D)$ if $\lim_{t\to 0} \gamma(t) = p$ in the topology of $D\cup P(D)$.

\subsubsection{Measures on spaces on curves}

The set of non-parameterized curves in $\mathbb{C}$ can be endowed with a metric, and consequently, with a metric topology and with the Borel $\sigma$-algebra (see, e.g., \cite[Chapter 5]{LawlerBook}). \emph{Non-self-traversing curves}, by definition, are limits of sequences of simple curves in this topology.

Let $D$ be a simply connected hyperbolic domain in $\hat{\mathbb{C}}$. Let $b\in D$ and $a \in P(D)$. Consider the family of curves
\[
 \Omega_{(D,a,b)} = \{ \gamma: \gamma \textrm{ is a non-self-traversing curve connecting $a$ and $b$}\},
\]
which can be given a natural topology  and the corresponding Borel $\sigma$-algebra, as before.

Now, let $\mathcal{T}$ be the set of all possible triples
\[
 \mathcal{T}:=\{(D,a,b) : D \textrm{ is a s.c.\ hyperbolic domain}, a \in P(D), b\in D\}.
\]

Note that according to the Riemann mapping theorem, for each pair $(D,a,b), (D',a',b') \in \mathcal{T},$ there exists a unique conformal isomorphism $\phi:D \to D'$ that can be extended to a homeomorphism $\phi:D\cup P(D) \to D'\cup P(D'),$ such that $\phi(a) = a',$ $\phi(b) = b'$.

Let $M$ be a family of measures indexed by $\mathcal{T}$, that is,
\[
 M  = \{\mu_{(D,a,b)} : (D,a,b) \in \mathcal{T}, \,\mu_{(D,a,b)} \textrm{ is a measure on } \Omega_{D,a,b}\}.
\]

\begin{definition}[Conformal invariance]
 We say that a family of measures $M$ indexed by $\mathcal{T}$ is conformally invariant if for each pair of triples $(D,a,b)$ and $(D', a', b')$ from $\mathcal{T}$ 
\[
 \mu_{(D',a',b')} =\phi_*\mu_{(D,a,b)},
\]
where $\phi:D\to D'$ is the unique conformal isomorphism sending $a\mapsto a',$ $b\mapsto b'$, and $\phi_*\mu$ denotes the pushforward of $\mu$ by $\phi$. 
\end{definition}
 
Note, that given a measure on $\Omega_{(D,a,b)}$ we can construct a conformally invariant family of measures indexed by $\mathcal{T}$ simply by pushing forward the original measure to all elements of $\mathcal{T}$.

\begin{definition}[Domain Markov property] We say that a family of measures indexed by $\mathcal{T}$ satisfies  the domain Markov property if for any Borel set $B$ the conditional law $\mu_{(D,a,b)} (B \,| \,\gamma[0,s] = \gamma_0)$ is such that
\[
 \mu_{(D,a,b)} (B \,| \,\gamma[0,s] = \gamma_0) = \mu_{(D_s, \gamma_0(s), b)}(B),
\]
 where $D_s$ denotes the connected component of $D \setminus \gamma_0$ containing $b$.
\end{definition}

Combined with the conformal invariance, this becomes
\[
 \mu_{(D,a,b)}(\, \cdot \,| \,\gamma[0,s]=\gamma_0) = \phi_* \mu_{(D_s,\gamma_0(s),b)} (\,\cdot \,),
\]
where $\phi:D \to D_s$ is the unique conformal isomorphism sending $a \to 
\gamma_0(s)$ and leaving $b$ unchanged.

\subsubsection{Radial $SLE$ measure}
Let $\kappa\geq 0$. The equation  
\[
 \begin{cases}
  \frac{\partial}{\partial t} g_t(z) = g_t(z) \, \frac{e^{i \sqrt{\kappa} B_t}+ g_t(z)}{e^{i\sqrt{\kappa}B_t}-g_t(z)},\\
  g_0(z) = z,
 \end{cases}\\z\in \mathbb{D},
\]
generates a random family of evolution domains, which is almost surely generated by simple curves connecting $1$ and $0$, thus inducing a probability measure on the family $\Omega_{(\mathbb{D}, 1, 0)}$. We denote this measure by $SLE^r_\kappa (\mathbb{D},1, 0)$. We can transfer this measure to all other triples $(D,a, b)\in \mathcal{T}$ in a conformally invariant way, i.e., by setting $SLE^r_{\kappa}(D, a, b) = \phi_* SLE^r_{\kappa}(D, a, b)$  where $\phi:\mathbb{D}\to D$ is the unique conformal isomorphism sending $1$ to $a$ and $0$ to $b$.

To see that this measure satisfies the domain Markov property, we use the fact that the family $G_t(z) := g_t(z)/e^{i\sqrt{\kappa} B_t}$ is the solution to the initial-value problem
\begin{equation}
\label{eq:radialSDE}
 dG_t(z) = G_t(z) \frac{1+ G_t(z)}{1-G_t(z)} \,dt - i \sqrt{\kappa} G_t  \circ dB_t, \quad  G_0(z) = z, \quad z \in \mathbb{D}.
\end{equation}
This is a time-homogeneous diffusion equation, therefore, $G_t(z)$ is a continuous time-homogeneous Markov process. In particular, the law of the process $G_{s+t} \circ G^{-1}_s$ conditioned on $G_s$ is the same as the law of $G_t$.

\subsubsection{Chordal and dipolar $SLE$ measures}
The definitions of conformal invariance and the domain Markov property for the other two cases are formulated analogously, with slight modifications which we describe below.

In the chordal  case, we define
\[
\mathcal{T}:= \{(D,a,b) | D \textrm{ is a s.c.\ hyperbolic domain, } a,b \in P(D)\},
\]
and
\[
 \Omega_{(D,a,b)} = \{ \gamma: \gamma \textrm{ is a non-self-traversing curve connecting $a$ and $b$}\}.
\]
Here, for a pair $(D,a,b), (D',a',b')\in \mathcal{T},$ the map $\phi:D\to D',$ sending $a\mapsto a'$ and $b\mapsto b'$ is not defined uniquely. However, we require that the chordal $SLE^c_{\kappa} (D,a,b)$ measure is invariant under conformal automorphisms of  $D$ leaving $a$ and $b$ invariant, which corresponds precisely to the scaling property of the chordal Löwner equation.

In the chordal case, the family $G_t := g_t(z) - \sqrt{\kappa} \, B_t$ also satisfies the diffusion equation:
\begin{equation}
\label{eq:chordalSDE}
 dG_t(z) = \frac{2}{G_t(z)} \,dt - \sqrt{\kappa} \, dB_t, \quad G_0(z) = z, \, z\in \mathbb{H}.
\end{equation}

In the dipolar case,
\begin{multline*}
\mathcal{T}:= \{(D,a,b,c) | D \textrm{ is a s.c.\ hyperbolic domain, } 
\\a,b,c \in P(D), \textrm{ ordered anticlockwise}\},
\end{multline*}
and
\begin{multline*}
 \Omega_{(D,a,b,c)} = \{ \gamma: \gamma \textrm{ is a non-self-traversing curve} 
\\\textrm{connecting $a$ and some point on the arc from $b$ to $c$}\}.
\end{multline*}

The diffusion equation corresponding to the dipolar $SLE$ equation is
\begin{equation}
\label{eq:dipolarSDE}
  d G_t(z) = \frac{1}{\tanh[G_t(z)/2]} - \sqrt{\kappa} B_t,\quad  G_0(z) = z, \quad z\in \mathbb{S},
\end{equation}
where $G_t(z) := g_t(z) - \sqrt{\kappa} B_t$

\subsection{General L\"owner theory}
The general L\"owner theory was developed in \cite{BracciEvolutionI, Bracci2, gumenyuk10}, and includes chordal, radial and dipolar Löwner theory as particular cases. In a certain sense, it is a  non-autonomous generalization of the theory of holomorphic semiflows (see Appendix~\ref{app:semigroups}).

\begin{definition}\label{evolution}
An evolution family of order $d\in[1,+\infty]$ is a two-parameter family  $\{g_{s,t}\}_{0\leq s\leq t<+\infty}\subset \mathrm{Hol}(\mathbb{D},\mathbb{D})$, such that the following three conditions are satisfied.
\begin{itemize}
 \item $g_{s,s} = id_{\mathbb{D}}$;
 \item $g_{s,t} = g_{u,t} \circ g_{s,u}$ for all $0 \leq s \leq u \leq t <+\infty$;
 \item for any $z \in \mathbb{D}$ and $T > 0$ there is a non-negative function $k_{z,T} \in L^d([0,T],\mathbb{R}),$ such that
\[
 |g_{s,u}(z) - g_{s,t}(z)| \leq \int_u^t k_{z,T}(\xi) d\xi, \quad z\in \mathbb{D}
\]
for all $0 \leq s \leq u \leq t \leq T$.
\end{itemize}
\end{definition}

An infinitesimal description of evolution families is given in terms of {\it Herglotz vector fields}.

\begin{definition}
 A (generalized) Herglotz vector field of order $d$ is a function \linebreak $V:\mathbb{D} \times [0,+\infty) \to \mathbb{C}$  satisfying the following conditions:
\begin{itemize}
 \item the  function $[0, +\infty) \ni t \mapsto V(t,z)$ is measurable for every $z\in \mathbb{D}$;
\item the function $z \mapsto V(t,z)$ is holomorphic in the unit disk for  $t \in [0,+\infty)$;
\item for any compact set $K \subset \mathbb{D}$ and for every  $T>0,$ there exists a non-negative function $k_{K,T} \in L^d([0,T],\mathbb{R}),$ such that
\[
 |V(t,z)|\leq k_{K,T}(t)
\]
for all $z\in K,$ and for  almost every $t \in [0,T]$;
\item for almost every $t\in[0,+\infty),$ the function $V(\cdot, t)$ is a semicomplete vector field.
\end{itemize}
\end{definition}

An important result of the general L\"owner theory is the fact that the evolution families can be put into a one-to-one correspondence with the Herglotz vector fields by means of the so-called generalized L\"owner ODE (in the same manner as semiflows can be put into a one-to-one correspondence with semicomplete vector fields in the autonomous case by means of equation (\ref{eq:semigroupivp}), Appendix~\ref{app:semigroups}). This can be formulated as the following theorem.
\begin{theorem}[{\cite[Theorem 1.1]{BracciEvolutionI}}]
\label{thm:LoewnerEq}
For any evolution family $\{g_{s,t}\}$ of order $d \geq 1$ in the unit disk there exists an essentially unique Herglotz vector field $V(t,z)$ of order $d,$ such that 
\begin{equation}
\label{eq:evolution}
\begin{cases}
 \frac{\partial g_{s,t}(z)}{\partial t} = V(t, g_{s,t}(z)),\\
g_{s,s}(z) = z,
\end{cases}
\end{equation}
for all $z\in \mathbb{D}$ and for almost all $t\in[s,+\infty)$.
Conversely, for any Herglotz vector field $V(t,z)$ of order $d\geq 1$ in the unit disk there exists a unique evolution family of order $d$, such that the equation above is satisfied.
\end{theorem}
In what follows we will always set $s=0$.

 Essential uniqueness in Theorem~\ref{thm:LoewnerEq} means that for any other Herglotz vector field $\tilde{V}(t,z)$ satisfying \eqref{eq:evolution}, the equality $\tilde{V}(t,z) = V(t,z)$ holds for all $z\in\mathbb{D},$ and almost all $t\in[0,+\infty)$.

Since the Herglotz vector $V(t,z)$ is a semicomplete vector field for almost every $t\in [0,+\infty),$  the following representation formula is valid for almost all $t\in [0,+\infty)$ by \eqref{eq:polynomsemicomplete} from Appendix \ref{app:semigroups}:
\begin{equation}
\label{eq:herglotzrepr}
V(t,z) = V(t,0) - z \,q (t,z) - \overline{V(t,0)} \, z^2.
\end{equation}
where  the function $z\mapsto q(t,z),$ $z\in \mathbb{D},$ is holomorphic and has nonnegative real part.

\subsection{Decreasing general L\"owner theory}
The theory discussed in the previous subsection describes the situation when the function $g_t(z):=g_{0,t}(z)$  maps the canonical domain (i.e, the unit disk) onto the evolution domain $D_t = g_t(\mathbb{D})$ for each moment $t\geq 0$. This framework was convenient for solving extremal problems in the theory of univalent functions in the \nth{20} century, however, when in the \nth{21} century the focus of L\"owner theory shifted towards describing growth processes ($SLE$), it  became unsatisfactory, mainly due to the fact that for $s<t$ it is not in general true that $D_t\subset D_s$.

$SLE$ theory uses the decreasing (time-reversed) approach, where the function $g_t(z)$ maps the evolution domain $D_t$ onto the canonical domain, $g_t(D_t) = \mathbb{D}$ (or the upper half-plane $\mathbb{H}:=\{z: \im z>0\}$, infinite strip $\mathbb{S}$, etc.) The family $\{D_t\}_{t\geq 0}$ is a decreasing family of domains, i.e.\ for $0\leq s < t< \infty$ the inclusion $D_t\subset D_s$ holds.

The decreasing approach was generalized in \cite{gumenyukDuality}. On the surface, the difference between the two versions is in the sign on the right-hand side of the governing equations.

\begin{definition}[{\cite[Definition 1.6]{gumenyukDuality}}]
 Let $d\in [1,+\infty]$. A family $\{h_t\}_{t\geq 0} \subset \mathrm{Hol}(\mathbb{D})$  is called a decreasing L\"owner chain of order $d$ if it satisfies the following conditions:
\begin{enumerate}
 \item each function $h_t:\mathbb{D}\to \mathbb{D}$ is univalent,
 \item $h_0 = \mathrm{id}_{\mathbb{D}},$ and $h_s(\mathbb{D}) \supset h_t(\mathbb{D})$ for $0 \leq s< t < +\infty$,
\item for any compact set $K \subset \mathbb{D}$ and for all $T > 0$ there exists a non-negative function $k_{K,T} \in L^d([0,T], \mathbb{R})$ such that
\[
 |h_s(z) - h_t(z)| \leq \int_s^t k_{K,T}(\xi)d\xi
\]
for all $z\in K$ and for all $0\leq s < t \leq T$.
\end{enumerate}

\end{definition}

\begin{theorem}[{\cite[Theorem 1.11]{gumenyukDuality}}]
\label{thm:gumenyukDuality}
 Let $V$ be a Herglotz vector field of order $d\in [1,+\infty]$. Then,
\begin{enumerate}
 \item For every $z\in \mathbb{D},$ there exists a unique maximal solution $g_t(z) \in \mathbb{D}$ to the following initial value problem
\begin{equation}
\label{eq:gendecLoewner}
\begin{cases}
\frac{\partial g_t(z)}{\partial t} = - V(t,g_t(z)),\\
g_0(z) = z.
\end{cases}
\end{equation}
\item For every $t\geq 0,$ the set $D_t$ of all $z\in \mathbb{D},$ for which $g_t(z)$ is defined at the point $t,$ is a simply connected domain, and the function $g_t(z)$ defined for all $z\in D_t$ maps $D_t$ conformally onto $\mathbb{D}$.
\item The functions $h_t := g_t^{-1}$ form a decreasing Löwner chain of order $d$, which is the unique solution to the following initial value problem for  PDE
\[
\begin{cases}
 \frac{\partial h_t(z)}{\partial t} = \frac{\partial h_t(z)}{\partial z} V(t,z),\\
h_0 = \mathrm{id}_{\mathbb{D}}.
\end{cases}
\]
\end{enumerate}

\end{theorem}

\subsection{Classical Löwner equations as special cases}

One of the motivations for developing general L\"owner theory was to create a unifying framework for studying chordal, radial and dipolar L\"owner equations which, despite having many similar properties, had to be treated as separate objects. Let us show that the radial, chordal and dipolar equations are indeed special cases of (\ref{eq:gendecLoewner}), where $V(t,z)$ is a Herglotz vector field and, in particular, it admits representation (\ref{eq:herglotzrepr}).

In the \emph{radial} case, 
\[
 V(t,z) = -z \, \frac{e^{i u_t}+ z}{e^{i u_t}-z},
\]
and it is easy to see that this vector field satisfies (\ref{eq:herglotzrepr}) with $q(t,z) = \frac{e^{i u_t}+ z}{e^{i u_t}-z}$.

The \emph{chordal} Löwner equation (\ref{eq:chordLoewner}) is formulated in the upper half-plane. This is, however, merely a convention, and to rewrite (\ref{eq:chordLoewner}) in the unit disk we consider the mapping $\phi: \mathbb{H}\to \mathbb{D}$ given by
\[
\phi(z):= - \frac{z- 2 i}{z+2 i}.
\]
The function $\phi(z)$ maps $2\,i \mapsto 0$, $\infty \mapsto 1$ and $0 \mapsto -1$. The family of maps $\tilde{g}_t:= \phi \circ g_t \circ \phi^{-1},$ then satisfies
\begin{equation}
\label{eq:chordalind}
 \begin{cases}
  \frac{\partial}{\partial t}\tilde{g}_t(z) = - \frac12 \frac{ i (\tilde{g}_t(z)  + 1)^3}{\tilde{g}_t(z) (u_t + 2\,i) + u_t - 2\,i},\\
g_0(z) = z,
 \end{cases}\\ z\in \mathbb{D}.
\end{equation}
so that 
\[
 V(t,z) = \frac{i}{2} \,\frac{(z + 1)^3}{z \,(u_t + 2 \,i) + u_t - 2\,i},
\]
which can be rewritten in the as
\[
 V(t,z) = \frac12 \frac{i}{u_t-2 \,i} - z \left[\frac{1}{(1+ u^2_t)^2} \,\frac{\frac{u_t - 2\,i}{u_t + 2\,i} - z}{\frac{u_t-2\,i}{u_t +2\,i} + z} - i \, \frac{u_t (3+u^2_t)}{(1+u^2_t)^2}\right] + \frac12 \frac{i}{u_t+ 2\,i} \,z^2,
\]
so that we can see that (\ref{eq:chordalind}) is indeed a special case of (\ref{eq:gendecLoewner}).

The function 
\[
 \phi(z):= i\, \frac{e^z-i}{e^z+i}
\]
maps $\mathbb{S}$ onto $\mathbb{D}$, sending $\pm\infty$ and 0 to $\pm i$ and $1$, respectively. We rewrite the \emph{dipolar} L\"owner equation (\ref{eq:dipolar}) in the unit disk by considering $\tilde{g}_t(z):= \phi \circ g_t \circ \phi^{-1}$:

\begin{equation}
 \begin{cases}
  \frac{\partial}{\partial t} \tilde{g}_t(z) =  \frac12 \, (1+\tilde{g}_t(z))^2 \,\frac{1 - i \tilde{g}_t(z) + e^{u_t} (\tilde{g}_t(z) -i)}{e^{u_t} (1+ i \tilde{g}_t(z)) - \tilde{g}_t(z) - i}, \\
  \tilde{g}_0(z) = z,
 \end{cases} z\in \mathbb{D}.
\end{equation}
Similarly to the previous case, we rewrite the vector field
\[
 V(t,z) = - \frac12 \, (1+z^2) \,\frac{1 - i \,z + e^{u_t} (z-i)}{e^{u_t}\, (1 + i \,z) - z - i},
\]
as
\begin{dmath*}
 V(t,z) = \frac{i}{2} + \frac{1}{1 - e^{u_t}} 
-z \left(\sech^2 u_t \,\frac{\frac{1 - e^{u_t}}{1-i\,e^{u_t}}-z}{\frac{1 - e^{u_t}}{1-i\,e^{u_t}}+z} - i \frac{\sinh u_t}{\cosh^2 u_t}\right) 
+\left(\frac{i}{2} - \frac{1}{1 - e^{u_t}}\right) z^2.
\end{dmath*}

In all these cases, the Herglotz vector fields $V(t,z)$ depend continuously on $t,$ and hence, the fields $V(t,z),$ as well as the corresponding Löwner chains, are of order $\infty$.

\section{Slit holomorphic vector fields}
\label{sec:Witt}
\subsection{Vector fields \texorpdfstring{$\ell_n$}{l_n}}

Consider the following Laurent polynomial vector fields in the upper half-plane
\[
 \ell^{\mathbb{H}}_{n}(z):= - z^{n+1}, \quad n \in \mathbb{Z}.
\]

Now, let $D\subset \hat{\mathbb{C}}$ be an arbitrary simply connected hyperbolic domain. Then there exists a conformal isomorphism $\phi: \mathbb{H} \to D$. We push forward the vector fields from $\mathbb{H}$ to $D$ by $\phi, $ and denote
\[
 \ell_n := \phi_* \ell^{\mathbb{H}}_n.
\]
Note that the fields $\ell_n$ are not determined uniquely by $D$ and depend, in fact, on the choice of the map $\phi$.

If we choose the unit disk $\mathbb{D}$ as the domain $D,$ then we agree to use the conformal isomorphism $\phi: \mathbb{H} \to \mathbb{D}$ given by 
\[
 \phi(z) =- \frac{z- 2 i}{z+2 i},
\]
and to use the notation $\ell^{\mathbb{D}}_{n}$ for $\phi_* \ell^{\mathbb{H}}$. Using the formula (\ref{eq:vfpf}) from Appendix~\ref{app:pushforward}, we can write the explicit expressions for $\ell^{\mathbb{D}}_n:$
\begin{equation}
-2^{n-1} \,(-i)^n  (z-1)^{n+1} (z+1)^{-n+1}.
\end{equation}

In the case of the infinite strip $\mathbb{S}:= \{z:0<\im z< \pi\}$  we choose the map $\psi: \mathbb{H} \to \mathbb{S}$ given by
\[
 \psi(z) = \Log \frac{2+z}{2-z},
\]
and denote $\ell^{\mathbb{S}}_n:= \psi_* \ell^{\mathbb{H}}_n,$ so that
\begin{equation}
\label{eq:stripell}
 \ell^{\mathbb{S}}_n(z) = -2^n\,\sinh (z) \tanh ^n\left(\frac{z}{2}\right).
\end{equation}

As we will see, the fields $\ell_{-1}$, $\ell_0$, $\ell_1$ are a convenient tool in the theory of holomorphic flows in a domain $D$, and the field $\ell_{-2}$ is directly relevant to $SLE$ theory.

\subsection{Representation of complete and semicomplete vector fields by \texorpdfstring{$\ell_n$}{l_n}}

It is known that complete holomorphic vector fields in a simply connected hyperbolic $D$ form a vector space of real dimension 3. The following proposition states that this vector space can be realized as $\mathrm{span}_{\mathbb{R}} \{\ell_{-1},\ell_0,\ell_1\}$.
 
\begin{proposition}
$V$ is a complete holomorphic vector field in $D$ if and only if it admits a decomposition
\begin{equation}
\label{genRepr}
 V = \sigma_{-1} \,\ell_{-1} + \sigma_0 \,\ell_{0} + \sigma_1 \,\ell_{1},
\end{equation}
where $\sigma_{-1},\sigma_0,\sigma_1 \in \mathbb{R}$.
\end{proposition}
\begin{proof}
Without loss of generality, we assume $D=\mathbb{D}$. By representation (\ref{eq:polynomsemicomplete}) from Appendix~\ref{app:semigroups}, a vector field is complete in $\mathbb{D}$ if and only if it admits the representation
\begin{equation}
\label{shoikhetRepr}
 V(z) = (a+i b) - i c z  + (-a + i b) z^2, \quad a,b,c\in \mathbb{R}.
\end{equation}
By expanding (\ref{genRepr}) we get
\begin{dmath*}
 V(z) = \left(\frac{\sigma_0}{2} + i \left(\sigma_1-\frac{\sigma_{-1}}{4}\right)\right) - i\, \left(2\,\sigma_{1} + \frac{\sigma_{-1}}{2}\right)\, z + \left(-\frac{\sigma_{0}}{2} + i \left(\sigma_1 - \frac{\sigma_{-1}}{4}\right)\right)z^2.
\end{dmath*}
Since this expression is of the form (\ref{shoikhetRepr}), the vector field is complete.

Conversely, every complete vector field $V(z) = (a+i b) - i c z  + (-a + i b) \, z^2$ may be decomposed as
\[
 V(z)= \left(- 2 b + c\right) \,\ell^{\mathbb{D}}_{-1}(z) + 2 a \,\ell^{\mathbb{D}}_0(z) + \frac14 \left( 2 b + c\right) \,\ell^{\mathbb{D}}_1(z).
\]
\end{proof}

To illustrate the geometric meaning of the basis vector fields $\ell^{\mathbb{D}}_{-1},$ $\ell^{\mathbb{D}}_0$ and $\ell^{\mathbb{D}}_1,$ we show their flow lines in Figure~\ref{fig:completebasis}.

\begin{figure}[hb]
 \centering
\subfloat[$\ell_{-1}$]{
\includegraphics[width=5cm,keepaspectratio=true]{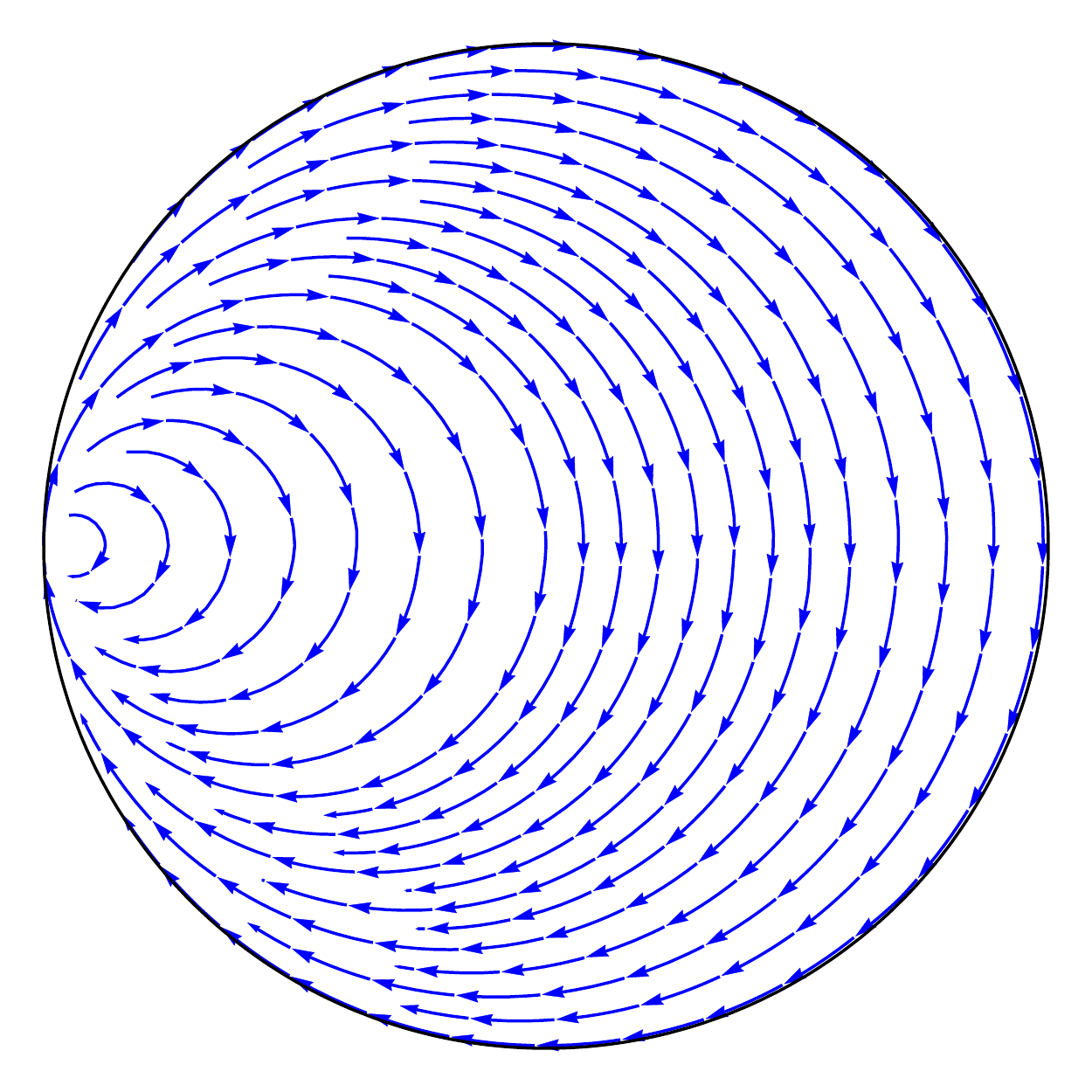}
}
\subfloat[$\ell_0$]{
 \includegraphics[width=5cm,keepaspectratio=true]{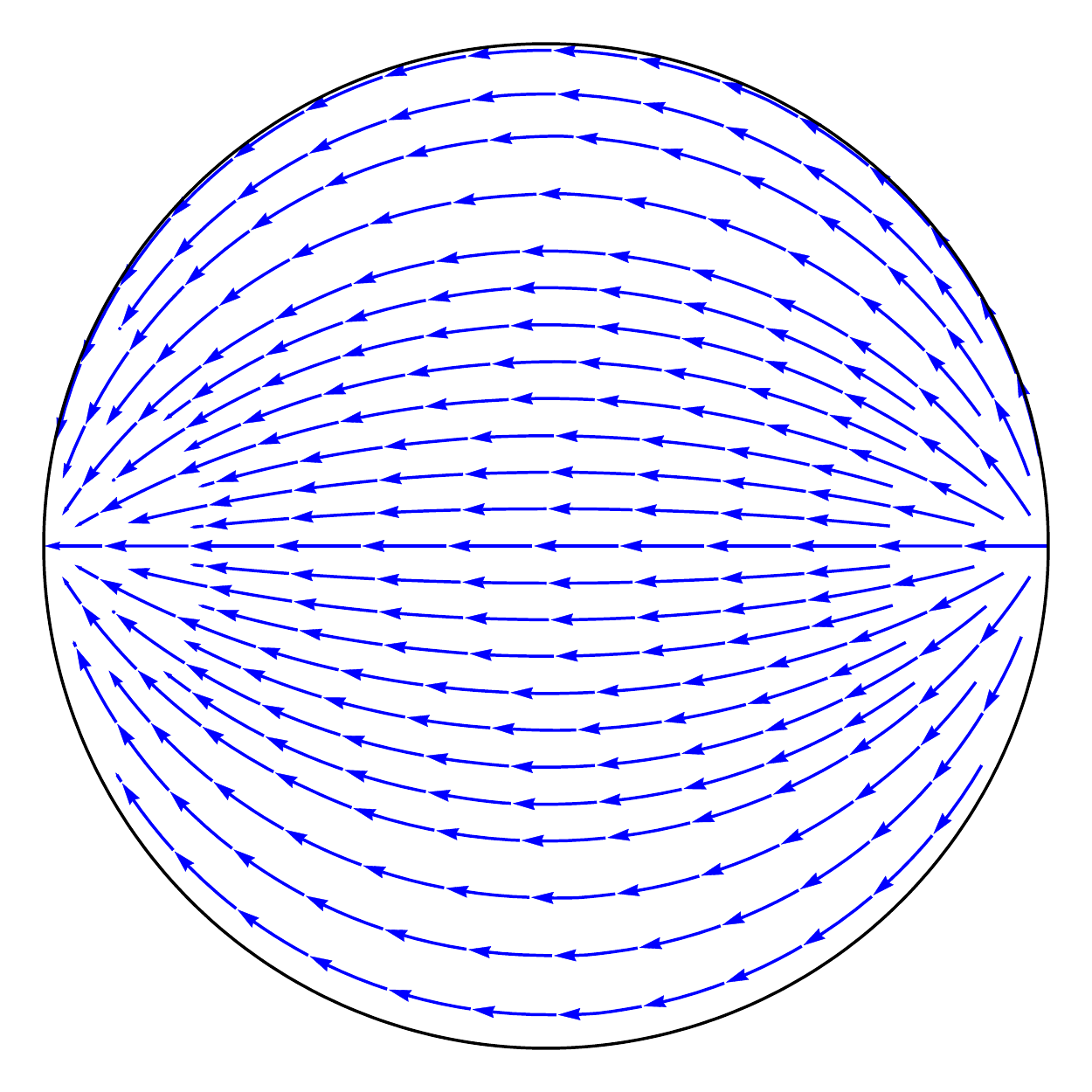}
}
{
\subfloat[$\ell_1$]{
 \includegraphics[width=5cm,keepaspectratio=true]{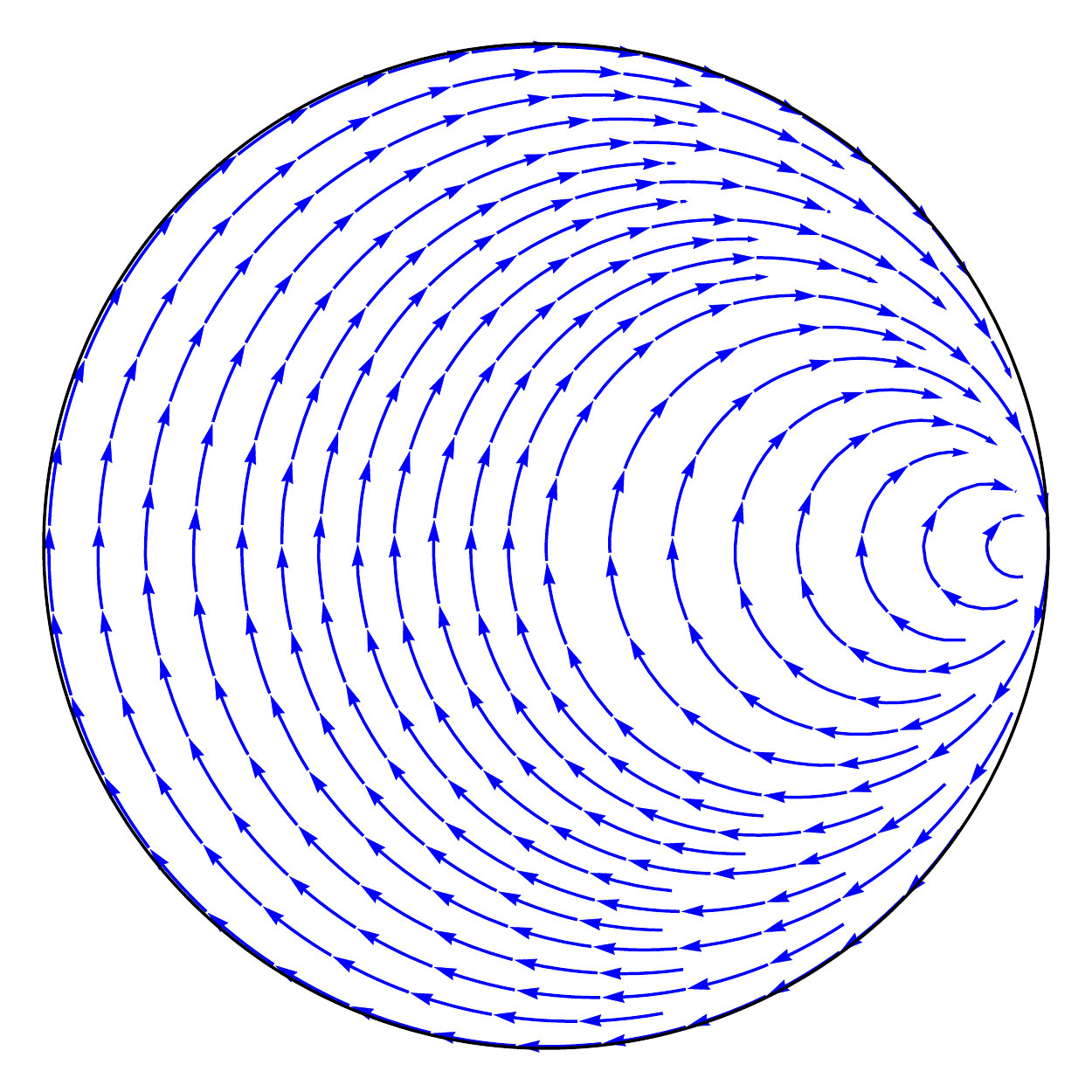}
}
}
 \caption{Basis for the space of complete holomorphic vector fields in $\mathbb{D}$.}
 \label{fig:completebasis}
\end{figure}

\begin{proposition}
 Vector fields $\ell_{-2}$ and $-\ell_{2}$ are semicomplete in $D$.
\end{proposition}
\begin{proof}
Again, we use the fact that a holomorphic vector field $V(z)$ in the unit disk is semicomplete if and only if it can be written as
\[
 V(z) = V(0)  -z \,q(z) - \overline{V(0)}\, z^2,
\]
where $q(z)$  is holomorphic and $\re q(z) \geq 0$. For $\ell^{\mathbb{D}}_{-2} = \frac18 \frac{(z+1)^3}{z-1},$ 
\[
 q(z) = \frac{\ell^{\mathbb{D}}_{-2}(0) - \ell^{\mathbb{D}}_{-2}(z)}{z} - \overline{\ell^{\mathbb{D}}_{-2}(0)} \, z = \frac{1}{2} \, \frac{1+z}{1 -z },
\]
has positive real part in $\mathbb{D}$. The same is true for $- \ell_{2}$:
\[
q(z) = \frac{-\ell^{\mathbb{D}}_{2}(0) +\ell^{\mathbb{D}}_{2}(z)}{z} + \overline{\ell^{\mathbb{D}}_{2}(0)} \, z = 8 \, \frac{1-z}{1 + z }.
\]
\end{proof}

Note that none of the vector fields $\pm \ell_{n},$ for $|n|\geq 3,$ is semicomplete. 

\subsection{Slit holomorphic vector fields}
In this subsection we consider vector fields in $\mathbb{D}$ that are tangent to the boundary (meaning that $\lim_{r\to 1} \re V(r e^{i \theta}) r e^{-i \theta} = 0$) at all boundary points, except perhaps for the point at $1$.

\begin{proposition}
\label{prop:slitFields}
The following statements are equivalent.
\begin{enumerate}[label=(\roman*)]
 \item \label{item:propstat1}  $V(z)$ is a semicomplete vector field in $\mathbb{D}$ satisfying 
\[
 \lim_{r \to 1}\re V(re^{i\theta})\, r e^{-i \theta} = 0
\]
for all $e^{i\theta} \in \partial \mathbb{D},$ except perhaps for $e^{i \theta} = 1$. 
\item \label{item:propstat2}$V(z)$ can be written as 
\begin{equation}
\label{eq:slitHerglotz}
 V(z) = \alpha  - z \left( i \beta  + \gamma \, \frac{1+z}{1-z}\right) -\overline{\alpha} z^2,
\end{equation}
for some $\alpha\in \mathbb{C},$ $\beta\in \mathbb{R}$ and $\gamma\geq 0$.
\item \label{item:propstat3} $V(z)$ can be written as
\begin{equation}
\label{eq:ellrepresentation}
 V(z) = b_{-2} \ell^{\mathbb{D}}_{-2}(z) + b_{-1} \ell^{\mathbb{D}}_{-1}(z) + b_0  \ell^{\mathbb{D}}_0 (z)+ b_1 \ell^{\mathbb{D}}_1(z),
\end{equation}
for some $b_{-2} \geq 0,$ $b_{-1}, b_0, b_1 \in \mathbb{R}$.
\end{enumerate}
\end{proposition}

\begin{proof} As a semicomplete field, $V(z)$ can be represented in the form
\[
 V(z) = \alpha - z\, (i\beta + p(z)) - \bar{\alpha} z^2,
\]
for some constants $\alpha\in \mathbb{C},$ $\beta \in \mathbb{R}$ and for a function $p(z),$ which is analytic in $\mathbb{D}$ and such that $\re p(z) \geq 0$ in $\mathbb{D}$ and $p(0) \geq 0$.

By the assumptions of the proposition we have that 
\begin{dmath*} 0 =  \lim_{r \to 1}\re V(re^{i\theta})\, r e^{-i \theta} 
=  \lim_{r\to 1} \re \left[ \alpha r e^{i\theta} - r^2 (i\beta + p(r {e^{ i\theta}})) - \bar{\alpha} r^3 e^{i\theta} \right] = \lim_{r\to 1} \re p(r e^{i\theta})
\end{dmath*}
for all $e^{i\theta} \in \partial \mathbb{D}$ except, perhaps, $e^{i\theta} =1$.

Then, $e^{-p(z)}$ is a singular function (see \cite{Hoffman}), and $p(z)$ can be written as
\[
 p(z) = \int \frac{e^{i\theta} + z}{e^{i\theta}-z} d\mu (\theta),
\]
for some singular measure $\mu$. Since the non-tangential limit $\lim_{r\to 1} \re p(re^{i\theta})$ vanishes for all $\theta\in (-\pi,\pi],$ except $\theta=0,$ the measure $\mu$ is given by $\mu(\theta) = \gamma \,\delta_0(\theta)$ for some $\gamma \geq 0$ (see \cite[proof of Lemma 3.7]{bracci2013contact} for a detailed discussion). Therefore,
\[
 p(z) = \gamma\, \frac{1+z}{1-z}, \quad \gamma\geq 0.
\]

Checking that \ref{item:propstat2} implies \ref{item:propstat1} is trivial.

To see that representations (\ref{eq:slitHerglotz}) and (\ref{eq:ellrepresentation}) are equivalent, note that

\begin{dmath*}
 V(z) = \alpha  - z \left( i \beta  + \gamma \, \frac{1+z}{1-z}\right) -\overline{\alpha} z^2 
= 2 \, \gamma \,\ell^{\mathbb{D}}_{-2}(z) + \left( \beta -2 \im \alpha \right)\ell^{\mathbb{D}}_{-1}(z) + \left(\frac{\gamma}{2} +2 \re \alpha \right)\ell^{\mathbb{D}}_0(z) + \frac14 \left( \beta +2 \im \alpha \right)\ell^{\mathbb{D}}_{1}(z),
\end{dmath*}
and the coefficients at $\ell^{\mathbb{D}}_{-2}$, $\ell^{\mathbb{D}}_{-1},$ $\ell^{\mathbb{D}}_0$, $\ell^{\mathbb{D}}_1$ are linearly independent.

\end{proof}

Representation (\ref{eq:ellrepresentation}) has an important advantage over (\ref{eq:slitHerglotz}), namely, the invariance of the coefficients $b_{j}$ under conformal maps. Given a map $\phi:\mathbb{D}\to D,$ the pushforward $\phi_* V(z)$ has the same form as $V(z)$:
\[
 \phi_*\left(b_{-2} \ell^{\mathbb{D}}_{-2} + b_{-1} \ell^{\mathbb{D}}_{-1} + b_0  \ell^{\mathbb{D}}_0 + b_1 \ell^{\mathbb{D}}_1\right) = b_{-2} \ell_{-2} + b_{-1} \ell_{-1} + b_0  \ell_0 + b_1 \ell_1.
\]

Since the flows generated by the vector fields described in Proposition~\ref{prop:slitFields} have slit geometry (i.e., the flow elements map the disk onto $\mathbb{D}\setminus \gamma,$ where $\gamma$ is a Jordan curve), in the rest of the paper we will call them \emph{slit vector fields}.

The set of  vector fields that are tangent to the unit circle everywhere except for one point, is of course, invariant under pushforwards with respect to the Möbius automorphisms of the unit disk, and the coefficients in representation (\ref{eq:slitHerglotz}) transform as described in the following proposition.
\begin{proposition}
\label{prop:slittransformation}
 Let 
\[
 V(z) = \alpha - z \,\left(i\beta + \gamma \, \frac{z_0 + z}{z_0 - z}\right) - \bar{\alpha} z^2,\quad z\in \mathbb{D},
\]
where
$ \alpha \in \mathbb{C}$, $\beta \in \mathbb{R}$, $\gamma \geq 0$, $|z_0| \in \partial\mathbb{D}$, and let
\[
 m(z) = e^{i\theta} \frac{z- a}{1 - \bar{a} z}, \quad a\in \mathbb{D}.
\]
Then

\begin{equation}
\label{eq:oftheform}
 (m_* V )(z) = \tilde{\alpha} - z \,\left(i\tilde{\beta} + \tilde{\gamma} \, \frac{m(z_0) + z}{m(z_0) - z}\right) - \bar{\tilde{\alpha}} z^2,
\end{equation}
with 
\begin{align}
 \tilde{\alpha} &= \frac{e^{i \theta}}{1 - |a|^2} \, V(a),\\
 \tilde{\beta} &=  \frac{4 \im (a \bar{\alpha}) + \beta (1+ |a|^2) + 4 \gamma \im (a \bar{z}_0) \left(-\frac{1+|a|^2}{|a-z_0|^4} \re(a \bar{z}_0) + \frac{1+|a|^4}{|a-z_0|^4}\right)}{1-|a|^2},\\
 \tilde{\gamma} &= \gamma \,|m'(z_0)|^2,
\end{align}
\end{proposition}
\begin{proof}
 Follows from straightforward calculations.
\end{proof}

\section{Stochastic holomorphic semiflows}
\label{sec:Kunita}

In this section we study stochastic differential equations generating stochastic flows of holomorphic self-maps of a simply connected domain $D\subset\mathbb{C}$.  Equations \eqref{eq:radialSDE}, \eqref{eq:chordalSDE}  and \eqref{eq:dipolarSDE} appearing in the classical $SLE$ theory are particular examples of such SDEs, and we use the results from the section later in the paper to define general slit holomorphic stochastic flows.

Kunita \cite{Kunita1984} considered the Stratonovich SDE
\[
 dG_t = b(G_t)\, dt  + \sum_{k=1}^n \sigma_k(G_t)\circ dB^k_t,
\]
on a paracompact manifold $M$, where $B^1_t, \ldots B^n_t$ are independent Brownian motions, and $b,\sigma_1, \ldots, \sigma_n$ are complete vector fields. Using a Lie theory approach, he showed that if the fields generate a finite-dimensional Lie algebra, then the corresponding flow consists of diffeomorphisms of $M$ (we quote this result in Appendix~\ref{app:completefields}).

In this section we prove two related results using the general Löwner theory. We restrict ourselves to the holomorphic case, but on the other hand, we allow the coefficient $b$ to be a time-dependent vector field, semicomplete for almost all  $t\geq 0$.

The following auxiliary statement is used in the proof of the main results of this section.

\begin{lemma}
\label{prop:complete}

 Let $\{M_{s,t}\}_{0\leq  s \leq t<\infty}$ be the evolution family associated to a Herglotz vector field $V(t, z),$ i.e., the functions $M_{s,t}$ satisfy the differential equation
\begin{equation}
\label{eq:appendixloewner}
\begin{cases}
 \frac{\partial }{\partial t}M_{s,t}(z) = V(t,M_{s,t}(z)),\\
M_{s,s}(z) = z,
\end{cases} \quad z\in D.
\end{equation}
Let $T\geq s$. The function $M_{s,T}$ is a holomorphic automorphism of the unit disk if and only if $V(t,z)$ is a complete vector field for almost all fixed $t\in [s,T]$.
\end{lemma}
\begin{proof}
Without loss of generality, we assume that $D = \mathbb{D}$.

 Denote by
\[
 \beta_t:= \frac{|M'_{s,t}(0)|}{1-|M_{s,t}(0)|^2}.
\]
By the Schwarz-Pick lemma, $\beta_t \leq 1,$ and the equality holds if and only if $M_{s,t} \in \mathrm{Aut}(\mathbb{D})$.
It follows from (\ref{eq:appendixloewner}) that 
\[
 \beta_t = \exp \int_s^t \re \left[ V'(r,M_{s,r}(0))+2 \, \frac{V(r,M_{s,r}(0)) \overline{M_{s,r}(0)}}{1 - |M_{s,r}(0)|^2} \right]dr.
\]

A well-known result from the theory of holomorphic semiflows states that 
\[
 \re \left[ V'(z) + 2 \, \frac{V(z)\, \bar{z}}{1-|z|^2} \right] \leq 0, \quad z\in \mathbb{D},
\]
for any semicomplete vector field $V(z)$ (see, e.g., \cite[Proposition 3.5.4]{ShoikhetSeimgroupsGFT},  but note that a different sign convention is used there). If, moreover, the equality holds for some $z\in \mathbb{D}$, then it holds as well for all $z\in \mathbb{D},$ and $V(z)$ is complete.

We conclude the proof by noting that, 
 \[M_{s,T} \in \mathrm{Aut}(\mathbb{D}) \Leftrightarrow  \beta_T = 1 \]
\[
\Leftrightarrow  \re \left[ V'(t,M_{s,t}(0))+2 \, \frac{V(t,M_{s,t}(0)) \,\overline{M_{s,t}(0)}}{1 - |M_{s,t}(0)|^2} \right] = 0 \textrm{ for almost all } t\in [s,T]
\]
\begin{multline*}
\Leftrightarrow  \re \left[ V'(t,z)+2 \, \frac{V(t,z) \,\overline{z}}{1 - |z|^2} \right] = 0 \quad  \textrm{ for almost all } t\in [s,T] \textrm{ and all }z\in \mathbb{D} 
\end{multline*}
\[
 \Leftrightarrow V(t,z) \textrm{ is a complete vector field for almost all }t\in[s,T].
\]
\end{proof}

We can use direct computations to find an explicit expression of the Herglotz vector field  $V(t,z)$ in terms of the evolution family $M_{s,t}$. Suppose $\{M_{s,t}\}_{0 \leq s \leq t <+\infty}$ is an evolution family consisting of automorphisms of the unit disk, so that
\[
M_{s,t}(z) = e^{i \Theta_{s,t}} \frac{z-A_{s,t}}{1 - \bar{A}_{s,t} z}.
\]

The function $M_{s,t}(z)$ is differentiable for almost all $t,$ therefore
\begin{dmath}
 \frac{\partial}{\partial t} M_{s,t}(z) 
= - \frac{e^{i\Theta_{s,t}} \dot{A}_{s,t}}{1 - |A_{s,t}|^2} + i \left( \frac{2 \im A_{s,t} \dot{\bar{A}}_{s,t}}{1-|A_{s,t}|^2} + \dot{\Theta}_{s,t}\right) M_{s,t}(z) + \frac{e^{-i\Theta_{s,t}} \dot{\bar{A}}_{s,t}}{1-|A_{s,t}|^2}\, M^2_{s,t}(z),
\end{dmath}

so that
\begin{dmath}
\label{eq:coeffrecalc}
 V(t,z) =- \frac{e^{i\Theta_{s,t}} \dot{A}_{s,t}}{1 - |A_{s,t}|^2} + i \left( \frac{2 \im A_{s,t} \dot{\bar{A}}_{s,t}}{1-|A_{s,t}|^2} + \dot{\Theta}_{s,t}\right) \,z + \frac{e^{-i\Theta_{s,t}} \dot{\bar{A}}_{s,t}}{1-|A_{s,t}|^2}\, z^2 
=- \frac{e^{i\Theta_{0,t}} \dot{A}_{0,t}}{1 - |A_{0,t}|^2} + i \left( \frac{2 \im A_{0,t} \dot{\bar{A}}_{0,t}}{1-|A_{0,t}|^2} + \dot{\Theta}_{0,t}\right) \,z + \frac{e^{-i\Theta_{0,t}} \dot{\bar{A}}_{0,t}}{1-|A_{0,t}|^2}\, z^2.
\end{dmath}

Now we return back to the theory of stochastic flows and study diffusion equations generating flows of holomorphic endomorphisms of the domain $D$. 
\begin{theorem}
\label{th:ourKunita}
Let $D$ be a simply connected hyperbolic domain, and consider the flow
\begin{equation}
\label{eq:theoreEq1}
\begin{cases}
 dG_t(z) = b(t, G_t(z))\,dt + \sum^n_{k=1}\sigma_k(G_t(z)) \circ dB^k_t,\\
G_0(z) = z,
\end{cases} \quad z\in D,
\end{equation}
where $B^1_t, \ldots, B^n_t$ are independent standard Brownian motions, the fields $\sigma_1(z), \ldots, \sigma_n(z)$ are complete, and $b(t,z)$ is a Herglotz vector field such that the maps $t\mapsto b(t,z)$ are continuous for each fixed $z\in D$.

Let  $H_t(z)$ be the flow 
\begin{equation}
\label{eq:theoreEq2}
 \begin{cases}
  dH_t(z) = \sum^n_{k=1}\sigma_k(H_t(z)) \circ dB^k_t,\\
   H_0(z) = z,
 \end{cases}z\in D,
\end{equation}
and denote by $g_t(z)$ the composite flow $g_t(z):= H^{-1}_t \circ G_t$.

Then
\begin{enumerate}
 \item the solutions to (\ref{eq:theoreEq1}) and (\ref{eq:theoreEq2}) exist for all $z\in D$ and for all $t\geq 0;$
 \item almost surely, $H_t \in \mathrm{Aut}(D)$ for all $t\geq 0;$
 \item almost surely, $G_t \in \mathrm{Hol}(D,D)$ for all $t\geq 0;$
 \item \label{stat:last} if moreover, $b(t_0, z)$ is complete for almost every $t_0\geq 0$, then, a.s., $G_t\in \mathrm{Aut}(D);$
 \item $g_t$ can be embedded into an $L^{\infty}$ evolution family $g_{s,t}$ (by putting $g_t = g_{s,t}),$ corresponding to the Herglotz vector field
\begin{equation}
\label{eq:herglotzfieldpf}
 V(t,z) = {H_{t}^{-1}}_*\, b (t,z) = \frac{1}{H_t'(z)} \,b(t,H_t(z)).
\end{equation}

\end{enumerate} 

\end{theorem}
\begin{proof}

Again, we assume in the proof that $D = \mathbb{D}$.

The flow $\{H_t\}_{t\geq 0}$ of SDE (\ref{eq:theoreEq2}) consists of automorphisms of $\mathbb{D}$ and exists for all $t\geq0$. Indeed, if $n=1$, the solution to \eqref{eq:theoreEq2} is given by 
\[
 H_t(z) = h_{B_t}(z),
\]
where $h_t(z)$ is the deterministic flow of automorphisms
\[
 \begin{cases}
  dh_t(z) = \sigma_1(h_t(z)) \,dt,\\
   h_0(z) = z,
 \end{cases}z\in \mathbb{D},\, t\in \mathbb{R}.
\]
In the general case, the Lie algebra generated by $\sigma_1, \ldots, \sigma_n$  is finite-dimensional (has dimension at most 3), and $\{H_t\}_{t\geq 0} \subset \mathrm{Aut}(\mathbb{D})$ by Kunita's theorem (Theorem \ref{th:Kunita}, Appendix \ref{app:completefields}).

By \eqref{eq:invflow} from Appendix \ref{subs:compinvflow}, the inverse flow $\{H^{-1}_t\}_{t\geq 0}$  satisfies
\[
 \begin{cases}
  dH^{-1}_t (z) = -  \sum^n_{k=1}\,{H^{-1}_{t}}_*\, \sigma_k (H^{-1}_t(z)) \circ dB^k_t,\\
 H_0(z) = z,
 \end{cases} z\in \mathbb{D},
\]
and by \eqref{eq:compflow}, the composite flow $\{g_t\}, $ where $g_t := H_t^{-1} \circ G_t,$  satisfies
\begin{dmath*}
 dg_t(z) = - \sum^n_{k=1}{H^{-1}_{t}}_*\, \sigma_k \,(g_t(z))\circ dB^k_t + {H^{-1}_{t}}_*\, b\,(t, g_t(z))\, dt + \sum^n_{k=1}{H^{-1}_{t}}_*\, \sigma_k( g_t(z)) \circ dB^k_t 
= {H^{-1}_{t}}_*\, b\,(t, g_t(z)) \,dt.
\end{dmath*}

In other words, $\{g_t\}_{t\geq 0}$ is the solution to the problem
\[
 \begin{cases}
 \frac{\partial}{\partial t} g_t(z) =V(t,g_t(z)),\quad t\geq 0,\\
 g_0(z) = z,
 \end{cases}z\in \mathbb{D}.
\]
where $V(t,z)$ is the time-dependent random vector field given by $V(t,z) := {H^{-1}_{t}}_*\, b\,(t, z)$. Since $b(t,z)$ is semicomplete for almost all $t\geq 0, $ and since $H^{-1}_t \in \mathrm{Aut}(\mathbb{D})$ for all $t\geq 0, $ the vector field ${H^{-1}_{t}}_*\, b\,(t, z)$ is semicomplete for almost every $t\geq0$  (as a pushforward of a semicomplete field by a disk automorphism). Moreover, since both $H^{-1}_t(z)$ and $b(t,z)$ are continuous with respect to $t,$ the map $t\mapsto {H^{-1}_{t}}_*\, b\,(t, z)$ is also continuous for each $z\in \mathbb{D}$, therefore, the vector field $V(t,z)$ is a Herglotz vector field of order $\infty$. 

By Theorem~\ref{thm:LoewnerEq}, we conclude that the family $\{g_t\}_{t\geq 0}$ can be embedded into an evolution family of order $\infty,$ and in particular, $\{g_t\}_{t\geq 0}\subset \mathrm{Hol}(\mathbb{D}, \mathbb{D})$. 

It follows that $G_t = H_t \circ g_t \in \mathrm{Aut}(\mathbb{D})$ for all $t\geq 0$.  Thus, the solution $G_t(z)$ exists for all $t\geq 0$ and for all $z\in \mathbb{D}$ almost surely.

If we additionally assume that $b(t_0, z)$ is complete for almost every $t_0\geq 0$, then $V(t,z)$ is a complete vector field for almost all $t\geq 0$. By Lemma \ref{prop:complete}, $\{g_t\}_{t\geq 0}\subset \mathrm{Aut}(\mathbb{D})$, $\{G_t\}_{t\geq 0}\subset \mathrm{Aut}(\mathbb{D})$, and statement \ref{stat:last} follows.
\end{proof}

The following theorem is a dual (i.e., time-reversed) version of Theorem~\ref{th:ourKunita}, and is more relevant to the $SLE$ theory.

\begin{theorem}
\label{th:ourReverseKunita}
  Let $G_t(z)$ be the solution to
\[
 dG_t(z) = -b(t,G_t(z)) \,dt + \sum^n_{k=1}\sigma_k(G_t(z)) \circ dB^k_t,\quad G_0(z)= z, \quad z\in D,
\]
where $B^1_t, \ldots, B^n_t$ are independent standard Brownian motions, the fields $\sigma_1(z), \ldots, \sigma_n(z)$ are complete, and $b(t,z)$ is a Herglotz vector field. Let the maps $t\mapsto b(t,z)$ be continuous for every fixed $z\in D$.

Let  $H_t(z)$ be the flow of 
\begin{equation}
 \begin{cases}
  dH_t(z) = \sum^n_{k=1}\sigma_k( H_t(z)) \circ dB^k_t,\\
   H_0(z) = z,
 \end{cases}z\in D,
\end{equation}
and denote by $g_t(z)$ the composite flow $H^{-1}_t \circ G_t$.

Let $T(z,\omega)$ be the explosion time of $G_t(z),$ and $D_t :=\{z:T(z)>t\} \subset D$. Then
\begin{enumerate}
 \item $G_t$ maps conformally $D_t$ onto $D;$
  \item the family $\{g^{-1}_t\}_{t\geq 0}$ is a decreasing L\"owner chain of order $\infty;$
\item  the Herglotz vector field of order $\infty$ corresponding to $g_{t}$ is given by
\begin{equation}
\label{eq:herglotzfieldreversepf}
 V(t,z) = -{H^{-1}_{t}}_*\, b (t,z) = -\frac{1}{H'_t(z)} \,b(t, H_t(z)).
\end{equation}
\end{enumerate}
\end{theorem}
\begin{proof}
 The proof of this theorem repeats the proof of Theorem~\ref{th:ourKunita} verbatim, except for the use of Theorem~\ref{thm:gumenyukDuality} instead of Theorem~\ref{thm:LoewnerEq}.
\end{proof}

In the rest of the paper we restrict the number of Brownian motions in Theorem~\ref{th:ourReverseKunita} to one ($n=1$), and take $b(t,z)$ to be a time-independent slit holomorphic vector field, i.e., a field of the form \eqref{eq:ellrepresentation}.

\section{General slit L\"owner chains}
\label{sec:GSLE}
The general Löwner theory provides us with the means for studying an exceptionally rich family of processes in a complex domain. Results of the previous section suggest, however, that we should concentrate our attention on chains generated by Herglotz vector fields of form \eqref{eq:herglotzfieldreversepf}, because they are closely related to diffusion processes.

Furthermore, if we require slit geometry for the evolution domains, the vector field $b(z)$   should be of the form \eqref{eq:ellrepresentation} (we will prove in Section~\ref{sec:geometry} that this choice leads to slit geometry). 

These considerations motivate the following definition.

\begin{definition}
\label{d8}
Let $b(z)$ be a slit vector field in $D$,
\[
 b = b_{-2} \ell_{-2} + b_{-1} \ell_{-1} + b_0 \ell_0 + b_1\ell_1, \quad   b_{-2} >0, \,b_{-1},\, b_0, \, b_1 \in \mathbb{R},
\] 
 and $\sigma(z)$ be a complete vector field in $D$,
\[
 \sigma = \sigma_{-1} \ell_1 + \sigma_0 \ell_0 + \sigma_1 \ell_1, \quad \sigma_{-1}, \, \sigma_0, \, \sigma_1 \in \mathbb{R},
\]
such that $\sigma_{-1} \neq 0$. Let $\{h_t\}_{t\in \mathbb{R}}$ be the flow of automorphisms of $D$ generated by $\sigma(z),$ i.e.,
\[
 \begin{cases}
  \frac{\partial}{\partial t} h_t(z) = \sigma(h_t(z)), \\
 h_0(z) = z, \quad z\in D.
 \end{cases}
\]

Let $u_t:[0,+\infty) \to \mathbb{R}$ be a continuous real-valued function, such that $u_0 = 0, $ and let
\[
 V(t,z) :=  ({h^{-1}_{u_t}}_*\, b)(z) =  \frac{1}{h'_{u_t}(z)} \,b(h_{u_t}(z)).
\]
The family of maps $\{g_t\}_{t\geq 0}$ solving the problem
\begin{equation}
\label{eq:gslitloewner}
 \begin{cases}
  \frac{\partial}{\partial t} g_t(z) =  - V(t,g_t(z)), \quad t\geq 0,\\
  g_0(z) = z, \quad z\in D,
 \end{cases}
\end{equation}
is called a \emph{general slit L\"owner chain} driven by $b,$ $\sigma$ and $u_t$. The differential equation in \eqref{eq:gslitloewner} is called the \emph{general slit Löwner equation}.
\end{definition}

In the terminology of \cite{gumenyukDuality}, the family $\{g^{-1}_t\}_{t\geq 0}$ is a reverse L\"owner chain. By Theorem~\ref{thm:gumenyukDuality}, we know that for each $t\geq0,$ $g_t$ maps the simply connected domain
\[
 D_t := \{z\in D :  \textrm{ the solution of (\ref{eq:gslitloewner}) is defined up to time }t\},
\]
conformally onto $D$. We call the domain $D_t$ the \emph{evolution domain of \eqref{eq:gslitloewner}} at time $t,$ the domain $D = D_0$ the \emph{canonical domain} and the set $K_t:=D\setminus D_t$ the hull generated by the Löwner chain $\{g_t\}_{t\geq 0}$ at  time $t$.

If we take the driving function to be  $u_t =  \sqrt{\kappa}\,B_t$ $\kappa >0,$  then $G_t:= h_{u_t} \circ g_t$ satisfies
\[
 \begin{cases}
  dG_t(z) = -b(G_t(z)) \,dt +  \sqrt{\kappa}\,\sigma (G_t(z)) \circ dB_t,\quad t\geq 0,\\
  G_0(z) = z, 
 \end{cases}z\in \mathbb{D}.
\]
In this case we call $\{G_t\}_{t\geq 0}$ a \emph{slit holomorphic stochastic flow} driven by $b$ and $\sigma$.

In Section~\ref{sec:geometry} we show that the local behavior of the hulls generated by this stochastic flow is similar to the hulls of classical $SLEs$. This suggests an alternative term for $\{G_t\}_{t\geq 0}$, namely, $(b,\sigma)$-$SLE_{\kappa}$.

\subsection{Classical \texorpdfstring{$SLEs$}{SLEs} as special cases} Using different combinations of the slit vector field $b$ and the complete vector field $\sigma$ we can obtain all classical versions of the L\"owner equation.

For instance, the SDE corresponding to the \textbf{chordal} $SLE$ in $\mathbb{H}$ is 
\[ dG_t(z) = \frac{2}{G_t(z)} \,dt - \sqrt{\kappa} \, dB_t, \quad G_0(z) = z, \, z\in \mathbb{H},
\]
so that 
\[
 b(z) = -\frac{2}{z} = 2\ell^{\mathbb{H}}_{-2}(z), \quad \sigma(z)= -1 = \ell^{\mathbb{H}}_{-1}(z).
\]

In the \textbf{radial} case,
\[
 dG_t(z) = G_t(z) \frac{1+ G_t(z)}{1-G_t(z)} \,dt - i \sqrt{\kappa} G_t  \circ dB_t, \quad  G_0(z) = z, \quad z \in \mathbb{D},
\]
so that
\[
 b(z) = - z\,\frac{1 + z}{1-z} = 2 \,\ell^{\mathbb{D}}_{-2} (z) + \frac12 \,\ell^{\mathbb{D}}_{0}(z),
\]
\[
 \sigma(z) = - i z = \ell^{\mathbb{D}}_{-1}(z) + \frac14 \ell^{\mathbb{D}}_1(z).
\]

In the \textbf{dipolar} case
\[
  d G_t(z) = \frac{dt}{\tanh[G_t(z)/2]} - \sqrt{\kappa} \,dB_t,\quad  G_0(z) = z, \quad z\in \mathbb{S},
\]
and
\[
 b(z) = -\frac{1}{\tanh[z/2]}= 2 \,\ell^{\mathbb{D}}_{-2} (z) - \frac12 \,\ell^{\mathbb{D}}_{0}(z), \quad \sigma(z) = -1= \ell^{\mathbb{S}}_{-1}(z)-\frac14\ell^{\mathbb{S}}_{1}(z).
\]
(see the  expressions for $\ell_n^{\mathbb{S}}$ in \eqref{eq:stripell}).

We summarize this in the table below.
\bigskip
\begin{center}
\begin{tabular}{|c|c|c|}
\hline 
$SLE$ type & $b$ & $\sigma$ \\ 
\hline 
Chordal & $2 \ell_{-2}$ & $\ell_{-1}$ \\ 
\hline 
Radial & $2 \ell_{-2} + \frac12 \ell_{0}$ & $\ell_{-1}+\frac14\ell_{1}$ \\ 
\hline 
Dipolar & $ 2 \ell_{-2} - \frac12 \ell_{0}$ & $\ell_{-1}-\frac14\,\ell_{1}$ \\ 
\hline 
\end{tabular} 
\end{center}
\bigskip

Similar representations for the fields $b$ and $\sigma$ in terms of $\ell_n$ were first introduced for the chordal case in \cite{BB02}, for the radial case in \cite{BBRadial} and for the dipolar case in \cite{bauer2004zigzag}.

Let us now take a look at two examples of slit holomorphic stochastic flows that have not been studied previously.

\subsection{Example 1: \texorpdfstring{$ABP$ $SLE$}{ABP SLE}}
Let us choose $b$ as in the chordal case, and $\sigma$ as in the radial case, i.e.,  we take
\[
b =2 \ell_{-2},\quad \sigma = \ell_{-1}+\frac14\ell_{1},\]
which can be written explicitly in the unit disk as
\[
 b^{\mathbb{D}}(z) =\frac14\,\frac{(z+1)^3}{z-1}, \quad \sigma^{\mathbb{D}}(z) = - i z.
\]
The complete vector field $\sigma^{\mathbb{D}}(z)$ generates the flow 
\[
 h_t(z) = z e^{-it},
\]
and then a real-valued continuous driving function $u_t$ determines the Herglotz vector field
\begin{equation}
\label{eq:abpfield}
 V(t,z) = \frac{1}{h'_{u_t}(z)} \, b(h_{u_t}(z)) = -\frac{1}{ 4 e^{i u_t}} \frac{(e^{iu_t} + z)^3}{e^{iu_t} - z}.
\end{equation}

We arrive at the differential equation
\begin{equation}
\begin{cases}
\frac{\partial}{\partial t} g_t(z) = \frac{1}{ 4\,e^{i u_t}} \frac{\left( e^{i u_t} + g_t(z) \right)^3}{e^{i u_t} -g_t(z) }, \\
g_0(z) = z, \quad z \in \mathbb{D}.
\end{cases}
\end{equation}

If we put $u_t = \sqrt{\kappa} \, B_t,$ then the composition
\[
 G_t(z) = h_{\sqrt{\kappa} \, B_t} \circ g_t(z) = \frac{g_t(z)}{e^{i\,\sqrt{\kappa}\,B_t}}
\]
satisfies the diffusion equation
\begin{equation}
\label{eq:abpdiff}
\begin{cases}
 dG_t(z) =-  \frac14\frac{(G_t(z) +1)^3}{G_t(z) - 1} \, dt-   i\,\sqrt{\kappa}\, G_t(z) \circ dB_t, \\
G_0(z) = z,
\end{cases} z\in \mathbb{D}.
\end{equation}

In \cite{Ivanov2012a}, Herglotz vector fields of the form 
\begin{equation}
\label{eq:ABPfield}
 V(t,z) = \frac{(\tau_t - z)^2}{\tau_t} \,p\left(z/\tau_t\right), \quad \re p(z) \geq 0,\quad \tau_t\in \partial \mathbb{D},
\end{equation}
were considered. For each fixed $t_0\geq 0,$ the point $\tau_{t_0}$ is the attracting point of the the semiflow generated by the vector field $V(t_0,z)$.  By putting $\tau_t = -e^{i \sqrt{\kappa} B_t},$ $p(z) = \frac14\,\frac{1+z}{1-z},$  we arrive at the same vector field as in \eqref{eq:abpfield}. By that reason we call the stochastic flow \eqref{eq:abpdiff} $ABP$ $SLE$ ($ABP$ stands for \emph{attracting boundary point}).

The vector fields $b(z)$ and $\sigma(z)$ have no common zeros, therefore the functions of the flow $\{g_t\}_{t\geq 0}$ have no common fixed points in $\hat{\mathbb{D}}$, in contrast to the classical cases of $SLE$.

The results of five numerical simulations for $\kappa = 2$, done using a modification of the Zipper algorithm \cite{MarshallZipper}, are shown in Figure~\ref{fig:ABP}. The simulations suggest that the hulls are generated by fractal curves, similar to the case of classical $SLEs$. We give a proof of this statement later in this section. Note that due to the lack of common fixed points of the fields $b(z)$ and $\sigma(z),$ the curves terminate at random points of $\mathbb{D}$.

\begin{figure}[!ht]
 \centering
 \includegraphics[width=6cm,keepaspectratio = true]{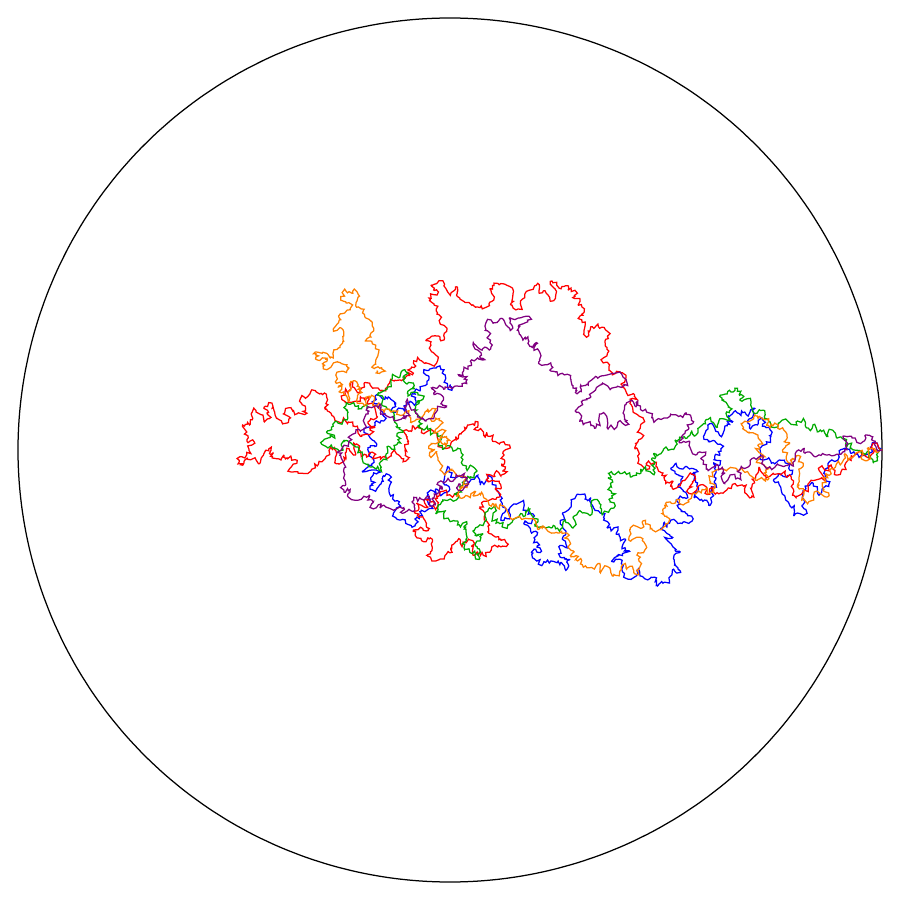}
 \caption{Sample curves of $ABP$ $SLE$. }
 \label{fig:ABP}
\end{figure}

\subsection{Example 2}
Let us now take $b(z)$ as in the radial case, and $\sigma(z)$ as in the chordal case, that is,
\[
 b = 2 \ell_{-2} + \frac12 \ell_0, \quad
 \sigma = \ell_{-1},
\]
 which can be written explicitly in $\mathbb{D}$ as
\[
 b^{\mathbb{D}}(z) =  z\,\frac{ z+1}{z-1}, \quad \sigma^{\mathbb{D}}(z) = -\frac{1}{4} i\, (z+1)^2,
\]
or in $\mathbb{H}$ as
\[
 b^{\mathbb{H}}(z) = -\frac{2}{z} - \frac12\, z, \quad \sigma^{\mathbb{H}}(z) = -1.
\]
Note that the vector fields have a common zero at $-1$ (if we work in the unit disk), or at $\infty$ (if we work in the half-plane).

From the technical point of view, in this example it is significantly easier to work with the half-plane in this example. 

The vector field $\sigma^{\mathbb{H}}(z)$ generates the flow
\[
 h_t(z) = z -t, \quad z\in \mathbb{H},
\]
of automorphisms of $\mathbb{H},$ and a driving function $u_t$ leads to the Herglotz vector field
\[
 V(t,z) = \frac{1}{h'_{u_t}(z)} \,b^{\mathbb{H}}(h_{u_t}(z)) = -\frac{2}{z-u_t} - \frac12 \,(z-u_t).
\]
For $u_t = \sqrt{\kappa}B_t$ we obtain the following ordinary and stochastic differential equations:
\[
\begin{cases}
 \frac{\partial}{\partial t}g_t(z) = \frac{2}{g_t(z) - \sqrt{\kappa} \,B_t} + \frac12 \, (g_t(z) - \sqrt{\kappa}\,B_t), \\
g_t(z) = z,
\end{cases} z\in \mathbb{H},
\]
\begin{equation}
\label{eq:antiabpflow}
\begin{cases}
  dG_t(z) = \left(\frac{2}{G_t(z)} + \frac12 G_t(z)\right)\,dt - \sqrt{\kappa} dB_t,\\
G_t(z)= z, 
\end{cases}z\in \mathbb{H}, 
\end{equation}
where $G_t(z) = g_t(z) - \sqrt{\kappa}\,B_t$.

Results of numerical simulations in the unit disk for $\kappa = 3$ are shown in Figure~\ref{fig:AntiABP}.

\begin{figure}[!ht]
 \centering
 \includegraphics[width=6cm,keepaspectratio=true]{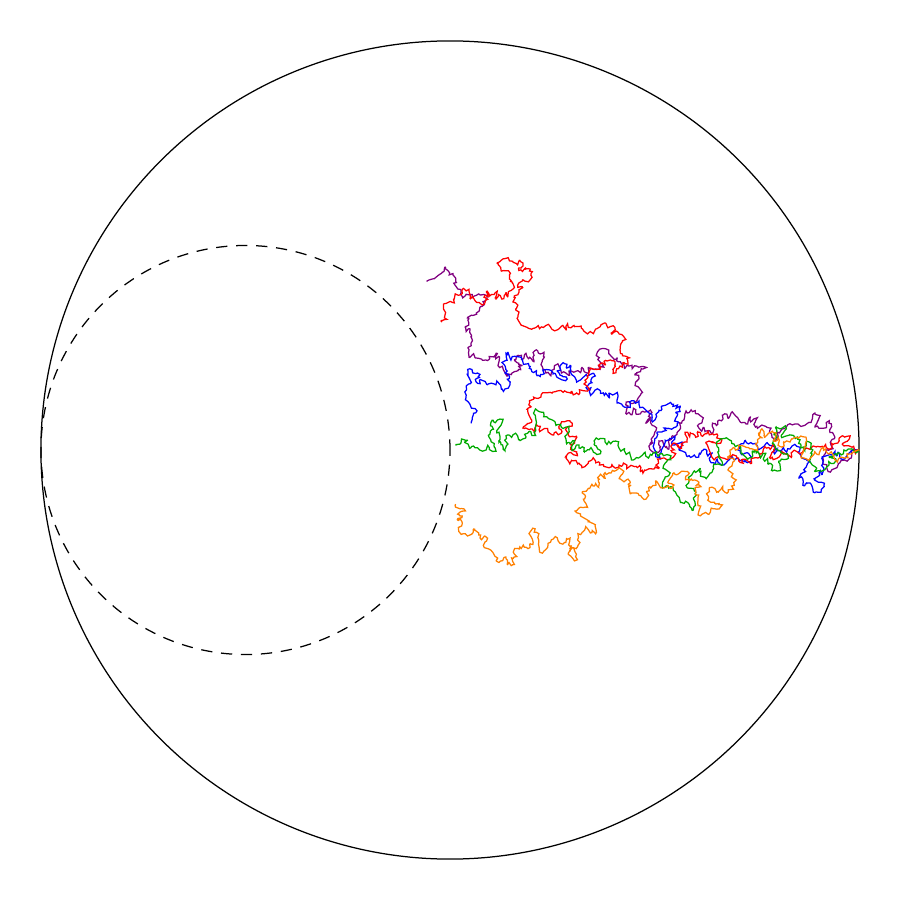}
 \caption{Curves generated by flow \eqref{eq:antiabpflow}.}
\label{fig:AntiABP}
\end{figure}

The hulls of this evolution always stay outside the disk $\{z: |z+1/2| <1/2 \}$ (if we work in $\mathbb{D}$) or inside the  strip $\{z: 0< \im z < 2\}$ (if we work in $\mathbb{H}$).

To see this in the case of $\mathbb{D}$, one can map the domain $\{z : |z|<1 \textrm{ and } |z + 1/2|>1/2\}$ onto $\mathbb{D}$ by a suitable conformal map $\phi$. Then we verify that the vector field $\phi_* \sigma$ is complete in $\mathbb{D},$ and the vector field $\phi_* b$ is semicomplete in $\mathbb{D}$, and, finally, apply Theorem~\ref{th:ourReverseKunita}.

\subsection{Equivalence and normalization of general slit Löwner chains}\label{subs:simplenorm}
A general slit L\"owner chain is determined by a triple $(b,\sigma, u_t)$. This correspondence, however, is not one-to-one. It may happen that different combinations of $b,$ $\sigma$ and $u_t$ produce the same L\"owner chain $\{g_t\}_{t\geq 0}$. It may  also happen that the resulting chains can be transformed one into another by means of a simple transformation, for instance, by a linear time reparameterization.

In this section we define precisely what we mean by a ``simple'', or \emph{elementary} transformation of a triple $(b,\sigma, u_t)$. If two slit Löwner chains are determined by triples that can be transformed one into the other by means of elementary transformations, we call such chains \emph{equivalent}.

In particular, we show that we can always find a representative in the equivalence class of triples 
so that the conditions
\begin{equation}
\label{eq:normc1}
 b_{-2} = 2,
\end{equation}
and
\begin{equation}
\label{eq:normc2}
 \sigma_{-1} = 1
\end{equation}
are satisfied. In other words we can always assume that that the vector fields $b$ and $\sigma$ have the form
\begin{equation}
\label{eq:bnormalize}
 b = 2\, \ell_{-2} + b_{-1} \ell_{-1} + b_0 \ell_0 + b_1\ell_1, \quad   b_{-1},\, b_0, \, b_1 \in \mathbb{R},
\end{equation}
 and
\begin{equation}
\label{eq:sigmanormalize}
 \sigma = \ell_1 + \sigma_0 \ell_0 + \sigma_1 \ell_1, \quad \sigma_0, \, \sigma_1 \in \mathbb{R},
\end{equation}
respectively. 

If the vector fields $b$ and $\sigma$ are of the form \eqref{eq:bnormalize} and \eqref{eq:sigmanormalize}, then we say that a general slit Löwner chain driven by $b$ and $\sigma$ is  a \emph{normalized slit Löwner chain}.

Below we list the transformations that we regard as elementary.

\subsubsection{Scaling of the driving function $V_c$.} Let $c\in \mathbb{R}\setminus\{0\}$. Then we define the transformation $V_c$ by the formula
\[
 V_c : (b, \sigma, u_t) \mapsto (b, c \, \sigma, c^{-1} \,u_t).
\]

The triple $ (b,\,\tilde{\sigma}, \, \tilde{u}_t)= (b,c\sigma, c^{-1} u_t)$ produces the same Löwner chain as the triple $(b,\sigma, u_t)$ . To see this, note that
the flow $\{\tilde{h}_t\}_{t\in \mathbb{R}}$ of $\tilde{\sigma}$ differs from the  flow of $\{h_t\}_{t \in \mathbb{R}}$ of $\sigma$ only by time reparameterization, $\tilde{h}_t = h_{c t},$ so that
\[
\tilde{h}^{-1}_{\tilde{u}_t*} b = {h^{-1}_{u_t}}_*\, b
\]


\subsubsection{Time scaling $T_c$} For a constant $c>0$ we can consider the time reparameterized chain $\{\tilde{g}_t\}_{t\geq 0},$ where $\tilde{g}_t = g_{ct},$ and note that $\{\tilde{g}_t\}_{t\geq 0}$ is generated by the triple $(c\,b,\sigma,   u_{ct})$. Indeed,
\[
 \frac{\partial}{\partial t} \tilde{g}_{t}(z) = -c \, V(ct, g_{ct}(z))  = -\left(h^{-1}_{u_{ct}*}\, c\, b \right)(\tilde{g}_t(z)).
\]

This motivates us to define the transformation $T_c$
\[
T_c : (b, \sigma, u_t) \mapsto  (c\,b, \sigma, u_{c\,t}), \quad c>0.
\]

\begin{remark*} The transformations $V_c$ and $T_c$ are sufficient for imposing the normalization conditions \eqref{eq:normc1} and \eqref{eq:normc2}.
Indeed, due to the fact that $\sigma_{-1} \neq 0,$ we can apply $V_{1/\sigma_{-1}}$ and transform $(b,\sigma, u_t) \mapsto (b,\frac{\sigma}{\sigma_{-1}},  \sigma_{-1} \, u_t),$ so that the vector field
\[
\frac{\sigma}{\sigma_{-1}} = \ell_{-1} + \frac{\sigma_0}{\sigma_{-1}} \,\ell_0 + \frac{\sigma_1}{\sigma_{-1}}  \,\ell_1
\]
has coefficient $1$ at $\ell_{-1}$. 

Then, since $b_{-2}>0,$ we apply $T_{2/b_{-2}}$, so that the triple  $(b,\frac{\sigma}{\sigma_{-1}},  \sigma_{-1} \, u_t)$ is further transformed into \[(\tilde{b}, \,\tilde{\sigma}, \,\tilde{u}_t) =  \left(\frac{2 b}{b_{-2}}, \,\frac{\sigma}{ \sigma_{-1}}, \,\sigma_{-1}\, u_{\frac{2 t}{b_{-2}}}\right),\]
and the vector field $\tilde{b}$ has coefficient $2$ at $\ell_{-2}$. 
\end{remark*}

\subsubsection{Drift $D_c$} We add a linear drift to the driving function $u_t$ and simultaneously modify the field $b$:
\[
 D_c: (b, \sigma, u_t) \mapsto (b + c \,\sigma, \sigma, u_t +  c\,t), \quad c\in \mathbb{R}.
\]
If $\{h_t\}_{t \in \mathbb{R}}$ is the flow of automorphisms generated by $\sigma$, then the new slit Löwner chain $\{\tilde{g}_t\}_{t\geq 0}$ can be expressed as $\tilde{g}_t = h_{-c \,t} \circ g_t$.

To show this, we use \eqref{eq:compflow} from Appendix~\ref{subs:compinvflow}:
\begin{dmath*}
\frac{\partial}{\partial t} \tilde{g}_t(z) = 
-c \, \sigma \,(\tilde{g}_t(z)) - \left({h_{-c \,t}}_{\, *} {h_{-u_t}}_{\, *} b \right)\,(\tilde{g}_t(z))
= - \left({h^{-1}_{u_t +c \,t}}_{\,*}\, (b + c\, \sigma) \right)\,(\tilde{g}_t(z)).
\end{dmath*}

We can apply $D_c$ with $c= \frac{\frac32\,b_{-2} \sigma_0 - b_{-1}}{\sigma_{-1}},$ so that the coefficients in the normalized fields $\tilde{b}$ and $\tilde{\sigma}$  satisfy  the following normalization condition
\begin{equation}
\label{eq:normc3}
-2\,\frac{b_{-1}}{b_{-2}}+3 \frac{\sigma_0}{\sigma_{-1}} = 0.  
\end{equation}

This particular choice is motivated by CFT considerations, which we do not deal with in this paper. 

\subsubsection{Parabolic rotation $R_c$} 
Let $\{r_c(z)\}_{c\in \mathbb{R}}$ be the solution to
\[
\begin{cases}
\frac{\partial}{\partial c}r_c(z) = \ell_1(r_c(z)), \\ 
r_0(z) = z, 
\end{cases}z\in D.
\]
In the case of $\mathbb{H},$ the explicit expression is given by  $r_c = \frac{z}{1 + c z}$. In $\mathbb{D}$, $r_t$ is a Möbius automorphism of $\mathbb{D}$ such that $r_c(1)=1$ and $r'_c(1)=1$.

The transformation $R_c$ is defined by
\[
 R_c: (b, \sigma, u_t) \mapsto ({r_c}_{\,*} b,{r_c}_{\,*} \sigma, u_t), \quad c\in \mathbb{R}.
\]

 The new slit Löwner chain $\{\tilde{g}_t\}_{t\geq 0}$ can be related to the original chain using conjugation with $r_c,$ i.e., $\tilde{g}_t = r_c \circ g_t \circ r^{-1}_c,$ $t\geq 0,$ as shown below
\begin{dmath*}
 \frac{\partial}{ \partial t} \tilde{g}_t(z) = \frac{\partial}{\partial t} \left(r_c \circ g_t \circ r_c^{-1} (z)\right)  = -\left({r_c}_{\,*} {h^{-1}_{u_t}}_{*} b\right)\,(\tilde{g}_t(z)) 
= -\left({r_c}_{\,*} {h^{-1}_{u_t}}_{*}  {r^{-1}_c}_{\,*} {r_c}_{\,*} b\right)\,(\tilde{g}_t(z))  =-\left( (r_c \circ h_{u_t} \circ r^{-1}_c)^{-1}_{*} {r_c}_{\,*} b\right)\,(\tilde{g}_t(z))
=  -\left({{\widetilde{h}}^{-1}_{u_t}} {}_{*} \,\tilde{b} \right)(\tilde{g}_t(z)),
\end{dmath*}
 where $\tilde{b} = {r_c}_{\,*} b,$ and the flow $\{\tilde{h}_t\}_{t\in \mathbb{R}}$ with $\tilde{h}_t = r_c \circ h_t  \circ r^{-1}_c,$ is generated by $\tilde{\sigma} = {r_c}_{\,*}\sigma$.

An easy way to see how the fields $b$ and $\sigma$ are transformed under the action of ${r_c}_{\,*}$ is to use the explicit expressions in $\mathbb{H}$ for $\ell_k,$ $k=-2,\ldots 1$ and $r_c$. Due to the linearity of the pushforward operation, the formulas we obtain remain valid for an arbitrary canonical domain $D$:

\begin{dmath*}
\tilde{b} = {r_c}_{\,*}(b_{-2} \ell_{-2} + b_{-1} \ell_{-1} + b_{0} \ell_0 + b_1 \ell_1)  = b_{-2} \ell_{-2} + (b_{-1} - 3 c b_{-2}) \,\ell_{-1} + (b_0 - 2 c \,b_{-1} + 3 c^2 b_{-2}) \,\ell_0 + (b_1 - c\,b_0 + c^2 
\,b_{-1} - c^3 b_{-2}) \,\ell_1,
\end{dmath*}

\begin{dmath*}
\tilde{\sigma} = {r_c}_{\,*}(\sigma_{-1} \ell_{-1} + \sigma_{0} \ell_0 + \sigma_1 \ell_1)   = \sigma_{-1} \,\ell_{-1} + (\sigma_0 - 2 c \,\sigma_{-1}) \,\ell_0 + (\sigma_1 - c\,\sigma_0 + c^2 
\,\sigma_{-1}) \,\ell_1. 
\end{dmath*}

Application of $R_c$ preserves the normalization conditions \eqref{eq:normc1}, \eqref{eq:normc2} and \eqref{eq:normc3}.

\subsubsection{Scaling $S_c$ and $S^{0}_c$} The transformation $S_c$ is analogous to $R_c,$ except for the fact that we use the flow  $\{s_c\}_{c\in \mathbb{R}}$ generated by $\ell_0$ in this case. 

The flow $\{s_c(z)\}_{c\in \mathbb{R}}$ is defined as the solution to
\[
\begin{cases}
\frac{\partial}{\partial c}s_c(z) =\ell_0(s_c(z)), \\ 
s_0(z) = z, 
\end{cases}z\in D.
\]
In the case of $\mathbb{H},$ $s_c$ is simply multiplication by the real constant $e^{-c}$, i.e., $s_c(z) = z \, e^{-c}$, hence the term ``scaling''. In the unit disk, $s_c$ is a Möbius automorphisms of $\mathbb{D}$ that leaves the points $\pm 1$ fixed.

The transformation $S_c$ is again defined by
\[
 S_c: (b, \sigma, u_t) \mapsto ({s_c}_{\,*} b,{s_c}_{\,*} \sigma, u_t).
\]
Similarly to the previous case, $\tilde{g}_t = s_c \circ g_t \circ s^{-1}_c,$ $t\geq 0$.

The coefficients of the vector fields $\sigma$ and $b$ transform as follows.
\begin{dmath*}
 s_{c\,*} (b_{-2} \ell_{-2} + b_{-1} \ell_{-1} + b_{0} \ell_0 + b_1 \ell_1) =  e^{-2c}\,b_{-2} \ell_{-1} + e^{-c}\,b_{-1} \ell_{-1} + b_{0} \ell_0 + e^{c}\,b_1 \ell_1,
\end{dmath*}
\[
 s_{c\,*}(\sigma_{-1} \ell_{-1} + \sigma_{0} \ell_0 + \sigma_1 \ell_1) =   e^{-c}\,\sigma_{-1} \ell_{-1} + \sigma_{0} \ell_0 + e^{c}\,\sigma_1 \ell_1.
\]

As we can see, when $S_c$ acts on $(b,\sigma, u_t),$ it does not necessarily  preserve the normalization conditions \eqref{eq:normc1}, \eqref{eq:normc2}, \eqref{eq:normc3}. To resolve this problem, we compose $S_c$ with $T_{e^{2c}}$ and $V_{e^c}$ and define the transformation
\[
 S^0_c = V_{e^c} \circ T_{e^{2c}} \circ S_c : (b, \sigma, u_t) \mapsto (e^{2c} \, {s_c}_{\,*}\, b, e^{c} \, {s_c}_{\,*}\, \sigma, e^{-c} \, u_{e^{2c}\,t}),
\]
which keeps the conditions unchanged.

\medskip

We summarize the properties of the transformations described above in Table~\ref{tab:transforms}.

\afterpage{
\begin{landscape}
 \vspace*{\fill}
\begin{table}
{%
\newcommand{\mc}[3]{\multicolumn{##1}{##2}{##3}}
\begin{center}
\begin{tabular}{l|l|lllll}\cline{2-7}
 & \mc{6}{c|}{Action of the transformation on}\\\cline{2-7}
 & $b$ & \mc{1}{l|}{$b_n$} & \mc{1}{l|}{$\sigma$} & \mc{1}{l|}{$ \sigma_n$} & \mc{1}{l|}{$ u_t$} & \mc{1}{l|}{$ g_t$}\\\hline
\mc{1}{|l|}{$V_c$} & $b \mapsto b$ & \mc{1}{l|}{$b_n\mapsto b_n$} & \mc{1}{l|}{$\sigma \mapsto c\,\sigma$} & \mc{1}{l|}{$\sigma_n\mapsto c\,\sigma_n$} & \mc{1}{l|}{$u_t\mapsto c^{-1} \, u_t$} & \mc{1}{l|}{$g_t\mapsto g_t$}\\\hline
\mc{1}{|l|}{$T_c$} & $b \mapsto c\, b$ & \mc{1}{l|}{$b_n\mapsto c\,b_n$} & \mc{1}{l|}{$\sigma \mapsto \sigma$} & \mc{1}{l|}{$\sigma_n\mapsto \sigma_n$} & \mc{1}{l|}{$u_t\mapsto u_{ct}$} & \mc{1}{l|}{$g_t\mapsto g_{ct}$}\\\hline
\mc{1}{|l|}{$D_c$} & $b\mapsto b+ c \sigma$ & \mc{1}{l|}{$\begin{aligned}b_{-2} &\mapsto b_{-2}\\ b_{-1} &\mapsto b_{-1}+ c \sigma_{-1}\\b_{0} &\mapsto b_{0}+ c \sigma_{0}\\b_{1} &\mapsto b_{1}+ c \sigma_{1}\end{aligned}$} & \mc{1}{l|}{$\sigma \mapsto \sigma$} & \mc{1}{l|}{$\sigma_n\mapsto\sigma_n$} & \mc{1}{l|}{$u_t\mapsto u_t+c\,t$} & \mc{1}{l|}{$g_t\mapsto h^{-1}_{ct} \circ g_t$}\\\hline
\mc{1}{|l|}{$R_c$} &
$b \mapsto {r_c}_* b$ 
 & \mc{1}{l|}{$\begin{aligned}
b_{-2}&\mapsto b_{-2} \\
b_{-1}&\mapsto b_{-1} -3c\,b_{-2} \\
b_{0}&\mapsto b_{0} -2cb_{-1} +3c^2b_{-2} \\
b_{1}&\mapsto b_{1} -c\,b_{0} +c^2b_{-1} -c^3 b_{-2}
 \end{aligned}
$} & \mc{1}{l|}{$\sigma \mapsto {r_c}_* \sigma$ } & \mc{1}{l|}{$\begin{aligned}
\sigma_{-1}&\mapsto\sigma_{-1}\\
\sigma_0&\mapsto\sigma_{0} -2c\sigma_{-1} \\
\sigma_1 &\mapsto \sigma_{1} -c\sigma_{0} +c^2\sigma_{-1} \\
 \end{aligned}$} & \mc{1}{l|}{$u_t\mapsto u_t$} & \mc{1}{l|}{$g_t\mapsto r_c \circ g_t \circ r^{-1}_c$}\\\hline
\mc{1}{|l|}{$S_c$} & 
$b\mapsto{s_c}_* b$   & \mc{1}{l|}{$b_n\mapsto e^{nc} b_n$} & \mc{1}{l|}{$\sigma \mapsto {s_c}_* \sigma$} & \mc{1}{l|}{$\sigma_n \mapsto e^{nc} \sigma_n$ } & \mc{1}{l|}{$u_t\mapsto u_t$} & \mc{1}{l|}{$g_t\mapsto s_c \circ g_t \circ s^{-1}_c$}\\\hline
\mc{1}{|l|}{$\begin{aligned}S^0_c&\\=V&_{e^c} T_{e^{2c}}  S_c \end{aligned}$} & $e^{2c} {s_c}_{*} b$ & \mc{1}{l|}{$\begin{aligned}b_{-2}&\mapsto b_{-2} \\
b_{-1}&\mapsto  e^{c} b_{-1} \\
b_{0}&\mapsto  e^{2c} b_{0} \\
b_{1} &\mapsto  e^{3s} b_{1}
\end{aligned}$} & \mc{1}{l|}{$\sigma \mapsto e^c {s_c}_* \sigma$} & \mc{1}{l|}{$\begin{aligned}
\sigma_{-1}&\mapsto   \sigma_{-1} \\
\sigma_{0}&\mapsto  e^{c} \sigma_{0} \\
\sigma_{1} &\mapsto  e^{2s} \sigma_{1}
\end{aligned}$} & \mc{1}{l|}{$u_t\mapsto e^{-c} u_{e^{2c} t}$} & \mc{1}{l|}{$g_t\mapsto s_c\circ g_{e^{2c} t} \circ s^{-1}_c$}\\\hline
\end{tabular}
\end{center}
} 
\caption{Elementary transformations of slit Löwner chains}
\label{tab:transforms}
\end{table}
\vspace*{\fill}
\end{landscape}
}

\subsubsection*{Stochastic case} Consider the slit Löwner chain generated by the triple $(b, \sigma, \sqrt{\kappa} B_t + \mu \,t),$ where $\kappa  > 0,$ $\mu \in \mathbb{R}$ and the fields $b$ and $\sigma$ are not necessarily normalized. The corresponding stochastic differential equation is 
\[
\begin{cases}
 dG_t(z) = -b(G_t(z)) \, dt + \sqrt{\kappa} \, \sigma(G_t(z))\circ d(B_t + \mu \, t),\\
G_0(z)= z, \quad z\in D.
\end{cases}
\]

After we apply the transformations $V_{1/\sigma_{-1}},$ $T_{2/b_{-2}}$ and $D_{\frac{\frac32 b_{-2} \,\sigma_0 -b_{-1}}{\sigma_{-1}}},$ the triple $(b, \sigma, \sqrt{\kappa} \, B_t + \mu t)$ is transformed into
\[
(\tilde{b}, \tilde{\sigma}, \sqrt{\tilde{\kappa}}\, \tilde{B}_t + \tilde{\mu} t)
\]
 with
\begin{align*}
\tilde{b} &= \frac{2}{b_{-2}} \, b + \frac{\frac32\, b_{-2} \,\sigma_0 -b_{-1}}{\sigma_{-1}} \,\sigma,\\ \tilde{\sigma} &=  \frac{\sigma}{\sigma_{-1}}, \\ \tilde{\kappa} &= \kappa\,\frac{2  \sigma^2_{-1}}{b_{-2}}, \\\tilde{B}_t &= \frac{1}{\sqrt{\frac{2}{b_{-2}}}}\, B_{\frac{2t}{b_{-t}}},\\
\tilde{\mu} &=\frac{2 \sigma_{-1}}{b_{-2}}\,\mu + \frac{\frac32 b_{-2} \,\sigma_0 - b_{-1}}{\sigma_{-1}}.
\end{align*}

We conclude that a slit Löwner chain driven by the multiple of a Brownian motion $\sqrt{\kappa}\,B_t$, after imposing the normalization conditions \eqref{eq:normc1}, \eqref{eq:normc2}, \eqref{eq:normc3} can  be considered as a  slit Löwner chain driven by the multiple of a Brownian motion $\sqrt{\tilde{\kappa}}\, \tilde{B}_t$ with the drift $\tilde{\mu} t$.

The new triple generates the slit holomorphic stochastic flow
\[
 \begin{cases}
  d\tilde{G_t}(z) = -\tilde{b}(\tilde{G}_t(z)) \,dt + \tilde{\sigma}(\tilde{G}_t(z)) \circ d\left(\sqrt{\tilde{\kappa}}\, \tilde{B}_t + \tilde{\mu}\, t\right),\\
\tilde{G}_0(z)= z, \quad z\in D.
 \end{cases}
\]

\subsubsection*{Analysis} 
Originally, the family of triples $(b, \sigma, u_t)$ determining general slit Löwner chains was parameterized by 7 real parameters (coefficients of $b$ and $\sigma$) and the driving function $u_t$. We have defined a 5-parameter family of elementary transformations: $V_c$, $T_c$, $D_c,$ $R_c$ and $S^0_c$. Using $V_c$, $T_c,$ and $D_c$ we have imposed normalization conditions \eqref{eq:normc1}, \eqref{eq:normc2}, \eqref{eq:normc3}  and eliminated three parameters. 

The transformations $R_c$ and $S^{0}_c$ preserve these normalization conditions. In principle, we can choose two more normalization conditions to eliminate two more parameters (we do not know if there is a canonical way to choose these conditions). 

In the end, we obtain a family of slit Löwner chains parameterized by two ($7-3-2$) real parameters and the continuous driving function $u_t$.  No two slit Löwner chains in this family are equivalent to each other, meaning that one chain cannot be transformed into another using a combination of the elementary transformations defined above.

Let us now repeat the same procedure in the stochastic case. The family of triples $(b, \sigma, \sqrt{\kappa} B_t + \mu t),$ $\kappa\geq 0, $ $\mu \in \mathbb{R}$ defines an 8-parameter family of slit holomorphic stochastic flows. After normalization, the driving function can be written as $\sqrt{\tilde{\kappa}} \tilde{B}_t + \tilde{\mu} t,$ and it is not true in general that $\tilde{\kappa} = \kappa,$ and $\tilde{\mu} =\mu$. In the end, we are left with a 4-parameter family of slit stochastic flows (two of them are the same as in the deterministic case, and the other two are the parameters $\tilde{\kappa}$ and $\tilde{\mu}$).

\section{Relations between different slit Löwner chains}
\label{sec:relations}

Throughout the section we use the normalization introduced in Section \ref{subs:simplenorm}, i.e.,
\begin{equation}
\label{eq:bnormalized}
 b = 2\, \ell_{-2} + b_{-1} \ell_{-1} + b_0 \ell_0 + b_1\ell_1, \quad   b_{-1},\, b_0, \, b_1 \in \mathbb{R},
\end{equation}
 and
\begin{equation}
\label{eq:sigmanormalized}
 \sigma = \ell_{-1} + \sigma_0 \ell_0 + \sigma_1 \ell_1, \quad \sigma_0, \, \sigma_1 \in \mathbb{R}.
\end{equation}

In most proofs we work with explicit unit disk expressions for such vector fields, i.e.,
\begin{dmath*}
 b^{\mathbb{D}}(z) = \alpha  - z \left( i \beta  + \frac{1+z}{1-z}\right) -\overline{\alpha} z^2 ,
\end{dmath*}
with $\alpha = \left(\frac{b_0}{2} -\frac14 + i \left(b_1-\frac{b_{-1}}{4}\right)\right)$, $\beta = \frac{b_{-1}}{2}+ 2\,b_1$, and
\begin{dmath*}
 \sigma^{\mathbb{D}}(z) = \left(\frac{\sigma_0}{2} + i \left(\sigma_1-\frac{1}{4}\right)\right) - i\, \left( 2\, \sigma_1 + \frac12 \right)\, z + \left(-\frac{\sigma_{0}}{2} + i \left(\sigma_1 - \frac{1}{4}\right)\right)\,z^2,
\end{dmath*}
so that $\sigma^{\mathbb{D}}(1) = - i$.

\subsection{Reducing a general slit Löwner chain to a radial chain}
Let $\{K_t\}_{t\geq 0}$ be the family of hulls generated by a slit Löwner chain. In the following theorem we explain how $\{K_t\}_{t\geq 0}$ can be described in terms of a radial Löwner chain.

\begin{theorem} \label{th:toradial}
 Let $\{g_t\}_{t\geq 0}$ be a normalized slit Löwner chain in $\mathbb{D}$ driven by $b$, $\sigma$ and $u_t$. Let $\{h_{t}\}_{t\in \mathbb{R}}$ be the flow of automorphisms generated by $\sigma$. Let $\{K_t\}_{t\geq 0}$ be the corresponding family of hulls, $K_t = \mathbb{D}\setminus g^{-1}_t(\mathbb{D})$,  and let $T_{\max{}}$ be the possibly infinite time
\[
 T_{\max{}} := \sup\{t: 0 \not\in K_t \}.
\]
Then 
\begin{enumerate}[label=\roman*.]
 \item there exists a radial Löwner chain $\{\tilde{g}_{\tilde{t}}\}_{\tilde{t} \geq 0}$  such that 
\[
 g^{-1}_t (\mathbb{D}) = \tilde{g}^{-1}_{\lambda(t)}(\mathbb{D}), \quad \textrm{for all } t\in[0,T_{\max{}});
\]
for a continuously differentiable monotonically increasing function $\lambda:[0,T_{\max{}}) \to \mathbb{R}$, $\lambda(0) =0;$ the Löwner chain $\{\tilde{g}_{\tilde{t}}\}_{\tilde{t} \geq 0}$ is defined uniquely on $[0, \tilde{T}_{\max{}})$, where $\tilde{T}_{\max{}} = \lim_{t\to T_{\max{}}} \lambda(t)$;
\item
the function $\lambda(t)$ satisfies
\begin{equation}
\label{eq:lambdaexpr}
 \dot{\lambda}(t) = \left(\frac{\left(M_t \circ h^{-1}_{u_t}\right)'(1)}{\left(M_t \circ h^{-1}_{u_t}\right)(1)}\right)^2, \quad t\in[0,T_{\max{}}),
\end{equation}
 and the radial driving function $\tilde{u}_{\tilde{t}}$ satisfies
\begin{equation}
\label{eq:utexpression}
 \tilde{u}_{\lambda(t)} = - i \log \left(M_t \circ h^{-1}_{u_t}\right) (1),\quad t\in[0,T_{\max{}}),
\end{equation}
where $M_t = \tilde{g}_{\lambda(t)} \circ g^{-1}_t$, and the branch of the logarithm is chosen so that $\tilde{u}_0 = 0$ and the value of $\tilde{u}_{\lambda(t)}$ changes continuously as $t$ increases from $0$ to $T_{\max{}}$.
 
\end{enumerate}

\end{theorem}
\begin{proof}

Let $D_t = g^{-1}_t(\mathbb{D})$.  The function
\[
 \lambda(t) := \log \frac{|g'_t(0)|}{1-|g_t(0)|^2},
\]
is well-defined for $t \in [0,T_{\max{}}),$ is monotonically increasing, continuously differentiable, and is such that $\lambda(0) = 0$. We set $\tilde{T}_{\max{}} := \lim_{t\to T_{\max{}}} \lambda(t)$. The inverse function $\lambda^{-1}(\tilde{t})$ is well-defined for $\tilde{t} \in [0, \tilde{T}_{\max{}})$.

Let
\[
M_t (z):= \frac{|g'_{t}(0)|}{g'_{t}(0)} \, \frac{z - g_{t}(0)}{1-\overline{g_{t}(0)} \, z},\quad t \in [0,T_{\max{}}),
\]
and let
\begin{dmath*}
 \tilde{g}_{\tilde{t}} (z) {:=} M_{\lambda^{-1}(\tilde{t})} \circ g_{\lambda^{-1}(\tilde{t})}(z) =\frac{|g'_{\lambda^{-1}(\tilde{t})}(0)|}{g'_{\lambda^{-1}(\tilde{t})}(0)} \, \frac{g_{\lambda^{-1}(\tilde{t})}(z) - g_{\lambda^{-1}(\tilde{t})}(0)}{1-\overline{g_{\lambda^{-1}(\tilde{t})}(0)} \, g_{\lambda^{-1}(\tilde{t})}(z)},\quad {\tilde{t} \in [0,\tilde{T}_{\max{}})}.
\end{dmath*}
Then $\tilde{g}_{\tilde{t}}$ maps $D_{\lambda^{-1}(\tilde{t})}$ conformally onto $\mathbb{D},$ 
so that
\[
 g^{-1}_t (\mathbb{D}) = \tilde{g}^{-1}_{\lambda(t)}(\mathbb{D}), \quad \textrm{for all } t\in[0,T_{\max{}}).
\]
Moreover, 
\[
 \tilde{g}'_{\tilde{t}}(0) = \frac{|g'_{\lambda^{-1}(\tilde{t})}(0)|}{1-|g_{\lambda^{-1}(\tilde{t})}(0)|^2} = e^{\lambda (\lambda^{-1}(\tilde{t}))} = e^{\tilde{t}},
\]
which implies that the family $\{\tilde{g}_{\tilde{t}}\}_{\tilde{t}\in [0,\,\tilde{T}_{\max})}$ satisfies the radial Löwner-Kufarev equation
\begin{equation}
\label{eq:radrecalc}
 \begin{cases}
  \frac{\partial}{\partial \tilde{t}} \tilde{g}_{\tilde{t}}(z) = \tilde{g}_{\tilde{t}}(z) \, p (\tilde{t}, \tilde{g}_{\tilde{t}}(z)),\\
  \tilde{g}_0(z) = z,
 \end{cases}\\ \tilde{t}\in [0, \tilde{T}_{\max{}}), \, z\in \mathbb{D},
\end{equation}
where $p(\tilde{t},\cdot)$ is analytic, $p(\tilde{t},z) = 1 + c_1(\tilde{t}) \, z+\ldots,$ $\re p(z) >0$ in $\mathbb{D}$, and $p(\cdot, z)$ is measurable (see \cite{Pommerenke2}).

Let $W(t,z)$ denote the velocity field of the family $\{M_t\}_{t\in[0,T_{\max{}})}$, so that
\[
 \begin{cases}
  \frac{\partial}{\partial t} M_t (z) = W(t, M_t(z)),\\
  M_0(z) = z,
 \end{cases} z \in \mathbb{D}, \, t\in[0, T_{\max{}}),
\]
and let $V(t,z)$ denote the velocity field of $\tilde{g}_{\lambda(t)}$ with respect to $t$ multiplied by $-1$, so that
\[
 \begin{cases}
  \frac{\partial }{\partial t} \tilde{g}_{\lambda(t)}(z) = - V(t,\tilde{g}_{\lambda(t)}(z)),\\
 \tilde{g}_{\lambda(0)} = z,
 \end{cases} z\in \mathbb{D}, \, t\in [0, T_{\max{}}).
\]

On the one hand, it follows from \eqref{eq:radrecalc} that
\begin{equation}
\label{eq:vsimpexpr1}
 V(t,z) =  -\dot{\lambda}(t) \, z \,p(\lambda(t),z).
\end{equation}
On the other hand, by our definition, $\tilde{g}_{\lambda(t)} =M_{t} \circ g_t ,$  $t\in [0, T_{\max{}}).$ Since $g_t$ satisfies
\[
\begin{cases}
 \frac{\partial}{\partial t} g_t(z) = -({h^{-1}_{u_t}}_*\, b)(g_t(z)),\\
  g_0(z) = z,
\end{cases} z\in \mathbb{D}, t\geq 0, 
\]
where
\[
 b(z) = \alpha - z \,\left(i\beta + \frac{1+z}{1-z}\right) - \overline{\alpha}\, z^2, \quad \alpha \in \mathbb{C}, \quad \beta \in \mathbb{R},
\]
the chain rule implies
\begin{dmath}
 \label{eq:vsimpexpr2}
V(t,z) = - W(t,z) - {(M_t \circ {h^{-1}_{u_t}})_*}\, b \,(z) = W(t,z) + {(M_t \circ {h^{-1}_{u_t}})_*}\, \left(\alpha - z \,\left(i\beta + \frac{1+z}{1-z}\right) - \overline{\alpha}\, z^2\right) 
= -W(t,z) - {(M_t \circ {h^{-1}_{u_t}})_*}\, \left(\alpha - i \beta z - \overline{\alpha}\, z^2\right) - \frac{h_{u_t} \circ M^{-1}_t(z)}{(h_{u_t} \circ M^{-1}_t)'(z)} \,\frac{1+h_{u_t} \circ M^{-1}_t(z)}{1-h_{u_t} \circ M^{-1}_t(z)}.
\end{dmath}
The expressions \eqref{eq:vsimpexpr1} and \eqref{eq:vsimpexpr2} must be identically equal.

By Proposition~\ref{prop:slittransformation}, the vector field in \eqref{eq:vsimpexpr2} is of the form \eqref{eq:oftheform}: it has a simple pole at $M_t\circ h^{-1}_{u_t}\,(1)$ and is tangent on the rest of the unit circle for each $t\geq 0$. This implies that the function $p(\tilde{t}, z)$ in \eqref{eq:vsimpexpr1} must be of the form
\[
 p(\tilde{t}, z) = \frac{e^{i\tilde{u}_{\tilde{t}}}+z}{e^{i\tilde{u}_{\tilde{t}}}-z},
\]
where
\[
 e^{i \tilde{u}_{\lambda(t)}} = M_t \circ h^{-1}_{u_t}(1), 
\]
(compare with Proposition \ref{prop:slitFields}).

The residues at $e^{i \tilde{u}_{\lambda(t)}}$ in \eqref{eq:vsimpexpr1} and \eqref{eq:vsimpexpr2} must coincide. From \eqref{eq:vsimpexpr1},
\[
 \res_{z = e^{i \tilde{u}_{\lambda(t)}}} V(t,z) =  2 \,\dot{\lambda}(t)  e^{2 i \tilde{u}_{\lambda(t)}} =  2 \,\dot{\lambda}(t)  \left(M_t \circ h^{-1}_{u_t}(1)\right)^2,
\]
On the other hand, from \eqref{eq:vsimpexpr2},
\begin{dmath}
 \res_{z = e^{i \tilde{u}_{\lambda(t)}}} V(t,z) =  2 \, \frac{1}{\left((h_{u_t} \circ M^{-1}_t)'(e^{i \tilde{u}_{\lambda(t)}})\right)^2}
= 2 \, \frac{1}{\left((h_{u_t} \circ M^{-1}_t)'( (M_t \circ h^{-1}_{u_t})(1))\right)^2} 
= 2 \,( (M_t \circ h^{-1}_{u_t})'(1))^2,
\end{dmath}
so that
\[
 \dot{\lambda}(t) = \left(\frac{\left(M_t \circ h^{-1}_{u_t}\right)'(1)}{\left(M_t \circ h^{-1}_{u_t}\right)(1)}\right)^2.
\]

Thus, $\{\tilde{g}_{\tilde{t}}\}_{\tilde{t} \in [0, \tilde{T}_\max)}$ is a radial Löwner chain driven by the continuous function $\tilde{u}_{\tilde{t}}$. 
\end{proof}

\begin{remark*}
Let $z^0 \in \mathbb{D}$ and let $\phi$ be a conformal automorphism of $\mathbb{D}$ such that $\phi(1) = 1, $ and $\phi(0) = z^0$. Then the vector fields $b^0 = \phi_* ( - z\,\frac{1 + z}{1-z})$ and $\sigma^{0} = \phi_* (- i z)$ correspond to a version of the radial evolution for which the solutions leave the point $z^0$ fixed.

We can generalize Theorem \ref{th:toradial} by showing that for $t\in [0, T^0_{\max{}}),$ with $T^0_{\max{}} := \sup\{t: z^0 \not\in K_t \},$ there exists a radial Löwner chain leaving $z^0$ fixed, that describes the same evolution of hulls as the original slit Löwner chain. The formulas \eqref{eq:lambdaexpr} and \eqref{eq:utexpression} become slightly more complicated in this case.
\end{remark*}

\subsection{Correspondence between two general chains}

Let $\{K_t\}_{t\geq 0}$ be the family of hulls generated by a slit Löwner chain driven by $b,$ $\sigma$ and $u_t$. Given another pair of vector fields, $\tilde{b}$ and $\tilde{\sigma}$, can one find a suitable driving function $\tilde{u}_{\tilde{t}}$, so that the corresponding slit Löwner chain generates the same family of hulls? The following theorem, which is a generalization of Theorem \ref{th:toradial}, states that this is possible, at least for small values of $t.$

The construction is implicit and relies on existence theorems for solutions to systems of ordinary differential equations.

\begin{theorem}
\label{th:main}
 Let $\{g_t\}_{t\geq 0}$ be a normalized slit L\"owner chain in a simply connected hyperbolic domain $D$ driven by $b,$ $\sigma$ and $u_t$. For any other pair of vector fields $(\tilde{b}, \tilde{\sigma}),$ where $\tilde{b}$ and $\tilde{\sigma}$ are normalized as in \eqref{eq:bnormalized}, \eqref{eq:sigmanormalized},
and for some $T_{\max{}}>0,$ there exist
\begin{itemize}
 \item a unique $C^1$-differentiable time reparameterization $\lambda :[0,\,T_{\max{}}) \to [0, \,\tilde{T}_{\max{}}),$ and
\item a unique continuous driving function $\tilde{u}_{\tilde{t}}:[0,\,\tilde{T}_{\max{}}) \to \mathbb{R}$,
\end{itemize}
such that the slit L\"owner chain  $\{\tilde{g}_{\tilde{t}}\}_{\tilde{t} \in[0,\,\tilde{T}_{\max{}})}$ driven by $\tilde{b},$ $\tilde{\sigma}$ and $\tilde{u}_{\tilde{t}}$, describes the same evolution of domains as $\{g_t\}_{t\in [0,\,T_{\max{}})},$ that is,
\[
 \tilde{g}^{-1}_{\lambda(t)} (D) = g^{-1}_t(D), \quad \textrm{ for all } t\in [0,\,T_{\max{}}).
\]

Moreover, $\dot{\lambda}(t) >0$ for all $\tinrange$.
\end{theorem}

\begin{proof}
Without loss of generality, we assume $D = \mathbb{D}.$

We use the following notation in this proof. The coefficients of the semicomplete fields $b(z)$ and $\tilde{b}(z)$ are denoted by $\alpha,$ $\beta$ and $\tilde{\alpha},$ $\tilde{\beta},$  respectively, so that
\[
 b(z) = \alpha - z \,\left(i\beta +  \frac{1+z}{1-z}\right) - \overline{\alpha}\, z^2, \quad \alpha \in \mathbb{C}, \quad \beta \in \mathbb{R},
\]
and
\[
 \tilde{b}(z) = \tilde{\alpha} - z \,\left(i\tilde{\beta} +  \frac{1+z}{1-z}\right) - \overline{\tilde{\alpha}}\, z^2, \quad \tilde{\alpha} \in \mathbb{C}, \quad \tilde{\beta} \in \mathbb{R}.
\]
The flows of disk automorphisms corresponding to the fields $\sigma(z)$ and $\tilde{\sigma}(z)$ are denoted by $h_t$ and $\tilde{h}_t,$ and the parameters of these M\"obius transformations by $\theta_t,$ $a_t,$ $\tilde{\theta}_t,$ $\tilde{a}_t,$ so that
\[
 h_t = e^{i \theta_t} \,\frac{z-a_t}{1 - \overline{a}_t z}, \quad \tilde{h}_t = e^{i \tilde{\theta}_t} \,\frac{z-\tilde{a}_t}{1 - \overline{\tilde{a}}_t z}, \quad \theta_t, \tilde{\theta}_t \in \mathbb{R},\quad a_t, \tilde{a}_t \in \mathbb{D},
\]
and the inverse maps are given by
\[
 h^{-1}_t = e^{-i \theta_t} \,\frac{z+a_t e^{i\theta_t}}{1 + \overline{a_t e^{i\theta_t}} z}, \quad  \tilde{h}^{-1}_t = e^{-i \tilde{\theta}_t} \,\frac{z+\tilde{a}_t e^{i\tilde{\theta}_t}}{1 + \overline{\tilde{a}_t e^{i\tilde{\theta}_t}} z}.
\]

In this proof we say that a family of functions $\{M_t\}_{t\in[0,T_{\max{}})}$ is a \emph{$t$-differentiable automorphic family}, if $\{M_t\}_{t\in[0,T_{\max{}})} \subset \mathrm{Aut}(\mathbb{D})$ and for every fixed $z_0\in \mathbb{D},$ $M_t(z_0) \in C^1[0,T_{\max{}}).$

We organize the proof as a sequence of three claims.

\bigskip

\textbf{Claim 1}. Let $\lambda \in C^1[0,T_{\max{}})$ and $\tilde{u}_{\tilde{t}}\in C^0[0,\tilde{T}_{\max{}})$. The pair $(\lambda(t),\tilde{u}_{\tilde{t}})$  possesses the required properties if and only if there exists a $t$-differentiable automorphic family $\{M_t\}_{t\in [0,T_{\max{}})}$ such that
\begin{equation}
\begin{cases}
\label{eq:Wequation}
 \frac{\partial}{\partial t} M_t(z) = V(t, M_t(z)), \quad t \in [0,T_{\max{}}),\\
 M_0(z) = z,
\end{cases} 
\end{equation}
where
\[
 V(t,z):=  -\dot{\lambda}(t)\cdot \left( \tilde{h}^{-1}_{\tilde{u}_{\lambda(t)}*} \tilde{b}\right)(z) + \left(M_{t*} {h^{-1}_{u_t}}_*\, b\right) (z).
\]

\bigskip

\textbf{Claim 2}. Let $\lambda\in C^1[0,T_{\max{}})$, $\tilde{u}_{\tilde{t}}\in C^0[0,\tilde{T}_{\max{}})$, and $\{M_t\}_{t\in[0,T_{\max{}})}$ be a $t$-differentiable automorphic family. The triple $(\lambda(t), \tilde{u}_{\tilde{t}}, \{M_t\}_{t\in[0,T_{\max{}})})$ satisfies (\ref{eq:Wequation}) if and only if 
\begin{itemize}
 \item $\{M_t\}_{t\in [0,T_{\max{}})}$ is a $t$-differentiable automorphic family satisfying
\begin{equation}
\label{eq:replacedEquation}
 \begin{cases}
 \frac{\partial}{\partial t} M_t(z) = \tilde{V}(t, M_t(z)), \\
 M_0(z) = z,
\end{cases}  t\in [0,T_{\max{}}), \quad z\in \mathbb{D},
\end{equation}
where
\begin{dmath*}
\tilde{V}(t,z) {:=} - \left(\left(\tilde{h}_{-T^{\tilde{\sigma}} \circ M_t \circ h^{-1}_{u_t} (1)}\circ M_t \circ h^{-1}_{u_t}\right)'(1) \right)^2 \cdot  \left(\tilde{h}^{-1}_{-T^{\tilde{\sigma}} \circ M_t \circ h^{-1}_{u_t}\,(1)}{}_* \tilde{b} \right)(z) + \left({M_t}_{*} h^{-1}_{u_t}{}_* b\right)(z)
\end{dmath*}
where $T^{\tilde{\sigma}}$ is the function defined in (\ref{eq:tdefinition}), Appendix~\ref{app:tsigma}.

\item  $\lambda(t)$ is given by

\begin{equation}
\label{eq:ttildedef}
 \dot{\lambda}(t) = \left(\left(\tilde{h}_{-T^{\tilde{\sigma}} \circ M_t \circ h^{-1}_{u_t} (1)}\circ M_t \circ h^{-1}_{u_t}\right)'(1) \right)^2 \quad \textrm{for all }t\in [0,T_{\max{}}),
\end{equation}
 \item$\tilde{u}_{\tilde{t}}$  is defined from the equation 
\[
 \tilde{h}_{\utilde} \circ M_t \circ h^{-1}_{u_t} \,(1) = 1,
\]
and hence is given by
\[
\utilde = - T^{\tilde{\sigma}}  \circ M_t \circ h^{-1}_{u_t}(1)  \quad \textrm{ for }t\in [0,T_{\max{}}).
\]
\end{itemize}

\bigskip

\textbf{Claim 3.} The initial-value problem (\ref{eq:replacedEquation}) has a unique $t$-differentiable automorphic solution $\{M_t\}_{t\in[0,T_{\max{}})},$ for some $T_{\max{}} >0$.

\bigskip

Once these claims are verified, the statement of the theorem follows immediately. Indeed, since the solution to (\ref{eq:replacedEquation}) exists and is unique, there exists a unique triple $(\lambda(t), \tilde{u}_{\tilde{t}}, \{M_t\}_{t\in[0,T_{\max{}})})$ satisfying (\ref{eq:Wequation}), and hence, there exists a unique time reparameterization $\lambda(t)$ and a a driving function $\tilde{u}_{\tilde{t}}$ describing the same evolution of domains as $\{g_t\}_{t\in [0,T_{\max{}})}$, but for the coefficients $\tilde{b}$ and $\tilde{\sigma}$. The inequality $\dot{\lambda}(t) >0$ follows directly from (\ref{eq:ttildedef}) because the derivatives of M\"obius automorphisms $M_t \circ h^{-1}_{u_t}$ and $\tilde{h}^{-1'}_{\utilde}$ remain non-zero.

\bigskip

\textbf{Proof of Claim 1}. Suppose there exist such $\lambda(t)\in C^1[0,T_{\max{}})$ and $\tilde{u}_{\tilde{t}}\in C^{0}[0,\tilde{T}_{\max{}}),$ so that $\tilde{g}^{-1}_{\lambda(t)} (\mathbb{D}) = g^{-1}_t(\mathbb{D})$ for all $t\in [0,\,T_{\max{}})$. Define
\[
M_t(z) := \tilde{g}_{\lambda(t)} (g^{-1}_t(z)), \quad  t\in [0,\,T_{\max{}}),
\]
and note that $\autfam$ is a $t$-differentiable automorphic family. Applying the chain rule, we see that $\autfam$ satisfies (\ref{eq:Wequation}).

In other words, if such functions $\lambda(t)$ and $\tilde{u}_{\tilde{t}}$ exist, then there exists at least one family $\{M_t\}_{t\in[0,T_{\max{}})}\subset\mathrm{Aut}(\mathbb{D})$ satisfying (\ref{eq:Wequation}).

Conversely, suppose for given $\lambda(t)$ and $\tilde{u}_{\tilde{t}}$ the problem (\ref{eq:Wequation}) has a $t$-differentiable automorphic solution $\autfam$. Then the family of functions $\{\tilde{g}_{\tilde{t}}\}_{\tilde{t} \in [0,\,\tilde{T}_{\max{}})}$ defined by
\[
 \tilde{g}_{\lambda(t)} = M_t(g_t(z)), 
\]
is a $(\tilde{b}, \tilde{\sigma})$-flow driven by $\tilde{u}_{\tilde{t}}$ describing the same evolution of domains as $\{g_t\}_\tinrange$.

\bigskip

\textbf{Proof of Claim 2.} Let $(\lambda(t), \tilde{u}_{\tilde{t}},\autfam)$ be a triple satisfying (\ref{eq:Wequation}). Consider the time-dependent vector field
\begin{equation}
\label{eq:Vdef}
V(t,z) = -\dot{\lambda}(t)\cdot \left( \tilde{h}^{-1}_{\tilde{u}_{\lambda(t)}*} \tilde{b}\right)(z) + \left(M_{t*} {h^{-1}_{u_t}}_*\, b\right) (z),
\end{equation}
and the related initial-value problem
\begin{equation}
\label{eq:Vequation}
\begin{cases}
 \frac{\partial}{\partial t} \phi_t(z) =V(t, \phi_t(z)), \\
\phi_0(z) = z,
\end{cases}\quad t\in [0,T_{\max{}}), \, z\in \mathbb{D}. 
\end{equation}

Trivially, the family $\autfam$ also satisfies (\ref{eq:Vequation}), and hence, $V(t,z)$ is the velocity field of an automorphic evolution family. Thus, $V(t,z)$ is a complete field for every fixed $t\in[0,T_{\max{}})$. In particular, it can be represented in the form
\[
 V(t,z) = \alpha_t - i \beta_t\, z -  \overline{\alpha_t}\, z^2, \quad \alpha_t\in \mathbb{C}, \, \beta_t \in \mathbb{R}, \, t\in[0,T_{\max{}}).
\]

Write
\[
 M_t(z) = e^{i\Theta_t} \frac{z - A_t}{1 - \overline{A_t} \, z}.
\]

We use Proposition~\ref{prop:slittransformation} to rewrite the first summand in the right-hand side of (\ref{eq:Vdef}) in the form
\[
 -\dot{\lambda}(t)\cdot \left( \tilde{h}^{-1}_{\tilde{u}_{\lambda(t)}*} \tilde{b}\right)(z) = -\tilde{\alpha}^{\star}_t + z \left( i \tilde{\beta}^{\star}_t + \tilde{\gamma}^{\star}_t \, \frac{\tilde{h}^{-1}_{\tilde{u}_{\lambda(t)}}(1)+z}{\tilde{h}^{-1}_{\tilde{u}_{\lambda(t)}}(1)-z}\right) + \overline{\tilde{\alpha}^{\star}_t} z^2,
\]
where
\begin{dmath}
\label{eq:alphatildeprime}
\tilde{\alpha}^{\star}_t  = \dot{\lambda}(t)\, \frac{e^{-i \tilde{\theta}_{\tilde{u}_{\lambda(t)}}}}{1-|\tilde{a}_{\tilde{u}_{\lambda(t)}}|^2} \,\tilde{b} \left(-\tilde{a}_{\tilde{u}_{\lambda(t)}} e^{i\tilde{\theta}_{\tilde{u}_{\lambda(t)}}}\right), 
\end{dmath}
\begin{multline}
\label{eq:betatildeprime}
 \tilde{\beta}^{\star}_t = \dot{\lambda}(t) \frac{1}{1 -\left|\atilde\right|^2} \Bigg{(}-4 \im \left(\atilde \etilde \overline{\tilde{\alpha}}\right) + {\tilde{\beta} \left(1 + |\atilde|^2\right)}
 \\- \left.4  \im \left(\atilde \etilde \right) \cdot\left(\frac{1 + |\atilde|^2}{|\atilde \etilde + 1|^4} \re \left( \atilde \etilde\right)
 + \frac{1 + |\atilde|^4}{|\atilde \etilde +1|^4}\right) \right),
\end{multline}
 
\[
\tilde{\gamma}^{\star}_t = \dot{\lambda}(t)\,\left| {\tilde{h}^{-1'}_{\tilde{u}_{\lambda(t)}}}(1)\right|^2.
\]

In a similar way, we can rewrite the second summand in (\ref{eq:Vdef}). First, note that

\[
 M_t (h^{-1}_{u_t}(z)) = e^{i (\Theta_t - \theta_{u_t})} \frac{z -\hutat }{1 - z\, \overline{\hutat}}.
\]
Then,
\[
 \left(M_{t*} {h^{-1}_{u_t}}_*\, b\right) (z) = \alpha^{\star}_t  - z\, \left(i\beta^{\star}_t + \gamma^{\star}_t \,\frac{M_t \circ h^{-1}_{u_t}(1) + z}{M_t\circ h^{-1}_{u_t}(1)  - z}\right) -\overline{\alpha^{\star}_t} \,z^2,
\]
where
\begin{equation}
\label{eq:alphaprime}
 \alpha^{\star}_t = \frac{e^{i (\Theta_t - \theta_{u_t})}}{1- \left|h_{u_t} (A_t)\right|^2} \,b(h_{u_t} (A_t)),
\end{equation}

\begin{dmath}
\label{eq:betaprime}
 \beta^{\star}_t = \frac{1}{1 - \left|\hutat\right|^2} \,\left( 4 \im(\hutat \, \overline{\alpha}) + {\beta \, \left(1 + |\hutat|^2\right)} 
  + 4 \,\frac{\im \hutat}{\left|\hutat-1\right|^4} \left(1 + |\hutat|^4 - \left(1 + \left|\hutat\right|^2\right)\, \re \hutat \right)\right), 
\end{dmath}

\[
\gamma^{\star}_t = \left|\left(M_t \circ h^{-1}_{u_t}\right)'(1) \right|^2.
\]

Thus, for the sum  (\ref{eq:Vdef}) to be a complete vector field (in particular, a polynomial of degree 2) for each $t \in[0,T_{\max{}}),$ it is necessary and sufficient that the fractions cancel out,
\[
 \tilde{\gamma}^{\star}_t \, \frac{\tilde{h}^{-1}_{\tilde{u}_{\lambda(t)}}(1)+z}{\tilde{h}^{-1}_{\tilde{u}_{\lambda(t)}}(1)-z} \equiv \gamma^{\star}_t \,\frac{M_t \circ h^{-1}_{u_t}(1) + z}{M_t\circ h^{-1}_{u_t}(1)  - z} \quad \textrm{for all }t\geq 0 \textrm{ and all }z\in \mathbb{D},
\]
which is possible if and only if
\begin{equation}
\label{eq:firstofthese}
 \tilde{h}^{-1}_{\tilde{u}_{\lambda(t)}}(1) = M_t \circ h^{-1}_{u_t}(1) \quad \textrm{ for all }t\in [0,T_{\max{}}),
\end{equation}
and
\begin{equation}
\label{eq:secondofthese}
  \dot{\lambda}(t)\,\left| {\tilde{h}^{-1'}_{\tilde{u}_{\lambda(t)}}}(1)\right|^2 =  \left|\left(M_t \circ h^{-1}_{u_t}\right)'(1) \right|^2  \quad \textrm{ for all }t\in [0,T_{\max{}}).
\end{equation}

Equation (\ref{eq:firstofthese}) implies that
\[
\utilde = - T^{\tilde{\sigma}} \circ M_t \circ h^{-1}_{u_t}(1)  \quad \textrm{ for }t\in [0,T_{\max{}}),
\]
and then (\ref{eq:secondofthese}) leads to
\[
  \dot{\lambda}(t) = \left|\left(\tilde{h}_{-T^{\tilde{\sigma}} \circ M_t \circ h^{-1}_{u_t} (1)}\circ M_t \circ h^{-1}_{u_t}\right)'(1) \right|^2 \quad \textrm{for all }t\in [0,T_{\max{}}).
\]
Note that $1$ is a fixed boundary point of the M\"obius transformation $\tilde{h}_{-T^{\tilde{\sigma}} \circ M_t \circ h^{-1}_{u_t} (1)}\circ M_t \circ h^{-1}_{u_t},$ hence the derivative $\left(\tilde{h}_{-T^{\tilde{\sigma}} \circ M_t \circ h^{-1}_{u_t} (1)}\circ M_t \circ h^{-1}_{u_t}\right)'(1)$ is a positive real number, and in fact we do not need the absolute value sign:

\[
  \dot{\lambda}(t) = \left(\left(\tilde{h}_{-T^{\tilde{\sigma}} \circ M_t \circ h^{-1}_{u_t} (1)}\circ M_t \circ h^{-1}_{u_t}\right)'(1) \right)^2 \quad \textrm{for all }t\in [0,T_{\max{}}).
\]

Proving sufficiency is trivial.

\bigskip

\textbf{Proof of Claim 3.} Calculations in the proof of Claim 2 show that for a $t$-differentiable automorphic family $\autfam,$
with
\[
 M_t(z) = e^{i\Theta_t} \frac{z - A_t}{1 - \overline{A_t} \, z},
\]
the vector field 
\begin{dmath*}
V(t,z) = -\dot{\lambda}(t)\cdot \left( \tilde{h}^{-1}_{\tilde{u}_{\tilde{t}}*} \tilde{b}\right)(z) + \left(M_{t*} {h^{-1}_{u_t}}_*\, b\right) (z),
\end{dmath*}
with 
$\tilde{u}_{\tilde{t}}= - T^{\tilde{\sigma}}  (M_t(h^{-1}_{u_t}(1)))$ and $\dot{\lambda}(t) = \left(\left(\tilde{h}_{-T^{\tilde{\sigma}} \circ M_t \circ h^{-1}_{u_t} (1)}\circ M_t \circ h^{-1}_{u_t}\right)'(1) \right)^2$, is complete for each $\tinrange$, and may be written as
\begin{equation}
\label{eq:fundeq1}
V(t,z) = \alpha^{\star\star}_t - i\beta^{\star\star}_t z - \overline{\alpha^{\star\star}_t} z^2,\quad \tinrange, 
\end{equation}
where
\[
 \alpha^{\star\star}_t = - \tilde{\alpha}^{\star}_t + \alpha^{\star}_t,
\]
\[
 \beta^{\star\star}_t = - \tilde{\beta}^{\star}_t + \beta^{\star}_t 
\]
with  $\alpha^{\star}_t$,  $\beta^{\star}_t$ are defined in \eqref{eq:alphaprime} and \eqref{eq:betaprime}, and

\begin{dmath*}
  \tilde{\alpha}^{\star}_t= \frac{ \left|\left(M_t \circ h^{-1}_{u_t}\right)'(1) \right|^2 }{\left| {\tilde{h}^{-1'}_{\newutilde}}(1)\right|^2} \cdot \frac{e^{-i \tilde{\theta}_{\newutilde}}}{1-|\tilde{a}_{\newutilde}|^2} \,\tilde{b} \left(-\tilde{a}_{\newutilde} e^{i\tilde{\theta}_{\newutilde}}\right),
\end{dmath*}

\begin{dmath*}
 \tilde{\beta}^{\star}_t 
=  \frac{ \left|\left(M_t \circ h^{-1}_{u_t}\right)'(1) \right|^2 }{\left| {\tilde{h}^{-1'}_{\newutilde}}(1)\right|^2}\cdot \frac{1}{1 -\left|\newatilde\right|^2} \Bigg{(}-4 \im \left(\newatilde \newetilde \overline{\tilde{\alpha}}\right) \\ + \tilde{\beta} \left(1 + |\newatilde|^2\right)
 \left.-4 \im \bigg{(}\newatilde \newetilde \right) \times 
\\ \times\Bigg{(}\frac{1 + |\newatilde|^2}{|\newatilde \newetilde + 1|^4} \re \left( \newatilde \newetilde\right) 
\\+ \frac{1 + |\newatilde|^4}{|\newatilde \newetilde +1|^4}\Bigg{)} \Bigg{)}.
\end{dmath*}

On the other hand, $V(t,z)$ is the velocity field of the flow $\{M_t\}_{\tinrange}$. Indeed, if we set $\phi_t  = M_t,$ then 
\[
 \begin{cases}
  \frac{\partial}{\partial t} \phi_t(z) = V(t,\phi(z)),\\
  \phi_0(z) = z,
 \end{cases}\quad \tinrange, z \in \mathbb{D},
\]
is satisfied. Therefore, according to \eqref{eq:coeffrecalc}, $V(t,z)$ can be written as
\begin{dmath}
\label{eq:fundeq2}
 V(t,z) =- \frac{e^{i\Theta_{t}} \dot{A}_{t}}{1 - |A_{t}|^2} + i \left( \frac{2 \im A_{t} \dot{\bar{A}}_{t}}{1-|A_{t}|^2} + \dot{\Theta}_{t}\right) \,z + \frac{e^{-i\Theta_{t}} \dot{\bar{A}}_{t}}{1-|A_{t}|^2}\, z^2.
\end{dmath}

By equating the coefficients at $z$ in \eqref{eq:fundeq1} and \eqref{eq:fundeq2}, we arrive at the following initial value problem for a system of ordinary differential equations of first order for $\Theta_t$ and $A_t$

\begin{equation}
\label{eq:finalCauchy}
 \begin{cases}
  \dot{A}_t = - e^{-i \Theta_t} \,(1-|A_t|^2) \,\alpha^{\star\star}_t,\\
 \dot{\Theta}_t = \beta^{\star\star}_t +2 \im \left(e^{-i \Theta_t} A_t \,\alpha^{\star\star}_t\right),\\
A_0 = 0, \quad\Theta_0=0.
 \end{cases}
\end{equation}

The right-hand side is analytic with respect to $(A_t, \Theta_t)$ in a neighborhood  of $(0,0),$ and is continuous with respect to $t$. Therefore, there exists a unique solution  to the initial value problem in some interval $\tinrange$.
\end{proof}

\subsection{Uniqueness of parameterization}
The following theorem establishes the uniqueness of the driving function describing a given family of hulls for given vector fields $b$ and $\sigma$.

\begin{theorem}
\label{th:uniquenessparam}
 Let $\{g_t\}_{t\geq 0}$ and $\{\tilde{g}_{\tilde{t}}\}_{\tilde{t}\geq 0}$ be two general slit L\"owner chains in a simply connected hyperbolic domain $D$ that are driven by $b$, $\sigma$, $u_t$ and $b,$ $\sigma,$ $\tilde{u}_{\tilde{t}},$ respectively.  Let $\{g_t\}_{t\geq 0}$ and $\{\tilde{g}_{\tilde{t}}\}_{\tilde{t}\geq 0}$ describe the same evolution of domains, i.e., for some continuous monotone function $\lambda(t),$ 
\[
 \tilde{g}^{-1}_{\lambda(t)}(D) = g^{-1}_t (D), \quad \textrm{ for } t\in [0, +\infty).
\]
 Then $\lambda(t) \equiv t$, $\utilde \equiv u_t$ and $g_t \equiv \tilde{g}_{\lambda(t)}$ for $t \in [0,+\infty)$.
\end{theorem}
\begin{proof}
Without loss of generality, we assume $D = \mathbb{D}.$

Let us first show that the function $\lambda$ is necessarily continuously differentiable on $[0,+\infty)$.

For a $T >0,$ choose a point $z^0 \in \cap_{t \in [0,T]} g^{-1}_t(\mathbb{D})$. Let $(b^0, \sigma^0)$ be the vector fields corresponding to the radial L\"owner evolution that fixes $z^0$. By the remark after  Theorem~\ref{th:toradial}, we have two different continuously differentiable $(b^{0}, \sigma^0)$-reparameterizations of the evolution: one, which we denote by $\lambda_1(t)$, coming from $g_t,$ and the other, $\lambda_2(\tilde{t})$, coming from $\tilde{g}_{\tilde{t}}$. However, according to Theorem \ref{th:toradial}, the parameterization of a radial L\"owner chain is uniquely determined, that is,
\[
 \lambda_1(t) = \lambda_2(\lambda(t)), \quad t \in [0, T].
\]
By Theorem~\ref{th:toradial}, $\dot{\lambda}_t(t)>0$ for all $t\geq \in [0, T]$ therefore, the inverse function $\lambda^{-1}_2$ exists and is continuously differentiable, and $\tilde{t}$ may be expressed as
\[
\tilde{t}= \lambda(t) = \lambda_2^{-1}(\lambda_1(t)),
\]
hence, $\lambda(t)$ is continuously differentiable on $[0,T]$ for any $T >0.$

Now, we can simply repeat the proof of Theorem~\ref{th:main} step-by-step, taking into account that in this case $b = \tilde{b}, $ and $\sigma  = \tilde{\sigma}$. When we finally arrive at the analogue of the initial value problem (\ref{eq:finalCauchy}), we notice that the right-hand side is analytic in $(A_t,\Theta_t)$ for all $(A_t, \Theta_t) \in \mathbb{D} \times \mathbb{R}.$ Since the dependence on $t$ is continuous, the right-hand side satisfies the Lipschitz condition in a neighborhood of $(0,0),$ for $t\in [0, T],$ where $T$ is an arbitrary positive number.  Therefore, if a solution to the initial value problem exists on $[0,T]$, it is unique there. It is easy to see that the unique solution is given by $A_t\equiv 0, \Theta_t \equiv 0,$  hence $\lambda(t) \equiv t,$ $\tilde{u}_{\lambda(t)} \equiv u_t$, $t\in[0, +\infty)$.
\end{proof}

We say that a function $\phi$ is \emph{embeddable} into a slit Löwner chain driven by $b$ and $\sigma$ if there exists a continuous function $u_t,$ $u_0=0,$ such that for the chain $\{g_t\}_{t\geq 0}$ driven by $b,$ $\sigma$ and $u_t,$ $\phi=g_{t_0}$ for some $t_0\geq 0$.
\begin{corollary}
\label{cor:uniqueness}
Let $\gamma:[0,T] \to \mathbb{D}$ be a simple curve, such that $\gamma(0,T]\subset\mathbb{D},$ $\gamma(0) = 1$. For a given pair of vector fields $b$ and $\sigma$ there exists at most one function $g_{t_0}$ embeddable into a slit Löwner  chain $\{g_t\}_{t\geq 0}$ driven by $b$ and $\sigma$, such that $g^{-1}_{t_0}(\mathbb{D}) = \mathbb{D}\setminus \gamma$.
\end{corollary}

\section{Geometry of the hulls and the domain Markov property}\label{sec:geometry}
In this section we investigate geometric properties of  hulls generated by slit Löwner chains and show that the geometry in the general case is closely related to the cases of classical Löwner equations and $SLEs$. In particular, we show that for sufficiently regular driving functions the hulls are quasislit curves. We also show that for the driving function $u_t = \sqrt{\kappa}\,B_t$ the hulls are generated by curves, which justifies the terms \emph{slit Löwner chain} and \emph{slit holomorphic stochastic flow}.

\subsection{Deterministic case}

Theorem~\ref{th:toradial} states that a normalized slit Löwner chain $\{g_t\}_{t\geq 0}$, driven by $b,$ $\sigma$ and $u_t$ can be related to a radial Löwner chain up to time $T_{\max{}}.$  The radial driving function $\utildenot$ can be expressed in terms of the original function $u_t$ by the formula

\begin{equation}
\label{eq:utildexpr}
 \tilde{u}_{\tilde{t}} = - i \log \left(M_{\lambda^{-1}(\tilde{t})} \circ h^{-1}_{u_{\lambda^{-1}(\tilde{t})}}\right) (1),\quad t\in[0,\tilde{T}_{\max{}}),
\end{equation}
where $M_t$ is defined as
 \[
M_t (z):= \frac{|g'_{t}(0)|}{g'_{t}(0)} \, \frac{z - g_{t}(0)}{1-\overline{g_{t}(0)} \, z},\quad t \in [0,T_{\max{}}),
\]
and the radial time reparameterization $\tilde{t} =\lambda(t)$ satisfies
\[
 \dot{\lambda}(t) = \left(\frac{\left(M_t \circ h^{-1}_{u_t}\right)'(1)}{\left(M_t \circ h^{-1}_{u_t}\right)(1)}\right)^2, \quad t\in[0,T_{\max{}}).
\]

Regarding formally $u_{\lambda^{-1}(\tilde{t})}$ as an independent variable, we may write
\[
 \frac{\partial \tilde{u}_{\tilde{t}}}{\partial \tilde{t}} =  -i\,e^{-i\tilde{u}_{\tilde{t}}}\,\dot{\lambda}^{-1}(\tilde{t})\,W(\lambda^{-1}(\tilde{t}), e^{i\tilde{u}_{\tilde{t}}}), 
\]
and
\begin{dmath*}
 \frac{\partial \tilde{u}_{\tilde{t}}}{\partial u_{\lambda^{-1}(\tilde{t})}} = -i \,\frac{M'_{\lambda^{-1}(\tilde{t})} \left(h^{-1}_{u_{\lambda^{-1}(\tilde{t})}}(1)\right) \,\left(\frac{\partial}{\partial u_{\lambda^{-1}(\tilde{t})}}h^{-1}_{u_{\lambda^{-1}(\tilde{t})}}(1)\right)}{(M_{\lambda^{-1}(\tilde{t})} \circ h^{-1}_{u_{\lambda^{-1}(\tilde{t})}})(1)}
= -i \,\frac{M'_{\lambda^{-1}(\tilde{t})} \left(h^{-1}_{u_{\lambda^{-1}(\tilde{t})}}(1)\right) \,\left( - \left(h^{-1}_{u_{\lambda^{-1}(\tilde{t})}}\right)'(1)\,\sigma(1)\right)}{(M_{\lambda^{-1}(\tilde{t})} \circ h^{-1}_{u_{\lambda^{-1}(\tilde{t})}})(1)}
=\frac{\left(M_{\lambda^{-1}(\tilde{t})} \circ h^{-1}_{u_{\lambda^{-1}(\tilde{t})}}\right)'(1)}{\left(M_{\lambda^{-1}(\tilde{t})} \circ h^{-1}_{u_{\lambda^{-1}(\tilde{t})}}\right)(1)} = \sqrt{\dot{\lambda}(\lambda^{-1}(\tilde{t}))}
\end{dmath*}
(here we use the normalization condition $\sigma(1) = -i$ and \eqref{eq:inverseflowpde} from Appendix ~\ref{app:inverseflows}).

Let $s, t \in [0, T_{\max{}}),$ $\tilde{s} = \lambda(s),$ $\tilde{t} = \lambda(t),$ $\Delta t= t-s,$ $\Delta \tilde{t} = \tilde{t} - \tilde{s},$ $\Delta u_t = u_t-u_s,$ and $\Delta \tilde{u}_{\tilde{t}} = \tilde{u}_{\tilde{t}} -\tilde{u}_{\tilde{s}}$. Let $W(t,z)$ denote the velocity field of $M_t(z)$. Then we can apply a multivariate version of Lagrange's mean value theorem , using the expressions for partial derivatives above, and write
\begin{dmath}
\label{eq:lagrange}
 \Delta \tilde{u}_{\tilde{t}} = -i e^{-i \tilde{u}_{\tilde{\theta}_1}} \dot{\lambda}^{-1}(\tilde{\theta}_1) \, W(\lambda^{-1}(\tilde{\theta}_1), e^{i\tilde{u}_{\tilde{\theta}_1}})\, \Delta \tilde{t} 
+ \sqrt{\dot{\lambda}(\lambda^{-1}(\tilde{\theta}_2))}\, \Delta u_t.
=\frac{W(\theta_1, e^{i\tilde{u}_{\tilde{\theta}_1}}) }{i\, e^{i \tilde{u}_{\tilde{\theta}_1}} \dot{\lambda}(\theta_1)} \,\Delta \tilde{t} + \sqrt{\dot{\lambda}(\theta_2)} \, \Delta u_t,
\end{dmath}
for some points $\theta_1, \theta_2$ lying between $s$ and $t,$ and $\tilde{\theta}_1 = \lambda(\theta_1),$ and $\tilde{\theta}_2 = \lambda(\theta_2)$.

Recall that a \emph{quasiarc} is the image of $[0,\infty)$ under a quasiconformal homeomorphism of $\mathbb{C}$.  Using the preliminary calculations above we can now prove in the general case sufficiently regular driving functions generate quasiarcs.
\begin{proposition}
Let $\{g_t\}_{t\geq 0}$ be a normalized slit Löwner chain in a simply connected hyperbolic domain $D$, driven by $b,$ $\sigma$ and $u_t$. Let $T>0$. If
\[
 \inf_{\epsilon >0} \mathop{\sup_{s,t\in[0,T]}}_{|t-s| < \epsilon} \frac{|u_t-u_s|}{\sqrt{|t-s|}} < 4,
\]
then $K_T = D\setminus g^{-1}_t(D)$ is a quasiarc.
\end{proposition}
\begin{proof}
Without loss of generality we assume that $D=\mathbb{D}$, and that $0 \not \in K_T$, so that $T < T_{\max{}},$  where $T_{\max{}}$ is the time defined in Theorem \ref{th:toradial}.

Let $\epsilon >0$, $0 \leq  s < t \leq T,$ $\tilde{s} = \lambda(s), $ $\tilde{t} = \lambda(t),$ and $\tilde{t} - \tilde{s} < \epsilon$.

Dividing \eqref{eq:lagrange} by $\sqrt{ \tilde{t}-\tilde{s}}$ yields

\begin{dmath*}
 \frac{|\tilde{u}_{\tilde{t}} -\tilde{u}_{\tilde{s}}|}{\sqrt{\tilde{t}-\tilde{s}}}
=\left|\frac{W(\theta_1, e^{i\tilde{u}_{\tilde{\theta}_1}}) }{i\, e^{i \tilde{\theta}_1} \dot{\lambda}(\theta_1)} \,\sqrt{\tilde{t}-\tilde{s}} + \sqrt{\dot{\lambda}(\theta_2)} \, \frac{u_t-u_s}{\sqrt{t-s}} \frac{1}{\sqrt{\frac{\tilde{t}-\tilde{s}}{t-s}}}\right|
\\=\left|\frac{W(\theta_1, e^{i\tilde{u}_{\tilde{\theta}_1}}) }{i\, e^{i \tilde{\theta}_1} \dot{\lambda}(\theta_1)} \,\sqrt{\tilde{t}-\tilde{s}} 
+ \sqrt{\dot{\lambda}(\theta_2)} \, \frac{u_t-u_s}{\sqrt{t-s}} \frac{1}{\sqrt{\dot{\lambda}(\theta_3)}}\right|
\end{dmath*}
for some $\theta_1, \theta_2,\theta_3\in (s,t),$ so that 
\begin{dmath*}
 \frac{|\tilde{u}_{\tilde{t}} -\tilde{u}_{\tilde{s}}|}{\sqrt{\tilde{t}-\tilde{s}}} 
\leq \frac{\max_{\theta \in [s, t]}\left| W(\theta, e^{i\tilde{u}_{\lambda(\theta)}})\right|}{\min_{\theta\in [s, t]} \sqrt{\dot{\lambda}(\theta)}}\,\sqrt{\epsilon}
+\frac{|u_t-u_s|}{\sqrt{t-s}} \frac{\max_{\theta\in [s, t]} \sqrt{\dot{\lambda}(\theta)}}{\min_{\theta\in [s, t]} \sqrt{\dot{\lambda}(\theta)}}.
\end{dmath*}

Due to the uniform continuity of $\dot{\lambda}(\theta)$ on $[0,T]$,
 $\sup_{|t-s| < \epsilon} \frac{\max_{\theta\in [s, t]} \sqrt{\dot{\lambda}(\theta)}}{\min_{\theta\in [s, t]} \sqrt{\dot{\lambda}(\theta)}}$ decreases monotonically to 1, as $\epsilon \to 0$, so that

\[
\inf_{\epsilon >0} \mathop{\sup_{\tilde{s},\tilde{t}\in[0,\tilde{T}]}}_{|\tilde{t}-\tilde{s}| < \epsilon} \frac{|\tilde{u}_{\tilde{t}}-\tilde{u}_{\tilde{s}}|}{\sqrt{|\tilde{t}-\tilde{s}|}} \leq \inf_{\epsilon >0} \mathop{\sup_{s,t\in[0,T]}}_{|t-s| < \epsilon} \frac{|u_t-u_s|}{\sqrt{|t-s|}} < 4,
\]
and by \cite[Theorem 1.1]{MarshallRohde}, $K_T$ is a quasiarc.
\end{proof}

\begin{proposition} \label{prop:angle}
 Let  $\{g_t\}_{t\geq 0}$ be a general normalized slit Löwner chain in $\mathbb{D},$ driven by $b$, $\sigma$ and $u_t$. Let the limit 
\[
 \lim_{t \to s+} \frac{u_t - u_s}{\sqrt{t-s}}
\]
exist for some $s\geq 0$.  Let  $\lambda$ and $\tilde{u}$ be the radial time reparameterization and the radial driving function defined in Theorem \ref{th:toradial}, respectively. Let $\tilde{s} = \lambda(s)$. Then 
\[
 \lim_{\tilde{t} \to \tilde{s}+} \frac{\utildenot - \tilde{u}_{\tilde{s}}}{\sqrt{\tilde{t}-\tilde{s}}} =  \,  \, \lim_{t \to s+} \frac{u_t - u_s}{\sqrt{t-s}}.
\]
\end{proposition}

\begin{proof}
Without loss of generality, we assume $0 \in g^{-1}_s(\mathbb{D})$, so that $s < T_{\max{}}$. Then we divide \eqref{eq:lagrange} by $\sqrt{\Delta \tilde{t}}, $ and let $\Delta \tilde{t} \to 0+$.
\end{proof}

\begin{corollary}
 Let $\gamma$ be a slit generated by a normalized slit L\"owner chain in $\mathbb{H}$, driven by $b,$ $\sigma$ and $u_t$. Let $\gamma(t)$ be  continuously differentiable on $(0,T]$, and $\dot{\gamma}(t)\neq 0$ for all  $t\in (0,T]$. Suppose there is $\theta\in (0,\pi)$ such that
\[
 \lim_{t\to 0+} \arg \dot{\gamma}(t)   = \theta.
\]
Then, 
\[
 \lim_{t\to 0+}\frac{u_t}{\sqrt{t}} = \frac{2(\pi - 2 \theta)}{\sqrt{\theta(\pi - \theta)}}.
\]
\end{corollary}
\begin{proof}
The corollary is a straightforward application of \cite[Theorem 1.2]{Wu2013} to the general setting, using the Proposition \ref{prop:angle} above.
\end{proof}

\subsection{Stochastic case}
\begin{lemma}
 Let $u_t = \sqrt{\kappa} B_t$ be the driving function of a normalized slit Löwner chain. Let $\tilde{u},$ $\lambda$, $\tilde{T}_{\max{}}$, $M$ and $W$ be defined as in Theorem \ref{th:toradial}. Then for $\tilde{t}\in [0, \tilde{T}_{\max{}}),$ $\utildenot$ satisfies the following SDE
\begin{dmath}
\label{eq:drivingrecalc}
 d\tilde{u}_{\tilde{t}} = 
\left(
\frac{W(\lambda^{-1}(\tilde{t}), e^{i\tilde{u}_{\tilde{t}}})}{i e^{i\tilde{u}_{\tilde{t}}}} \,\dot{\lambda}^{-1}(\tilde{t}) 
-\frac{\kappa}{2} \left(   
\frac{\left(M_{\lambda^{-1}(\tilde{t}) }\circ h^{-1}_{\sqrt{\kappa} \, B_{\lambda^{-1}(\tilde{t})}} \right)''(1)}{i e^{i \tilde{u}_{\tilde{t}}}}\, \dot{\lambda}^{-1}(\tilde{t}) + \sigma'(1) \sqrt{\dot{\lambda}^{-1}(\tilde{t})}+i\right)
\right)\, d\tilde{t}
+\sqrt{\kappa} \, d\tilde{B}_{\tilde{t}},
\end{dmath}
where $d\tilde{B}_{\tilde{t}} = \frac{1}{\sqrt{\dot{\lambda}^{-1}(\tilde{t})}}\, dB_{\lambda^{-1}(\tilde{t})}$. 

\end{lemma}
\begin{proof}
The expression
\begin{equation}
\label{eq:thito}
 \frac{1}{\sqrt{\dot{\lambda}^{-1} (\tilde{t})}}=\sqrt{\dot{\lambda}(\lambda^{-1}(\tilde{t}))} = \frac{\left(M_{\lambda^{-1}(\tilde{t})} \circ h^{-1}_{u_{\lambda^{-1}(\tilde{t})}}\right)'(1)}{\left(M_{\lambda^{-1}(\tilde{t})} \circ h^{-1}_{u_{\lambda^{-1}(\tilde{t})}}\right)(1)}= e^{-i \tilde{u}_{\tilde{t}}} \,\left(M_{\lambda^{-1}(\tilde{t})} \circ h^{-1}_{u_{\lambda^{-1}(\tilde{t})}}\right)'(1)
\end{equation}
may be formally regarded as a function of three independent variables $\tilde{t},$ $u_{\lambda^{-1}(\tilde{t})}$ and $\tilde{u}_{\tilde{t}}$. In particular, its partial derivatives satisfy
\[
 \frac{\sqrt{\dot{\lambda}(\lambda^{-1}(\tilde{t}))}}{\partial u_{\lambda^{-1}(\tilde{t})}} = i e^{-i \tilde{u}_{\tilde{t}}} \left(M_{\lambda^{-1}(\tilde{t})} \circ h^{-1}_{u_{\lambda^{-1}(\tilde{t})}}\right)''(1) - \frac{\sigma'(1)}{\sqrt{\dot{\lambda}^{-1}(\tilde{t})}} ,
\]
and
\[
 \frac{\partial \sqrt{\dot{ \lambda}(\lambda^{-1}(\tilde{t}))}}{\partial \tilde{u}_{\tilde{t}}}  = -\frac{i}{ \sqrt{\dot{\lambda}^{-1}(\tilde{t})}}.
\]

Let $u_t = \sqrt{\kappa} \, B_t$, $\sqrt{\kappa} \geq 0$. Then It\^{o}'s formula applied to \eqref{eq:thito} yields
\begin{dmath}
\label{eq:lambdaito}
 d \frac{1}{\sqrt{\dot{\lambda}^{-1}(\tilde{t})}} = \left(\ldots\right) \, d\tilde{t}-\frac{i}{ \sqrt{\dot{\lambda}^{-1}(\tilde{t})}} \circ d \tilde{u}_{\tilde{t}}  + \sqrt{\kappa}\,\left(i e^{-i \tilde{u}_{\tilde{t}}} \left(M_{\lambda^{-1}(\tilde{t})} \circ h^{-1}_{\sqrt{\kappa}\, B_{\lambda^{-1}(\tilde{t})}}\right)''(1) - \frac{\sigma'(1)}{\sqrt{\dot{\lambda}^{-1}(\tilde{t})}} \right)  \circ dB_{\lambda^{-1}(\tilde{t})}.
\end{dmath}

An application of It\^{o}'s formula to \eqref{eq:utildexpr} gives a stochastic version of \eqref{eq:lagrange}
\begin{dmath*}
 d\tilde{u}_{\tilde{t}} = \frac{W(\lambda^{-1}(\tilde{t}), e^{i\tilde{u}_{\tilde{t}}})}{i e^{i\tilde{u}_{\tilde{t}}}} \dot{\lambda}^{-1}(\tilde{t})\, d\tilde{t} + \frac{\sqrt{\kappa}}{\sqrt{\dot{\lambda}^{-1}(\tilde{t})}}\circ dB_{\lambda^{-1}(\tilde{t})} 
= \frac{W(\lambda^{-1}(\tilde{t}), e^{i\tilde{u}_{\tilde{t}}})}{i e^{i\tilde{u}_{\tilde{t}}}} \dot{\lambda}^{-1}(\tilde{t}) \,d\tilde{t} + \frac{\sqrt{\kappa}}{\sqrt{\dot{\lambda}^{-1}(\tilde{t})}}\, dB_{\lambda^{-1}(\tilde{t})}  
+ \frac{\sqrt{\kappa}}{2}\, d \frac{1}{\sqrt{\dot{\lambda}^{-1}(\tilde{t})}} \cdot dB_{\lambda^{-1}(\tilde{t})}.
\end{dmath*}

Using \eqref{eq:lambdaito}, we calculate the quadratic covariation
\begin{dmath*}
 \frac{\sqrt{\kappa}}{2}\,d \frac{1}{\sqrt{\dot{\lambda}^{-1}(\tilde{t})}} \cdot dB_{\lambda^{-1}(\tilde{t})} = -{\frac{\kappa}{2} \left(   
\frac{\left(M_{\lambda^{-1}(\tilde{t}) }\circ h^{-1}_{\sqrt{\kappa} \, B_{\lambda^{-1}(\tilde{t})}} \right)''(1)}{i e^{i \tilde{u}_{\tilde{t}}}}\, \dot{\lambda}^{-1}(\tilde{t}) + \sigma'(1) \sqrt{\dot{\lambda}^{-1}(\tilde{t})}+i\right)\, d\tilde{t}}, 
\end{dmath*}
and arrive at
\begin{dmath*}
 d\tilde{u}_{\tilde{t}} = 
\left(
\frac{W(\lambda^{-1}(\tilde{t}), e^{i\tilde{u}_{\tilde{t}}})}{i e^{i\tilde{u}_{\tilde{t}}}} \,\dot{\lambda}^{-1}(\tilde{t}) 
-\frac{\kappa}{2} \left(   
\frac{\left(M_{\lambda^{-1}(\tilde{t}) }\circ h^{-1}_{\sqrt{\kappa} \, B_{\lambda^{-1}(\tilde{t})}} \right)''(1)}{i e^{i \tilde{u}_{\tilde{t}}}}\, \dot{\lambda}^{-1}(\tilde{t}) + \sigma'(1) \sqrt{\dot{\lambda}^{-1}(\tilde{t})}+i\right)
\right)\, d\tilde{t}
+\sqrt{\kappa} \, d\tilde{B}_{\tilde{t}},
\end{dmath*}
where $d\tilde{B}_{\tilde{t}} = \frac{1}{\sqrt{\dot{\lambda}^{-1}(\tilde{t})}}\, dB_{\lambda^{-1}(\tilde{t})}$. Brownian scaling implies that $\tilde{B}_{\tilde{t}}$ is a standard Brownian motion.
\end{proof}

\begin{theorem}
\label{th:laweqiuivalence}
Let $\{K_t\}_{t\geq 0}$  be the family of random hulls generated by a normalized slit Löwner chain driven by $\sqrt{\kappa}\, B_t$. Let $\{\tilde{K}_{\tilde{t}}\}_{\tilde{t}\geq 0}$ be the family of radial $SLE_{\kappa}$-hulls. Let $\lambda$ and $T_{\max}$ be defined as in Theorem~\ref{th:toradial}. There exists a family of positive stopping times $\{T_n\}_{t\in \mathbb{N}}$, $T_n\to T_{\max}$,  such that the laws of $(K_t, t\in[0,T_n])$ and $\left(\tilde{K}_{\lambda(t)},t\in[0,T_n]\right)$ are absolutely continuous with respect to each other.
\end{theorem}

\begin{proof} For simplicity, we work in the unit disk. Let $a(\tilde{t}, \tilde{u}_{\tilde{t}})$ denote the coefficient at $d\tilde{t}$ in \eqref{eq:drivingrecalc}, and let $b(t)$ be the continuous process defined as $b(t) = a(\lambda(t), \tilde{u}_{\lambda(t)})$. Let $d_0(K) = \inf\{|z|:z\in K\}$ denote the distance from the set $K$ to the origin.

We define
\[
 T_n := \min\left\{n,\, \inf\{t: b(t) > n\}, \,\inf\left\{t:d_0(K_t) < \frac{1}{2n}\right\}\right\}, \quad  \tilde{T}_n =\lambda(T_n).
\]
Then we apply Girsanov's theorem to (\ref{eq:drivingrecalc}) and argue as in \cite[Proposition 4.2]{SLE2}.
\end{proof}

This observation has several important implications, in particular, most results about the local properties of classical $SLE$ hulls are also true for the hulls of general slit stochastic flows. In the following three corollaries we adapt results from \cite{RohdeSchrammBasic} and \cite{Beffara} to our setting.
\begin{corollary}
 Let $\{G_t\}_{t \geq 0}$ be a slit holomorphic stochastic flow in a simply connected hyperbolic domain $D$. Then the corresponding family of hulls $\{K_{t}\}_{t\geq 0}$ is generated by a curve with probability 1, i.e., there exists a curve $\gamma$ such that for each $t\geq 0,$ the evolution domain $D_t = D\setminus K_t$ is a connected component of $D\setminus \gamma[0,t]$. 
\end{corollary}
\begin{proof}
Let us denote the stopping time $T_1$ defined in the proof of Theorem \ref{th:laweqiuivalence} by $\tau_0$. Then the law of $\{K_t\}_{t\in [0, \tau_0]}$ is absolutely continuous with respect to the law of a stopped family of radial $SLE$ curves, in particular this family of hulls is generated by a curve almost surely.

The strong Markov property of $\{G_t\}_{t\geq 0}$ implies that the law of $G_{\tau_0 + t} \circ G^{-1}_{\tau_0}$ conditioned on $G_{\tau_0}$ is the same as the law of $\{G_t\}_{t\geq 0}$. Similarly, the law of the hulls $\tilde{K}_t = G_{\tau_0} \circ K_{\tau_0 +t}$ conditioned on $K_{\tau_0}$ is the same as the law of $\{K_t\}_{t\geq 0}$. In particular, there exists a stopping time $\tau_1$, having the same distribution as $\tau_0,$ such that the hulls $\{\tilde{K}_{t}\}_{t\in [0,\tau_1]}$ are generated by a curve almost surely. This implies that the original family of hulls $\{K_t\}_{t\geq 0}$ is almost surely generated by a curve not only on $[0,\tau_0]$, but also on $[0, \tau_0+\tau_1]$. 

Continuing by induction, we conclude that the hull is generated by a curve up to the time $\sum^{\infty}_{n=0} \tau_n$. A series of identically distributed positive random variables diverges with probability 1, and the corollary follows.
\end{proof}

The other two corollaries are proved similarly.

\begin{corollary}
\label{prop:phases}
 Let $\{G_t\}_{t\geq 0}$ be a normalized slit holomorphic stochastic flow driven by $b,$ $\sigma$ and $\sqrt{\kappa}B_t$. Let $\gamma$  denote the curve generating the hulls $\{K_{t}\}_{t\geq 0}$. Then, with probability 1,
\begin{itemize}
 \item if $0\leq \kappa \leq 4,$  $\gamma$ is a simple curve,
\item if $4 < \kappa <8,$   $\gamma$ has self-intersections,
 \item if $\kappa \geq 8,$  $\gamma$ is a space-filling curve.
\end{itemize}
\end{corollary}

\begin{corollary}
 The Hausdorff dimension of the curve generating the hulls  of a normalized slit Löwner chain driven by $\sqrt{\kappa}\, B_t$ is equal to $\min(2, 1+\kappa/8)$ with probability 1.
\end{corollary}

\subsection{General formulation of the domain Markov property for the case of simple curves} 

The fact that the hulls of a general stochastic flow driven by $\sqrt{\kappa} \,B_t$ for $0\leq \kappa \leq 4$  are simple curves, allows us to formulate versions of the conformal invariance property and the domain Markov property for this general process.

For simplicity, we work in the unit disk $\mathbb{D}$.

Let us call a curve $\gamma \subset \hat{\mathbb{D}}$ a $(b,\sigma)$-\emph{admissible} curve, if the following conditions are satisfied:
\begin{enumerate}
 \item $\gamma$ is a simple curve such that $\gamma(0)=1,$ and $\gamma\setminus \gamma(0) \subset \mathbb{D}$,
 \item $\gamma$ is embeddable into a slit Löwner chain driven by $b$ and $\sigma$, i.e., there exists a continuous function $u_t,$ $u_0=0,$ and $t_0\geq 0,$ such that for the corresponding slit L\"owner chain $g_t$ driven by $b,$ $\sigma$ and $u_t,$ $g^{-1}_{t_0}(\mathbb{D}) = \mathbb{D} \setminus \gamma$.
\end{enumerate}

By Corollary~\ref{cor:uniqueness}, the conformal isomorphism $g_{t_0}: \mathbb{D}\setminus \gamma \to \mathbb{D}$ is completely determined by $\gamma$, so that we can use the notation 
\[g_{\gamma}:=g_{t_0}\quad \textrm{and}\quad G_{\gamma} := G_{t_0} = h_{u_{t_0}} \circ  g_{t_0} .\]

Recall that $P(D)$ denotes the set of prime ends of $D$ (see Section \ref{subsec:confinvdmp}). The map $g_{\gamma}$ extended to the homeomorphism $g^{\gamma}:(\mathbb{D}\setminus \gamma)\cup P(\mathbb{D}\setminus \gamma)  \to  \hat{\mathbb{D}} $  maps the tip of the curve $\gamma$ to the point $h^{-1}_{u_{t_0}}(1)$. The function $G_{\gamma}$ maps the tip of $\gamma$ to 1.

Let
\[
 \Omega_{\mathbb{D} } := \{\gamma: \gamma \textrm{ is a curve in }\hat{\mathbb{D}}\textrm{ starting from }1\},
\]
and define for a $(b,\sigma)$-admissible curve $\gamma_0$,
\[
 \Omega_{\mathbb{D} \setminus \gamma_0} := \{\gamma: \gamma \textrm{ is a curve in }(\mathbb{D} \setminus \gamma_0) \cup P(\mathbb{D} \setminus \gamma_0)\textrm { starting from the tip of } \gamma_0\}.
\]

Let $\{K_t\}_{t\geq 0}$ be the random family of hulls generated by the slit Löwner equation driven by $b$, $\sigma$ and $\sqrt{\kappa}\,B_t$. The family $K_t$ induces a measure on the family of curves $\Omega_{\mathbb{D}},$ which we denote by $(b,\sigma)$-$SLE^{\kappa}_{\mathbb{D}}$  . Let $0\leq \kappa \leq 4$. By Corollary~\ref{prop:phases} we know that in this case the measure $(b,\sigma)$-$SLE^{\kappa}_{\mathbb{D}}$ is concentrated on the subset of $\Omega_{\mathbb{D}}$ consisting of simple curves.

For any admissible curve $\gamma_0$  we can define the measure $(b,\sigma)\mathhyphen SLE^{\kappa}_{\mathbb{D}\setminus \gamma_0}$ on the family $\Omega_{(\mathbb{D} \setminus \gamma_0)}$ by putting
\[
 (b,\sigma)\mathhyphen SLE^{\kappa}_{\mathbb{D}\setminus \gamma_0} = G^{-1}_{\gamma_0} {}_* \,(b,\sigma)\mathhyphen \, SLE^{\kappa}_{\mathbb{D}},
\]
i.e., as a pushforward measure.

Note that this measure coincides with the measure induced by the equation
\[
 \begin{cases}
  \frac{\partial}{\partial t} g_t(z) = - (h^{-1}_{\sqrt{\kappa} \, B_t *} b ) (g_t(z)), \quad t\geq 0,\\
  g_t(z) = G_{\gamma_0}(z), \quad z\in \mathbb{D}\setminus \gamma_0.
 \end{cases}
\]

Thus we have obtained a family of measures $\{(b,\sigma)\mathhyphen SLE^{\kappa}_{\mathbb{D}\setminus \gamma}\}$  indexed by the set of all $(b,\sigma)$-admissible curves $\gamma$. This family of measures is \emph{conformally invariant} meaning that for two different admissible curves $\gamma_1, \gamma_2$ 
\[
 (b,\sigma)\mathhyphen SLE^{\kappa}_{\mathbb{D}\setminus\gamma_2} = ({G^{-1}_{\gamma_2}\circ G_{\gamma_1}})_{*} (b,\sigma)\mathhyphen SLE^{\kappa}_{\mathbb{D}\setminus\gamma_1}.
\]

Due to the fact that the process $G_t$ is a time-homogeneous diffusion, the family of measures $\{(b,\sigma)\mathhyphen SLE^{\kappa}_{\mathbb{D}\setminus \gamma}\}$ possesses the \emph{domain Markov property}:
\[
 (b,\sigma)\mathhyphen SLE^{\kappa}_{\mathbb{D}\setminus\gamma_0} (\{\gamma: \gamma|_{[s,+\infty)} \in B \} |\gamma[0,s] = \gamma_1) = (b,\sigma)\mathhyphen SLE^{\kappa}_{\mathbb{D}\setminus(\gamma_0\cup \gamma_1)} (B),
\]
for any Borel set $B\subset \Omega_{\mathbb{D}\setminus(\gamma_0\cup \gamma_1)}$.

%

\appendix

\section{Holomorphic semiflows in the unit disk}
\label{app:semigroups}
Let $\mathbb{D}= \{z:|z|<1\}$  denote the unit disk, let $\mathrm{Hol}(\mathbb{D},\mathbb{D})$ denote the set of holomorphic maps $\phi:\mathbb{D} \to \mathbb{D}$ of $\mathbb{D}$ into itself, and let $\mathrm{Aut}(\mathbb{D}) \subset \mathrm{Hol}(\mathbb{D},\mathbb{D})$  be the set of holomorphic automorphisms of $\mathbb{D}$.

\begin{definition}
A family $\{g_t\}_{t\geq 0} \subset \mathrm{Hol}(\mathbb{D}, \mathbb{D})$ is called a \emph{holomorphic semiflow} in $\mathbb{D}$ (or, alternatively, a one-parameter continuous semigroup of holomorphic self-mappings of $\mathbb{D}$) if 
\begin{enumerate}
 \item $g_0 = \id_{\mathbb{D}},$
\item $g_{t+s} = g_t \circ g_s,$ for $s,t \geq 0$,
\item $\lim_{t\to 0^{+}} g_t(z) = z$ for each $z\in\mathbb{D}$.
\end{enumerate}
\end{definition}
These conditions imply that the family $g_t$ is continuous in local uniform topology  and, moreover, differentiable on all $[0,+\infty)$ with respect to $t$.

In the case when $g_t \in \mathrm{Aut}(\mathbb{D})$ for all $t\geq 0,$ the semiflow $\{g_t\}_{t\geq 0}$ can be extended to a flow $\{g_t\}_{t \in \mathbb{R}}$ by setting $g_{-t} := g_t^{-1}$.

For every holomorphic semiflow $\{g_t(z)\}_{t\geq 0}$ there exists a unique holomorphic function $V:\mathbb{D} \to\mathbb{C},$ such that the semiflow $\{g_t(z)\}_{t\geq 0}$ is the unique solution to the initial-value problem
\begin{equation}
\label{eq:semigroupivp}
 \begin{cases}
  \frac{\partial}{\partial t}g_t(z) = V(g_t(z)),\\
  g_0(z) = z.
 \end{cases}
\end{equation}
The function $V$ is called the \emph{infinitesimal generator} of the semiflow $\{g_t \}_{t\geq 0}$.

Infinitesimal generators of flows  and semiflows are often called \emph{complete} and  \emph{semicomplete} holomorphic vector fields, respectively, referring to the fact that the problem (\ref{eq:semigroupivp}) has a solution for all $t\in \mathbb{R}$ in the case of complete fields, and for all $t \geq 0$ in the case of semicomplete fields.

There is a simple representation formula for semicomplete vector fields. A holomorphic function $V(z)$ is a semicomplete vector field if and only if it can be written in the form
\begin{equation}
\label{eq:polynomsemicomplete}
 V(z) = V(0) - z \, q(z) - \overline{V(0)} z^2,
\end{equation}
where $q(z)$ is a holomorphic function with $\re q(z) \geq 0$ (see \cite{ShoikhetSeimgroupsGFT}). Moreover, $V(z)$ is complete if and only if $q(z) = i b,$ $b\in \mathbb{R}$.

The set of all semicomplete vector fields in $\mathbb{D}$ forms a closed (in the local uniform topology) real cone, that is, the vector field
\[
 V(z) = \alpha V_1(z) + \beta V_2(z), \quad \alpha, \beta \geq 0,\]
is semicomplete, provided that $V_1(z)$ and $V_2(z)$ are. Similarly, complete fields form a vector space of real dimension 3.

\subsection{Inverse flows}\label{app:inverseflows}
Let $\{h_{t}\}_{t\in \mathbb{R}}$ be the flow generated by $V(z)$. Then the inverse flow $\{h_{-t}\}_{t\in \mathbb{R}} = \{h^{-1}_{t}\}_{t\in \mathbb{R}}$ solves the ODE initial-value problem
\begin{equation}
  \begin{cases}
  \frac{\partial}{\partial t}h_{-t}(z) = -V(h_{-t}(z)),\\
  h_0(z) = z,
 \end{cases}z\in \mathbb{D},
\end{equation}
and the PDE initial-value problem
\begin{equation}
\label{eq:inverseflowpde}
 \begin{cases}
  \frac{\partial}{\partial t} h_{-t}(z) = -h'_{-t}(z)\, V(z),\\
  h_0(z) = z,
 \end{cases}z\in \mathbb{D}.
\end{equation}

\section{Elementary differential geometry in \texorpdfstring{$\mathbb{C}$}{C}}
In this section we write explicit formulas for some most basic operations of differential geometry for the case when the manifold in question is a simply connected domain $D$ in $\mathbb{C}$  and all functions in consideration are holomorphic.

Throughout the section we assume that $\{g_t\}_{t\geq 0}$ is a holomorphic semiflow in  $D,$ and $V(z)$ is the corresponding semicomplete vector field:
\[
 \begin{cases}
  \frac{\partial}{\partial t} g_t(z) = V(g_t(z)),\\
  g_0(z) = z, \quad z\in D.
 \end{cases}
\]

\subsection{Pushforwards of vector fields by conformal maps}\label{app:pushforward}
Let $V_1$ be a vector field in $D_1$ corresponding to the holomorphic semiflow $\{g^1_t\}_{t\geq 0}$  in $D_1$. Consider a conformal isomorphism $\phi: D_1 \to D_2$. We \emph{push forward} the semiflow $\{g^1_t\}_{t\geq 0}$ to $D_2$ by $\phi,$ if we define $\{g^2_t\}_{t\geq 0} = \{ \phi \circ g^1_t \circ \phi^{-1}\}_{t\geq 0}$. 

Denote by $V_2$ the vector field on $D_2$ corresponding to the flow $\{g^2_t\}_{t\geq 0}$. Then the vector fields $V_1$ and $V_2$ are said to be  \emph{$\phi$-related}, and  $V_2$ is also called the \emph{pushforward} (or, the \emph{direct image}) of $V_1$ by $\phi$. We use the notation $V_2 = \phi_* V_1$.

One of the ways to find an explicit expression for $V_2 = \phi_* V_1$ is to do the following calculation:
\begin{dmath*}
 \frac{\partial}{\partial t} g^2_t(z) = \frac{\partial}{\partial t} \phi(g^1_t(\phi^{-1}(z))) = \phi'(g^1_t(\phi(z))\cdot V_1(g^1_t(\phi^{-1}(z))) 
=\phi'(\phi^{-1} (g^2_t(z))) \cdot V_1(\phi^{-1}(g^2_t(z))).
\end{dmath*}
Thus, the explicit formula for the pushforward of a vector field $V$ by a conformal isomorphism $\phi$ is
\begin{equation}
\label{eq:vfpf}
(\phi_* V)(z) = \phi'(\phi^{-1}(z)) V (\phi^{-1}(z)) =\frac{1}{{\phi^{-1}}'(z)}V(\phi^{-1}(z)). 
\end{equation}
Note that the pushforward operation is linear in $V$.
\subsection{Pushforward of a flow by itself}
Let $\{h_t\}_{t\in \mathbb{R}}$ be the flow generated by $V(z)$, let $s\in \mathbb{R}$ and let $\{\tilde{h}\}_{t\in \mathbb{R}}$ be the flow generated by $({h_s}{}_* \, V)(z)$. By definition, 
\[
 \tilde{h}_{t} = h_s \circ h_t \circ h^{-1}_{s} = h_{s+t-s} = h_t,
\]
so that $\tilde{h}_t = h_t$ for all $t\in \mathbb{R},$ and ${h_s} {}_* \, V = V$.

\section{Stochastic flows in complex domains}
In this appendix we give a summary of basic results related to flows of stochastic differential equations in a domain $D\subset\mathbb{C}$. The definitions and theorems below are in fact adaptations from general theory of stochastic flows on paracompact manifolds, which can be found, e.g., in \cite[Chapter III]{Kunita1984}, \cite[Chapter V]{IkedaWatanabe} or, in full generality, in \cite[Section 4.8]{KunitaBook}.

\subsection{Solutions to SDEs}

Let $D$ be a domain in $\mathbb{C},$ and $\hat{D}= D \cup \{\partial\}$  be its one-point compactification.

Let $b(t,z)$, $\sigma_1(t,z), \ldots, $ $\sigma_{m}(t,z)$ be time-dependent vector fields on $D$. We assume that the time parameter $t$ varies in some interval $[0,a],$ $a>0$.

Consider the following Stratonovich stochastic differential equation
\begin{equation}
\label{eq:stratonovichSDE}
 dG_t = b(t,G_t) \,dt + \sum_{k=1}^m \sigma_k (t,G_t) \circ dB_t^k.
\end{equation}

Denote by $\{\mathcal{F}_t\}$ the filtration \[\mathcal{F}_t = \sigma(B^k_u - B^k_v: 0\leq u \leq v \leq t, \, k = 1, \ldots, m).\] 

Let $z\in D,$ and suppose there exists an $\mathcal{F}_t$-stopping time $T(z, \omega)$  and a stochastic process $G_t(z, \omega),$  $0 \leq t \leq \min(T(x, \omega),a)$  such that the following conditions are satisfied:
\begin{enumerate}
 \item $G_t(z)$ is a continuous $\mathcal{F}_t$-semimartingale;
 \item for all $t <\min(T(z, \omega),a)$, $G_t(z)$ satisfies 
\[
 G_t(z) = z + b(t,G_t(z)) dt + \sum_{k = 1}^m \int_0^t \sigma_k(s,G_s(z)) \circ dB^k_s;
\]
\item almost surely, $T(z,\omega)<\infty$ implies that $\lim_{t\uparrow T(z,\omega)}G_t(z) = \partial$.
\end{enumerate}
Then $G_t(z)$ is called a \emph{maximal solution} of (\ref{eq:stratonovichSDE}) with the initial condition $G_0 = z$. The stopping time $T(z,\omega)$ is  called the \emph{explosion time} (or \emph{escaping time}) of $G_t(z)$.

It is known that for existence and uniqueness of a maximal solution it is sufficient to require that the vector fields  $b(t,z)$, $\sigma_1(t,z), \ldots, $ $\sigma_{m}(t,z)$ are $C^1$ in $t$ and $C^2$ in $z=x+iy$ \cite[Theorem 8.3]{Kunita1984}. 

\subsection{Stochastic flows}
The solution $G_t(z),$ considered as a function of the initial condition $z,$ is called the \emph{flow} of SDE (\ref{eq:stratonovichSDE}). 

Denote by $D_t$ the domain of $G_t(z)$ (i.e., $D_t(\omega):=\{z:T(z,\omega) >t\}$) and by $R_t$ its range. Then $G_t:D_t\to R_t$ is a homeomorphism \cite[Theorem 9.1]{Kunita1984}. Moreover, if we additionally require that the coefficients $b(t,z)$, $\sigma_1(t,z), \ldots, $ $\sigma_{m}(t,z)$ are $C^{k+1}$ in $z=x+i y,$ and $k+1$-st derivatives are $\alpha$-H\"older continuous for some $\alpha>0,$ then $G_t$ becomes a $C^{k}$-diffeomorphism, with $k$-th derivatives being $\beta$-H\"older continuous for any $0<\beta<\alpha$.

The most important result for us is contained the following theorem.

\begin{theorem}[{\cite[Theorem 5.7.]{Kunita1984}}]
The functions $G_t(z):D_t \to R_t$ are holomorphic in $z$ for each $t\in[0,a]$ if and only if the vector fields  $b(t,z)$, $\sigma_1(t,z), \ldots, $ $\sigma_{m}(t,z)$ are holomorphic in $z$.
\end{theorem}

In this paper we only work  with holomorphic $b(t,z)$, $\sigma_1(t,z), \ldots, $ $\sigma_{m}(t,z)$, and hence most formulas look just as neat as in the one-dimensional real case. For instance, the Stratonovich SDE
\[ dG_t = b(G_t) dt + \sum_{k=1}^m \sigma_k (G_t) \circ dB_t^k
\]
can be easily rewritten in the It\^o form as
\[
 dG_t = \left[b(t,G_t)  + \frac12 \sum_{k=1}^m \sigma_k (t,G_t) \sigma'_k(t,G_t) \right]\, dt+ \sum_{k=1}^m \sigma_k(t,G_t)\, dB_t^k,
\]
and vice versa (as always, $\sigma'$ denotes $z$-derivative). Nevertheless, Stratonovich equations are preferred, due to the fact that they obey usual calculus rules.

\subsection{Composition and inversion of flows}
\label{subs:compinvflow}
 Consider two flows in the same domain $D$: 

\begin{equation}
\label{eq:exampleflow1}
 \begin{cases}
  dG_t(z) = b(t, G_t(z)) \,dt + \sum_{k=1}^{m} \sigma_k (t, G_t(z)) \circ dB^k_t,\\
 G_0(z) = z,
 \end{cases} z\in D,
\end{equation}
and
\begin{equation}
\label{eq:exampleflow2}
 \begin{cases}
  d\tilde{G}_t(z) = \tilde{b}(t, \tilde{G}_t(z)) \,dt + \sum_{k=1}^{m} \tilde{\sigma}_k (t, \tilde{G}_t(z)) \circ dB^k_t,\\
 \tilde{G}_0(z) = z,
 \end{cases} z\in D,
\end{equation}
with explosion times $T(z,\omega)$ ans $\tilde{T}(z,\omega),$ respectively. 

Let $D_t:= \{z \in D: T(z,\omega)>t\},$ and $U(z,\omega):= \min[\inf\{t>0 : \tilde{G}_t(z) \not\in D_t\}, \tilde{T}(z,\omega)]$. Then the composite flow
$K_t(z):=G_t \circ \tilde{G}_t(z)$ is well-defined for $t < U(z,\omega),$ and satisfies 
\begin{equation}
\label{eq:compflow}
 dK_t = \left[ b(t, K_t) + {G_t}_{\,*} \tilde{b}(t, K_t)\right]\, dt + \sum_{k=1}^m\left[\sigma_k(t,K_t) + {G_t}_{\,*}\,\tilde{\sigma}(t, K_t) \right] \circ dB^k_t.
\end{equation}

The inverse flow $G^{-1}_t$ satisfies
\begin{dmath}
\label{eq:invflow}
 dG^{-1}_t(z) = -{G^{-1}_t(z)}_{*} b(t, G_t^{-1}(z))\, dt - \sum_{k=0}^m {G^{-1}_t(z)}_{*} \sigma_k(t, G_t^{-1}(z)) \circ dB^k_t \\
 = -{G^{-1}_t}'(z) \,b(t, z)\, dt - \sum_{k=0}^m {G^{-1}_t}'(z) \,\sigma_k(t, z) \circ dB^k_t.
\end{dmath}

The proofs of these formulas can be found, e.g.,  in \cite[pp.~268--270]{Kunita1981}, but they can also be formally derived by application of usual calculus rules to Stratonovich SDEs (\ref{eq:exampleflow1})-(\ref{eq:exampleflow2}). 

\subsection{Complete fields}
\label{app:completefields}
The following is a classical result saying that if the coefficients of a Stratonovich SDE are complete vector fields, then the SDE generates a stochastic flow of diffeomorphisms.
\begin{theorem}[{\cite[Theorem 5.1.]{Kunita1984}}]
\label{th:Kunita}
 Consider the time-homogeneous Stratonovich SDE
\[
 dH_t = b(H_t)\, dt  + \sum_{k=1}^m \sigma_k(H_t)\circ dB^k_t,
\]
with $b, \sigma_1, \ldots \sigma_k$ complete $C^{\infty}$ vector fields in a domain $D$. Suppose the Lie algebra $\mathfrak{g}$ generated by $b, \sigma_1, \ldots \sigma_k$ is of finite dimension and let $G$ be the associated Lie group. Then the flow $H_t$ takes values in $G$.
\end{theorem}

\section{Explicit expression for \texorpdfstring{$T^\sigma$}{T^sigma}}
\label{app:tsigma}
 Let $\sigma(z)$ be a complete field, 
\[
 \sigma(z) = \alpha - i\beta z -\overline{\alpha} z^2,
\]
and denote by $\{h_t(z)\}_{t\in \mathbb{R}}$ the corresponding flow of automorphisms of the unit disk, extended to the closed unit disk $\overline{\mathbb{D}}$.

We look for values of $t,$ for which the flow  $\{h_t\}$ sends $1$ to some given point $e^{i\theta}\in \partial \mathbb{D},$ i.e.,
\begin{equation}
\label{eq:trequirement}
 h_t(1) = e^{i\theta}.
\end{equation}

Suppose $\sigma(e^{i\theta}) = 0$. In this case, $e^{i\theta}$ is a common fixed point of all maps $\{h_t(z)\}_{t\in \mathbb{R}}$, and there are no values of $t$ satisfying $h_t(1) = e^{i\theta}$  (unless $e^{i\theta}$ coincides with $1$, in which case the equality holds for all $t\in \mathbb{R}$). 

Otherwise, there is always at least one such $t,$ for which (\ref{eq:trequirement}) holds. 

Denote
\[
 D = -\beta^2 + 4 |\alpha|^2;
\]

The flow $\{h_t(z)\}_{t\in \mathbb{R}}$ is hyperbolic if $D=0,$ and then either $\beta = 2 |\alpha|$ or $\beta = - 2|\alpha|$. In the first case, $V(z) = - \overline{\alpha} \left(z + i \frac{|\alpha|}{\overline{\alpha}}\right)^2,$ and
\[
 t = \frac{\tan \frac{\theta}{2}}{|\alpha|-\im \alpha + \re \alpha \cdot \tan \frac{\theta}{2}}.
\]

In the second case, $V(z) = - \overline{\alpha} \left(z - i \frac{|\alpha|}{\overline{\alpha}}\right)^2,$ and
\[
 t = \frac{\tan \frac{\theta}{2}}{|\alpha|+\im \alpha - \re \alpha \cdot \tan \frac{\theta}{2}}.
\]

If $D>0$ then the flow is hyperbolic, and 
\[
 t = -\frac{2}{\sqrt{D}} \, \arctanh \frac{\sqrt{D} \,\re \left[(1- e^{-i\theta}) (2\alpha - i \beta)\right]}{|2\alpha-i\beta|^2 - \re \left[e^{-i\theta} (2\alpha -i\beta)^2\right]}.
\]

Finally, if  $D<0,$ the flow is elliptic (in particular, periodic), and there are infinitely many values of $t$ satisfying (\ref{eq:trequirement}):
\[
 t = -\frac{2}{\sqrt{-D}} \,\arctan \frac{\sqrt{-D}\, \re \left[(1-e^{-i\theta}) (2 \alpha - i\beta))\right]}{|2\alpha - i\beta|^2 - \re \left[e^{-i\theta} (2\alpha - i\beta)^2\right]} +\frac{2 \pi k}{\sqrt{-D}}, \quad k \in \mathbb{Z}.
\]

This motivates us to define a function
\[
 T^\sigma :  \{e^{i\theta}:\sigma(e^{i\theta})\neq 0, \theta \in \mathbb{R}\} \to \mathbb{R}
\]
as follows

\begin{equation}
\label{eq:tdefinition}
T^{\sigma}(e^{i\theta}) =
 \begin{cases}  \frac{\tan \frac{\theta}{2}}{|\alpha|-\im \alpha + \re \alpha \cdot \tan \frac{\theta}{2}}, \quad \textrm{if }\beta = 2 |\alpha|,\\
\frac{\tan \frac{\theta}{2}}{|\alpha|+\im \alpha - \re \alpha \cdot \tan \frac{\theta}{2}}, \quad \textrm{if }\beta = -2 |\alpha|, \\
-\frac{2}{\sqrt{D}} \, \arctanh \frac{\sqrt{D} \,\re \left[(1- e^{-i\theta}) (2\alpha - i \beta)\right]}{|2\alpha-i\beta|^2 - \re \left[e^{-i\theta} (2\alpha -i\beta)^2\right]}, \quad \textrm{if }D>0,\\
 -\frac{2}{\sqrt{-D}} \,\arctan \frac{\sqrt{-D}\, \re \left[(1-e^{-i\theta}) (2 \alpha - i\beta))\right]}{|2\alpha - i\beta|^2 - \re \left[e^{-i\theta} (2\alpha - i\beta)^2\right]},\quad \textrm{if } D <0,
 \end{cases}
\end{equation}
where $\sigma(z) =  \alpha - i\beta z -\overline{\alpha} z^2,$ and $D =-\beta^2 + 4 |\alpha|^2$.

The function is defined in such a way that
\[
 h_{T^{\sigma}(e^{i\theta})}(1) =e^{i\theta}.
\]
Since $h^{-1}_t = h_{-t},$ the following formula is also true
\[
 h^{-1}_{-T^{\sigma}(e^{i\theta})}(1) = e^{i\theta}.
\]


\bibliographystyle{plain}
\bibliography{sle}
\end{document}